\documentclass{amsart}
\usepackage{amssymb,amsmath,amsthm}
\usepackage[inline,shortlabels]{enumitem}
\usepackage{eucal}
\usepackage{mathabx}
\usepackage{xcolor}
\usepackage{hyperref}
\usepackage{mathtools}

\hypersetup{
    colorlinks,
    linkcolor={red!75!black},
    citecolor={green!65!black},
    urlcolor={magenta!75!black}
}

\usepackage[backend=biber,sorting=nyt,style=alphabetic,minalphanames=3,maxbibnames=99]{biblatex}

\DeclareNameAlias{author}{family-given}
\theoremstyle{definition}
\newtheorem{Theorem}{Theorem}[section]
\newtheorem{Proposition}[Theorem]{Proposition}
\newtheorem{Lemma}[Theorem]{Lemma}
\newtheorem{Corollary}[Theorem]{Corollary}

\newtheorem*{Claim*}{Claim}

\newenvironment{proofC}{
    \begin{proof}[Proof of Claim.]
    \let\oldqed=\qedsymbol
    \renewcommand{\qedsymbol}{\oldqed$_{\mathrm{Claim}}$}
}{
    \end{proof}
}
 
\theoremstyle{definition}
\newtheorem{Definition}[Theorem]{Definition}
\newtheorem{Fact}[Theorem]{Fact}

\newtheorem{Example}[Theorem]{Example}
\newtheorem{Remark}[Theorem]{Remark}
\newtheorem{Observation}[Theorem]{Observation}
\newtheorem{Question}{Question}

\newtheorem{Theorem1}{Theorem}

\usepackage{etoolbox}\usepackage{orcidlink}

\def\forces{\vdash}
\def\strok{\restriction}

\def\Ind#1#2{#1\setbox0=\hbox{$#1x$}\kern\wd0\hbox to 0pt{\hss$#1\mid$\hss}
\lower.9\ht0\hbox to 0pt{\hss$#1\smile$\hss}\kern\wd0}

\def\ind{\mathop{\mathpalette\Ind{}}}

\def\notind#1#2{#1\setbox0=\hbox{$#1x$}\kern\wd0
\hbox to 0pt{\mathchardef\nn=12854\hss$#1\nn$\kern1.4\wd0\hss}
\hbox to 0pt{\hss$#1\mid$\hss}\lower.9\ht0 \hbox to 0pt{\hss$#1\smile$\hss}\kern\wd0}

\def\dep{\mathop{\mathpalette\notind{}}}

\DeclareMathOperator{\wor}{\perp\hspace*{-0.4em}^{\mathit w}}
\DeclareMathOperator{\nwor}{\not\perp\hspace*{-0.4em}^{\mathit w}}
\DeclareMathOperator{\fwor}{\perp\hspace*{-0.4em}^{\mathit f}}
\DeclareMathOperator{\nfor}{\not\perp\hspace*{-0.4em}^{\mathit f}}
\DeclareMathOperator{\Mon}{\mathfrak C}
\DeclareMathOperator{\tp}{tp}
\DeclareMathOperator{\itp}{itp}
\DeclareMathOperator{\dcl}{dcl}
\DeclareMathOperator{\conv}{conv}

\DeclareMathOperator{\Th}{Th}
\DeclareMathOperator{\Aut}{Aut}

\DeclareMathOperator{\Lev}{Lev}
\DeclareMathOperator{\dom}{dom}
\DeclareMathOperator{\supin}{in}
\DeclareMathOperator{\fin}{fin}

\def\mm{\mathbf{m}}

\usepackage[margin=30mm]{geometry}

\title{Countable models of weakly quasi-o-minimal theories I}  
\author[S.\ Moconja]{Slavko Moconja\ \orcidlink{0000-0003-4095-8830}}
\address[S.\ Moconja]{University of Belgrade, Faculty of mathematics, Belgrade, Serbia}
\email{slavko@matf.bg.ac.rs}

\author[P.\ Tanovi\'c]{Predrag Tanovi\'c\ \orcidlink{0000-0003-0307-7508}}
\address[P.\ Tanovi\'c]{Mathematical Institute of the Serbian Academy of Sciences and Arts, Belgrade, Serbia}
\email{tane@mi.sanu.ac.rs}
\thanks{The authors are supported by the Science Fund of the Republic of Serbia,
grant 7750027--SMART}

\begin{document} 

\begin{abstract} 
We introduce the notions of triviality and order triviality for global invariant types in an arbitrary first-order theory and show that they behave well in the NIP context. We show that these two notions agree for invariant global extensions of a weakly o-minimal type, in which case we say that the type is trivial. In the o-minimal case, we prove that every definable complete 1-type over a model is trivial. We prove that the triviality has several favorable properties; in particular, it is preserved in nonforking extensions of a weakly o-minimal type and under weak nonorthogonality of weakly o-minimal types. We introduce the notion of a shift in a linearly ordered structure that generalizes the successor function in infinite discrete orders. Then we apply the techniques developed to prove that every weakly quasi-o-minimal theory that admits a definable shift has $2^{\aleph_0}$ countable models.
\end{abstract}

 \maketitle
 
\section{Introduction}

In this paper, we continue the work on the classification theory for countable models of weakly quasi-o-minimal theories; recall that a complete first-order theory is weakly quasi-o-minimal if for some (equivalently, all by \cite[Theorem 1(ii)]{MTwmon}) $0$-definable linear order $<$ every definable\footnote{Throughout the paper, definable means definable with parameters.}
subset of $\Mon$ is a Boolean combination of unary $0$-definable sets and convex sets.
By classification, we mean, first, finding the relevant properties of $T$ that imply $I(T,\aleph_0)=2^{\aleph_0}$, and then, assuming that all these properties fail for $T$, finding a reasonable system of invariants for the isomorphism types of countable models of $T$. By reasonable invariants, objects that can be easily counted are usually meant, so that Vaught's conjecture for $T$ can be confirmed.
In \cite{MT}, we completed the classification for binary weakly quasi-o-minimal $T$; recall that $T$ is binary if every formula is equivalent to a Boolean combination of formulae in two free variables: Theorem 1 of
\cite{MT} states that each of the following five conditions implies $I(T,\aleph_0)=2^{\aleph_0}$:
\begin{enumerate}[label=(C\arabic*),left=0em,labelsep=.7em]
    \item\label{C1} $T$ is not small;
    \item\label{C2} There is a non-convex type $p\in S_1(T)$;
    \item\label{C3} There is a non-simple type $p\in S_1(T)$;
    \item\label{C4} There is a nonisolated forking extension of some $p\in S_1(T)$  over a 1-element domain; 
    \item\label{C5} There are infinitely many weak non-orthogonality classes of nonisolated types in $S_1(T)$.
\end{enumerate}
Then, in \cite[Theorem 2]{MT}, assuming that none of \ref{C1}-\ref{C4} holds, the invariants of countable models were described as certain sequences of dense order-types. In \cite[Theorem 8.10]{MT}, assuming in addition that \ref{C5} doesn't hold, the exact value of $I(T,\aleph_0)$ was calculated.       

Outside of the binary context, the above list is inadequate. In Section \ref{Section_Examples} we give examples of weakly o-minimal theories with three countable models that satisfy conditions \ref{C3} and \ref{C4}. As for the other conditions, \ref{C1} is evidently valid for general $T$, while in the article that continues this, \cite{MTwqom2}, we will prove the validity of \ref{C2} and conjecture a strengthening of \ref{C5}.   

The principal theorem of this paper introduces a novel significant condition.

\begin{Theorem1}\label{Theorem1_shift}
If a weakly quasi-o-minimal theory has a definable shift on $(\Mon,<)$, then $I(T,\aleph_0)=2^{\aleph_0}$.
\end{Theorem1}
The nonexistence of shifts, together with the failure of \ref{C1} and \ref{C2}, will suffice to confirm Vaught's and Martin's conjecture for a wide subclass of weakly quasi-o-minimal theories that includes all quasi-o-minimal and all almost $\aleph_0$-categorical theories (\cite{MTwqom2}); however, for the whole class, it is unlikely that the three conditions suffice.  

The motivation for Theorem \ref{Theorem1_shift} we found in the work of Alibek, Baizhanov, and Zambarnaya \cite{Alibek}, where they introduce the notion of quasi-successors as a generalization of the successor function in infinite discrete orders. Our notion of shifts, defined in Section \ref{Section_shifts}, is a modification of this notion.     
Since it is known that the existence of an infinite definable discrete order in a model of a countable complete first-order theory $T$ implies $I(T,\aleph_0)=2^{\aleph_0}$, see \cite{Tane}, it is interesting to know whether Theorem \ref{Theorem1_shift} holds outside the weakly quasi-o-minimal context, even in the following special case.     

\begin{Question}
If $<$ is a definable linear order on $\Mon$ and $f:\Mon\to\Mon$ is a definable increasing function such that $x<f(x)$ holds for all $x\in \Mon$, must $T$ have $2^{\aleph_0}$ countable models? 
\end{Question}

Throughout the paper, first-order model theory techniques will be used exclusively. We will rely on results from our recent paper \cite{MTwom}, in which the notion of weak o-minimality for complete types of an arbitrary first-order theory is introduced and several favorable model-theoretic properties of weakly o-minimal types are proved; these are reviewed in Section \ref{Section_preliminaries}. 
The key new notion in this work is that of trivial weakly o-minimal types. They play a central role in the proof of Theorem \ref{Theorem1_shift}, making the proof not generalize outside the weakly quasi-o-minimal context. We believe also that trivial types will play a decisive role in establishing the proof of Vaught's conjecture for weakly quasi-o-minimal theories.   

\smallskip 
In Section \ref{Section_preliminaries}, we give a thorough overview of the results from \cite{MTwom} related to the weakly o-minimal types that will be used. In Section \ref{Section_IT convex}, we prove some properties of convex weakly o-minimal types ($p\in S(A)$ is convex if its locus is a convex subset of some $A$-definable linear order $(D,<)$). 

In Section \ref{Section_trivial}, motivated by the notion of triviality for types in stable theories that was introduced and studied by Baldwin and Harrington in \cite{BH}, 
we introduce the notions of triviality and order-triviality for invariant types in an arbitrary theory. The idea behind the triviality of the type in the stable case is that every pairwise independent set of realizations is independent (as a set). Instead of sets of realizations, we consider sequences of realizations of $\mathfrak p_{\restriction A}$, where $\mathfrak p$ is an $A$-invariant global type, and consider them independent if they are Morley in $\mathfrak p$ over $A$; thus $\mathfrak p$ is trivial over $A$ if $(a_i,a_j)\models \mathfrak p^2_{\restriction A}$ (for all $i<j<\omega$) implies that $(a_i\mid i\in\omega)$ is Morley in $\mathfrak p$ over $A$.
For order-triviality, we require that $(a_i,a_{i+1})\models \mathfrak p^2_{\restriction A}$ (for all $i\in\omega$) implies that $(a_i\mid i\in\omega)$ is Morley; this cannot occur in the stable case. In the NIP case, we show that the triviality of $\mathfrak p$ over $A$ transfers to any $B\supseteq A$. We also prove some nice properties of weak orthogonality ($\wor$) for trivial and order-trivial types. 

In Section \ref{Section_trivial_wom} we introduce and study triviality in the context of weakly o-minimal types. 
If $\mathfrak p$ is $A$-invariant and $p=\mathfrak p_{\restriction A}$ is weakly o-minimal, we show that the triviality and order-triviality of $\mathfrak p$ agree; we also prove that the triviality of one $A$-invariant extension of $p$ implies triviality of the other; in that case, we say that $p$ is trivial. Surprisingly, trivial types are common in o-minimal theories; we show that if $M$ is an o-minimal structure, then all definable types in $S_1(M)$ are trivial. 
We prove in Theorem \ref{Theorem_nwor preserves triviality}: the triviality of a weakly o-minimal type $p$ transfers to its nonforking extensions, the triviality is preserved under $\nwor$ of weakly o-minimal types, and that $\wor$ of trivial types transfers to their nonforking extensions. We also prove that every trivial type over a finite domain in a theory with few countable models is both convex and simple. These results are used in Section \ref{Section_shifts}, in which we introduce shifts and prove Theorem \ref{Theorem1_shift}. Section \ref{Section_Examples} contains examples.

\section{Preliminaries}\label{Section_preliminaries} 

Throughout the paper, we use standard model-theoretical terminology and notation.  We work in $\Mon$, a large, saturated (monster) model of a complete, first-order (possibly multi-sorted) theory $T$ in a language $L$. 
The singletons and tuples of elements from $\Mon$ are denoted by $a,b,c,\dots$. The letters $A,B,A', B',\dots$ are reserved for small subsets (of cardinality $<|\Mon|$) of the monster, while $C,D,C',D',\dots$ are used to denote arbitrary sets of tuples. By a formula $\phi(x)$, we mean one with parameters. For a formula $\phi(x)$, we denote $\phi(C)= \{c\in C\mid \models \phi(c)\}$,  and similarly for (partial) types. 
A set $D$ is definable ($A$-definable) if $D=\phi(\Mon)$ for some formula $\phi(x)$ ($\phi(x)\in L_A$); similarly for type-definable sets. $S_x(C)$ denotes the set of all complete types over $C$ in variables $x$; if $T$ is one-sorted (e.g.\ o-minimal), then 
$S_n(C)$ denotes the set of all complete $n$-types over $C$ and $S(C)=\bigcup_{n\in\mathbb N}S_n(C)$. 
The following notation will be used often:
\begin{enumerate}[label=$\bullet$,left=1em,labelsep=.7em]
    \item $S_{\pi}(B)=\{p(x)\in S_x(B)\mid \pi(x)\subseteq p(x)\}$, where $\pi(x)$ is a partial type over $A$ and $B\supseteq A$.
\end{enumerate}
If $A=B$, then we also write $S_{\pi}$ instead of $S_{\pi}(A)$.

Types from $S(\Mon)$ are called global types.
A global type $\mathfrak p(x)$ is {\em $A$-invariant} if $(\phi(x,a_1)\leftrightarrow \phi(x,a_2))\in\mathfrak p$ for all $L_A$-formulae $\phi(x,y)$ and all tuples $a_1,a_2$ with $a_1\equiv a_2\,(A)$; $\mathfrak p$ is {\em invariant} if it is $A$-invariant for some $A$. Let $\mathfrak p$ be $A$-invariant and let $(I,<)$ be a linear order; we say that the sequence of tuples $(a_i\mid i\in I)$ is a {\em Morley sequence} in $\mathfrak p$ over $A$ if $a_i\models \mathfrak p_{\restriction A\,a_{<i}}$ holds for all $i\in I$. Note that we allow Morley sequences to have arbitrary (even finite) order-type. 

Dividing and forking have the usual meaning. Complete types over $A$, $p(x)$ and $q(y)$, are weakly orthogonal, denoted by $p\wor q$, if $p(x)\cup q(y)$ determines a complete type; $p$ is forking orthogonal to $q$, denoted by $p\fwor q$, if $\tp(a/bA)$ does not fork over $A$ for all $a\models p$ and $b\models q$. 

A partial type $p(x)$ over $A$ is {\em NIP} if there is no formula $\varphi(x,y)$, an $A$-indiscernible sequence $(a_n\mid n<\omega)$ in $p$, and a tuple $b$, such that $\models\varphi(a_n,b)$ iff $n$ is even. A partial type $p(x)$ over $A$ is {\em dp-minimal} if given any realization $a$ of $p$ and any two mutually $A$-indiscernible sequences, at least one of them is indiscernible over $Aa$; dp-minimal types are NIP.

The notation related to linear orders is mainly standard. Let $(D,<)$ be a linear order; $<$ is defined on the power set of $D$ by $X<Y$ iff $x<y$ for all $x\in X$ and $y\in Y$.
$X\subseteq D$ is convex if $a<c<b$ and $a,b\in X$ imply $c\in X$; 
$a\in D$ is an upper (lower) bound of $X$ if $X<a$ ($a<X$) is valid; $X$ is an initial part of $D$ if $a\in X$ and $b<a$ imply $b\in X$; similarly, the final parts are defined. We emphasize the following notation:
\begin{enumerate}[label=$\bullet$,left=1em,labelsep=.7em]
\item $\supin_{(D,<)}(X):=\{d\in D\mid (\exists x\in X)\ d\leqslant x\}$ is the smallest initial part of $D$ that contains $X$. 
\end{enumerate}
When the meaning of $<$ or $D$ is clear from the context, we omit it in the subscript and write $\supin_D(X)$ or $\supin (X)$. The smallest final part that contains $X$ is denoted by $\fin X$. Note that two convex subsets, $X_1$ and $X_2$, are equal if and only if $\supin (X_1)=\supin (X_2)$ and $\fin (X_1)=\fin (X_2)$.

\subsection{Relative definability} 
In this subsection, we recall some facts and conventions from \cite{MTwom}. 
Let $p(x)$ be a partial type over $A$. A subset $X\subseteq  p(\Mon)$ is {\em relatively $B$-definable within $p(\Mon)$} if $X=\phi(\Mon)\cap p(\Mon)$ for some $L_B$-formula $\phi(x)$ (usually we will have $A\subseteq B$); here, $\phi(x)$ is called a {\em relative definition} of $X$ within $p(\Mon)$, and we will also say that $X$ is {\em relatively defined by $\phi(x)$ within $p(\Mon)$}.

The family of all relatively $B$-definable subsets of a type-definable set is closed for finite Boolean combinations. Also,
if $X_i$ is relatively $B$-definable within $P_i$ ($1\leqslant i\leqslant n$), then $\prod_{1\leqslant i\leqslant n}X_i$ is relatively $B$-definable within $\prod_{1\leqslant i\leqslant n}P_i$.
In particular, the relative definability of  (finitary) relations within $P$ is well defined. We will say that a structure $(P,R_i)_{i\in I}$ is relatively definable if each relation $R_i$ is such within $P$. For example, if a relation $R\subseteq P^2$ is relatively defined by a formula $\phi(x,y)$ and if $(P,R)$ is a linear order, then we say that $\phi(x,y)$ relatively defines a linear order on $P$. Also, if $P$ and $Q$  are type-definable sets, then the relative definability of (graphs of) functions $f:P\to Q$ is well defined.

The universal properties of relatively definable structures are expressed by certain $L_{\infty,\omega}$-sentences, which we  call {\em $\tp$-universal sentences} and informally denote by $(\forall x_1\models p_1)\dots(\forall x_n\models p_n)\ \psi(x_1,\dots,x_n)$, where $p_1(x_1),\dots,p_n(x_n)$ are partial types and $\psi(x_1,\dots,x_n)$ is a first-order formula; formally, $(\forall x\models p)\ \phi(x)$ is $(\forall x)(\bigwedge_{\theta\in p}\theta(x)\rightarrow \phi(x))$. 
The properties of relatively definable relations (and their defining formulae) expressed by these sentences are called {\em $\tp$-universal properties}. For example, ``the formula $x\leq y$ relatively defines a preorder on $p(\Mon)$'' and ``$\leq$ is a relatively definable preorder on $p(\Mon)$'' are $\tp$-universal properties.
The following is a version of the compactness that will be applied when dealing with $\tp$-universal properties. 

\begin{Fact}\label{Fact_L_P_sentence}
Suppose that $p_1(x_1),\dots,p_n(x_n)$ are partial types that are closed for finite conjunctions and let $\phi(x_1,\dots,x_n)$ be an $L_{\Mon}$-formula such that $\Mon\models (\forall x_1\models p_1)\dots(\forall x_n\models p_n)\ \phi(x_1,\dots,x_n)$. Then there are   $\theta_i(x_i)\in p_i$ ($1\leqslant i\leqslant n$) such that: 
$$\Mon\models (\forall x_1\dots  x_n)\left(\bigwedge_{1\leqslant i\leqslant n}\theta_i'(x_i)\rightarrow \phi(x_1,\dots,x_n)\right)$$
for all formulae $\theta_i'(x_i)$ such that $\theta_i'(\Mon)\subseteq \theta_i(\Mon)$ ($1\leqslant i\leqslant n$).  
\end{Fact}

\noindent {\bf Convention.} Whenever $<$ is a relatively definable order on $p(\Mon)$, we assume the formula $x<y$ is its relative definition; similarly, $E(x,y)$ is a relative definition of an relatively definable equivalence relation $E$.

\begin{enumerate}[label=$\bullet$,left=1em,labelsep=.7em]
    \item Let $<$ be a relatively definable linear order on $p(\Mon)$. This is a tp-universal property, so by compactness, there is a formula $\theta(x)\in p$ such that the formula $x<y$ relatively defines a linear order on $\theta(\Mon)$, that is, $(\theta(\Mon),<)$ is a linear order. Orders of this form are called {\it definable extensions} of $(p(\Mon),<)$.
    
    \item Consider the structure $\mathcal P=(p(\Mon),<,E)$ where $<$ is a relatively definable linear order and $E$ a relatively definable convex equivalence relation on $p(\Mon)$. It is not hard to see that a finite conjunction of tp-universal properties is also tp-universal, so by compactness, there is a formula $\theta(x)\in p$ such that $(\theta(\Mon),<,E)$ is a definable extension of $\mathcal P$: $x<y$ defines a linear order and $E(x,y)$ a convex equivalence relation on $\theta(\Mon)$. 
\end{enumerate}

\subsection{Weakly o-minimal types}

In this subsection, we survey basic properties of weakly o-minimal types. A complete type $p\in S(A)$ is {\it weakly o-minimal} if there is a relatively $A$-definable linear order $<$ on $p(\Mon)$ such that each relatively definable subset of $(p(\Mon),<)$ consists of a finite number of convex components; in that case, $(p,<)$ is called a {\it weakly o-minimal pair over $A$}. The weak o-minimality of a type does not depend on the particular choice of the order: By \cite[Corollary 4.3]{MTwom}, if $(p,<)$ is a weakly o-minimal pair over $A$, then so is $(p,<')$ for each relatively $A$-definable linear order $<'$ on $p(\Mon)$. 

We introduced so-types\footnote{Here, {\em so} stands for {\em stationarily ordered}.} and so-pairs in \cite{MT}. It is easy to see that weakly o-minimal types (pairs) are so-types (so-pairs), that is, if $\mathbf p=(p,<)$ is a weakly o-minimal pair over $A$, then for each relatively definable subset $D\subseteq p(\Mon)$ one of the sets $D$ and $p(\Mon)\smallsetminus D$ is right-eventual in $(p(\Mon),<)$ (contains a final part of $(p(\Mon),<)$) and one of them is left-eventual. In particular, we can define the left ($\mathbf p_l$) and the right ($\mathbf p_r$) globalization of $\mathbf p$:
\begin{enumerate}[label=$\bullet$,left=1em,labelsep=.7em] 
\item $\mathbf p_l(x):=\{\phi(x)\in L_{\Mon}\mid \phi(\Mon)\cap p(\Mon) \mbox{ is left-eventual in } (p(\Mon),<)\}$ \ and

\item $\mathbf p_r(x):=\{\phi(x)\in L_{\Mon}\mid \phi(\Mon)\cap p(\Mon) \mbox{ is right-eventual in } (p(\Mon),<)\}$.
\end{enumerate}
It is easy to see that $\mathbf p_l$ and $\mathbf p_r$ are global $A$-invariant extensions of $p$; by \cite[Fact 2.13]{MTwom}, they are the only such extensions. Moreover, by \cite[Corollary 5.2]{MTwom}, they are the only global nonforking extensions of $p$. The forking extensions of $p$ have a simple description: By \cite[Corollary 5.2]{MTwom}, for any formula $\phi(x)$ consistent with $p(x)$ we have: $p(x)\cup \{\phi(x)\}$ forks over $A$ if and only if $p(x)\cup \{\phi(x)\}$ divides over $A$ if and only if
$\phi(x)$ relatively defines a bounded (from both sides) subset in $(p(\Mon),<)$. Such formulas are called {\it relatively $p$-bounded}. Note that the first two conditions do not refer to any specific order, so these formulae relatively define subsets that are bounded with respect to any relatively $A$-definable linear order.    

\begin{Fact}\label{Fact_basic_S_p_wom}
Let $\mathbf p= (p,<_p)$ be a weakly o-minimal pair over $A$ and let $B\supseteq A$. 
Consider the set $S_p(B)=\{q\in S_x(B)\mid p(x)\subseteq q(x)\}$.
\begin{enumerate}[(a),left=0em,labelsep=.7em]
\item  $q(\Mon)$ is a convex subset of $p(\Mon)$ and $(q,<_p)$ is a weakly o-minimal pair over $B$ for each $q\in S_p(B)$.
\item The set $S_p(B)$ is linearly ordered by:  $q<_pq'$ iff $q(\Mon)<_pq'(\Mon)$;  $(\mathbf p_{l})_{\restriction B}=\min S_p(B)$ and  $(\mathbf p_{r})_{\restriction B}=\max S_p(B)$.
\item For all $q\in S_p(B)$: $q$ forks over $A$ if and only if the locus $q(\Mon)$ is bounded (from both sides) in $(p(\Mon),<_p)$ if and only if $q$ contains a relatively $p$-bounded formula.
\item $(\mathbf p_{l})_{\restriction B}$ and  $(\mathbf p_{r})_{\restriction B}$ are the only nonforking extensions of $p$ in $S_p(B)$.
\end{enumerate}
\end{Fact}
\begin{proof}The proofs for all statements can be found in \cite{MTwom}: part (a) corresponds to Lemma 3.6(ii), part (b) to Corollary 3.7, and parts (c) and (d) are derived from Corollary 5.2. 
\end{proof}   

The key notion in studying the forking independence of weakly o-minimal types in \cite{MTwom} was the right (and left) $\mathbf p$-genericity. 
Let $\mathbf p=(p,<)$ be a weakly o-minimal pair over $A$ and let $B\supseteq A$. By Fact \ref{Fact_basic_S_p_wom}(a), the sets $(q(\Mon)\mid q\in S_p(B))$ form a convex partition of $(p(\Mon),<_p)$. The maximal, or the rightmost element of this partition, is $(\mathbf p_{r})_{\restriction B}(\Mon)$; this motivates the definition of the right $\mathbf p$-genericity.

\begin{Definition}Let $\mathbf p=(p,<_p)$ be a weakly o-minimal pair over $A$ and let $B$ be {\it any} small set.
\begin{enumerate}[(a),left=0em,labelsep=.7em]
    \item $a$ is {\em right $\mathbf p$-generic} over $B$, denoted by $B\triangleleft^{\mathbf p} a$, if $a\models (\mathbf p_{r})_{\restriction AB}$; 
    \item $a$ is {\em left $\mathbf p$-generic} over $B$ if $a\models (\mathbf p_{l})_{\restriction AB}$;
    \item $\mathcal D_p(B):=\{a\in p(\Mon)\mid a\dep_A B\}$;
    \item $\mathcal D_p:=\{(x,y)\in p(\Mon)^2\mid x\dep_A y\}$.
\end{enumerate}
\end{Definition}
\begin{Remark}\label{Remark_pgeneric_first}
Let $\mathbf p=(p,<_p)$ be a weakly o-minimal pair over $A$. 
\begin{enumerate}[(a),left=0em,labelsep=.7em]
\item To avoid possible confusion, no special notation for the left $\mathbf p$-genericity was chosen. However, $B\triangleleft^{\mathbf p^*}a$ means that $a$ is left $\mathbf p$-generic over $B$, where $\mathbf p^*=(p,>_p)$ is the reverse of $\mathbf p$.
\item By Fact \ref{Fact_basic_S_p_wom}(d), the only nonforking extensions of $p$ in $S(AB)$ are $(\mathbf p_{l})_{\restriction AB}$ and $(\mathbf p_{r})_{\restriction AB}$. Therefore, $a\models p$ is (left or right) $\mathbf p$-generic over $B$ if and only if $a\ind_A B$; that is:

\noindent\hfill\hfill $a\ind_A B$ \ \ if and only if \ \ $B\triangleleft^{\mathbf p}a$ or $B\triangleleft^{\mathbf p^*}a$.\hfill($\triangleleft^{\mathbf p}$-comparability) 

\item By Fact \ref{Fact_basic_S_p_wom}(b) we have $(\mathbf p_{l})_{\restriction AB}=\min S_p(B)$, so the locus $(\mathbf p_{l})_{\restriction AB}(\Mon)$, which is the set of all left $\mathbf p$-generic elements over $B$, is an initial part of $p(\Mon)$. Similarly, the set of all right $\mathbf p$-generic elements over $B$, is a final part of $p(\Mon)$.

\item By Fact \ref{Fact_basic_S_p_wom}(c), the set $\mathcal D_p(B)$ consists of all non-$\mathbf p$-generic elements. 
Since the left $\mathbf p$-generic elements form an initial part and the right $\mathbf p$-generic a final part of $(p(\Mon),<_p)$, the set $\mathcal D_p(B)$, if nonempty, is a convex subset of $p(\Mon)$; in that case, by part (c), $a<_p\mathcal D_p(B)<_p b$ holds whenever $a$ is left $\mathbf p$-generic and $b$ right $\mathbf p$-generic over $B$. 
In particular, $\mathcal D_p(a)$ is a convex, bounded subset of $p(\Mon)$ for all $a\models p$.

\item The set $\mathcal D_p(B)$, if nonempty, is relatively $\bigvee$-definable over $AB$ within $p(\Mon)$, by which we mean that $\mathcal D_p(B)=p(\Mon)\cap\bigcup_{\phi\in\Phi}\phi(\Mon)$ holds for some family $\Phi$ of $L_{AB}$-formulas in variable $x$. Here, $\Phi$ can be taken to be the set of all relatively $p$-bounded $L_{AB}$-formulae, since by Fact \ref{Fact_basic_S_p_wom}(c) we have $a\in \mathcal D_p(B)$ if and only if $p(x)\cup \{\phi(x)\}\subseteq \tp(a/AB)$ for some relatively $p$-bounded formula $\phi(x)$. Therefore, $\mathcal D_p(B)$ is a convex relatively $\bigvee$-definable subset of $p(\Mon)$.  

\item $p\wor\tp(B/A)$  if and only if $(\mathbf p_{l})_{\restriction AB}=(\mathbf p_{r})_{\restriction AB}$. 
\end{enumerate}
\end{Remark}

The following theorem gathers key properties of right (left) $\mathbf p$-genericity, viewed as a binary relation on $p(\Mon)$; these are expressed as properties of the structure
$(p(\Mon), \triangleleft^{\mathbf p},<_p,\mathcal D_p)$.

\begin{Theorem}\label{Theorem4} Let $\mathbf p=(p,<_p)$ be a weakly o-minimal pair over $A$. 
\begin{enumerate}[(a),left=0em,labelsep=.7em] 
\item  \ $b\models p$ is right $\mathbf p$-generic over $a\models p$ if and only if $a$ is left $\mathbf p$-generic over $b$ \  ($a\triangleleft^{\mathbf p} b\Leftrightarrow b\triangleleft^{\mathbf p^*}a$). 
\item  $(p(\Mon), \triangleleft^{\mathbf{p}})$ is a strict partial order and $(p(\Mon), <_p)$ its linear extension. 
\item The relations $\mathcal D_{p}$ and the $\triangleleft^{\mathbf p}$-incomparability are the same, $<_p$-convex equivalence relation on $p(\Mon)$.  In particular, for all $a,b\models p$, $\mathcal D_p(a)$ is a convex subset of $(p(\Mon),<_p)$ and:
\begin{center}$a\in\mathcal D_p(b)\Leftrightarrow a\dep_A b\Leftrightarrow \lnot(a\triangleleft^{\mathbf p} b\vee b\triangleleft^{\mathbf p} a)\Leftrightarrow b\dep_A a\Leftrightarrow b\in\mathcal D_p(a)$.
\end{center}
\item The quotient $(p(\Mon) / \mathcal{D}_p, <_p)$ is a dense linear order.
\end{enumerate}   
\end{Theorem}
\begin{proof}
(a) is \cite[Lemma 5.8(ii)]{MTwom}, (b)-(d) are from \cite[Theorem 4]{MTwom}.    
\end{proof}

\begin{Remark}\label{Remark_trianglep_binary_basicprops}
Let $\mathbf p=(p,<_p)$ be a weakly o-minimal pair over $A$ and let $a,b\models p$. 
\begin{enumerate}[(a),left=0em,labelsep=.7em]
\item Theorem \ref{Theorem4}(c) implies the following version of $\triangleleft^{\mathbf p}$-comparability: $a\ind_A b\Leftrightarrow (a\triangleleft^{\mathbf p} b\vee b\triangleleft^{\mathbf p} a)$.

\item By \cite[Remark 5.14]{MTwom}, each of the following conditions is equivalent to $a\triangleleft^{\mathbf p} b$:\\
$(a,b)\models (\mathbf p_r^2)_{\restriction A}$; \   $(b,a)\models  (\mathbf p_l^2)_{\restriction A}$; \ $a<_pb\land a\ind_A b$; \
  $a<_p \mathcal D_p(b)$; \ $\mathcal D_p(a)<_p b$; \ $\mathcal D_p(a)<_p\mathcal D_p(b)$.
  \end{enumerate}
\end{Remark}

\begin{Definition}
Two weakly o-minimal pairs over $A$, $\mathbf p=(p,<_p)$ and $\mathbf q=(q,<_q)$, are said to be {\it directly non-orthogonal}, denoted by $\delta_A(\mathbf p,\mathbf q)$, if $p\nwor q$ and  whenever $a\models p$ is right $\mathbf p$-generic over $b\models q$, then $b$ is left $\mathbf q$-generic over $a$, and conversely, whenever $b$ is right $\mathbf q$-generic over $a$, then $a$ is left $\mathbf p$-generic over $b$. 
\end{Definition}

Intuitively, $\delta_A(\mathbf p,\mathbf q)$ describes that orders $<_p$ and $<_q$ ``have the same direction''. This is justified by the next theorem, where we proved that $\delta_A$ is an equivalence relation and that whenever $\mathbf p\nwor \mathbf q$ (meaning $p\nwor q$), the pair $\mathbf p$ has the same direction with exactly one of the pairs $\mathbf q$ and $\mathbf q^*$; that is, $\delta_A(\mathbf p,\mathbf q)$ or $\delta_A(\mathbf p,\mathbf q^*)$.

\begin{Theorem}[{\cite[Theorem 6.9]{MTwom}}]\label{Theorem_nonorthogonality}
Let $\mathcal{W}_A$ denote the set of all weakly o-minimal types over $A$ and $\mathcal P_A$ the set of all such pairs over $A$. Let $\mathcal{W}_A(\Mon)$  be the set of all realizations of types from $\mathcal{W}_A$.
\begin{enumerate}[(a),left=0em,labelsep=.7em]
    \item $\nwor$ is an equivalence relation on both $\mathcal{W}_A$ and $\mathcal P_A$. \
    \item  $\delta_A$ is an equivalence relation on $\mathcal P_A$; $\delta_A$ refines  $\nwor$  by splitting each class into two classes, each of them consisting of the reverses of pairs from the other class.
    \item $x\dep_A y$ is an equivalence relation on $\mathcal{W}_A(\Mon)$. 
    \item $\nfor$ is an equivalence relation on $\mathcal{W}_A$.
\end{enumerate}
\end{Theorem}

That $\nwor$ and $\nfor$ are equivalence relation on the set of all complete 1-types over the same domain in a weakly o-minimal theory was first proved by Baizhanov in \cite{Baizhanov1999} (where $\fwor$ is called almost orthogonality). 

 \begin{Fact}\label{Fact_notdelta}
Suppose that $\mathbf p=(p,<_p)$ and $\mathbf q=(q,<_q)$ are weakly o-minimal pairs over $A$ and $\mathbf p\nwor\mathbf q$. 
\begin{enumerate}[(a),left=0em,labelsep=.7em]
\item  Exactly one of $\delta_A(\mathbf p,\mathbf q)$ and $\delta_A(\mathbf p,\mathbf q^*)$ holds. 

\item $\lnot\delta_A(\mathbf p,\mathbf q)$ if and only if $\delta_A(\mathbf p,\mathbf q^*)$ if and only if $a\triangleleft^{\mathbf q} b\triangleleft^{\mathbf p}a$ holds for some $a\models p$ and $b\models q$. 

\item  If $p=q$, then $\delta_A(\mathbf p,\mathbf q)$ if and only if $\mathbf p_r=\mathbf q_r$ if and only if $\triangleleft^{\mathbf p}=\triangleleft^{\mathbf q}$ on $p(\Mon)$.
\end{enumerate}
\end{Fact}
\begin{proof}
(a) follows from Theorem \ref{Theorem_nonorthogonality}(b), (b) from \cite[Remark 6.3]{MTwom}, and (c) from \cite[Lemma 6.4]{MTwom}.
\end{proof}

As a consequence of Theorem \ref{Theorem_nonorthogonality} and Fact \ref{Fact_notdelta}(c), we have the following definition.

\begin{Definition}
Let $\mathcal F$ be a part of a $\delta_A$-class of weakly o-minimal pairs and let $\mathcal F(\Mon)$ be the set of all realizations of types from $\mathcal F$. 
\begin{enumerate}[(a),left=0em,labelsep=.7em]
\item For $a,b\in \mathcal F(\Mon)$ 
define \ $a\triangleleft^{\mathcal F} b$ \ if and only if \ $a\triangleleft ^{\mathbf q}b$ holds for some (or equivalently, {\it any} according to Fact \ref{Fact_notdelta}(c)) weakly o-minimal pair $\mathbf q=(q,<_q)\in\mathcal F$ such that $b\models q$. 

\item $\mathcal D_{\mathcal F}:=\{(x,y)\in \mathcal F(\Mon)^2\mid x\dep_A y\}$.
\end{enumerate}
\end{Definition}

\begin{Theorem}[{\cite[Theorem 5]{MTwom}}]\label{Theorem5}
Let $\mathcal F$ be a part of a $\delta_A$-class of non-algebraic weakly o-minimal pairs over $A$. Then there is an $A$-invariant order $<^{\mathcal F}$ on $\mathcal F(\Mon)$ such that the structure
$(\mathcal F(\Mon),\triangleleft^{\mathcal F},<^{\mathcal F},\mathcal D_{\mathcal F})$ satisfies the following conditions:  
\begin{enumerate}[(a),left=0em,labelsep=.7em]
\item $(\mathcal{F}(\Mon), \triangleleft^{\mathcal F})$ is a strict partial order and $(\mathcal{F}(\Mon), <^{\mathcal F})$ is its linear extension.

\item Relations $\mathcal D_{\mathcal F}$ and  the $\triangleleft^{\mathcal F}$-incomparability are the same, $<^{\mathcal F}$-convex equivalence relation on $\mathcal{F}(\Mon)$.  The quotient $(\mathcal{F}(\Mon)/\mathcal D_{\mathcal F},<^{\mathcal F})$ is a dense linear order.
\item For each $p\in S(A)$ represented in $\mathcal F$, the order $<_p=(<^{\mathcal F})_{\restriction p(\Mon)}$ is relatively $A$-definable and the pair $(p,<_p)$ is directly non-orthogonal to pairs from $\mathcal F$.   $(p(\Mon), \triangleleft^{\mathbf p},<_p,\mathcal D_p)$ is a substructure of $(\mathcal F(\Mon),\triangleleft^{\mathcal F}, <^{\mathcal F},\mathcal D_{\mathcal F})$.  
\end{enumerate}
Moreover, if $(\mathbf p_i=(p_i,<_i)\mid i\in I)$ is a family of pairs from $\mathcal F$, with the types $(p_i\mid i\in I)$ mutually distinct, then the order $<_{\mathcal F}$ can be chosen to extend each $<_i$ for $i\in I$. 
\end{Theorem}

\begin{Theorem}\label{Theorem_triangle_mathcal F}
Let $\mathcal F$ be part of a $\delta_A$-class of weakly o-minimal pairs over $A$ and let $<_{\mathcal F}$ be given by Theorem \ref{Theorem5}. Then for all small $B$, $\mathbf p=(p,<_p)\in\mathcal F$, and  $a,b,c\in \mathcal{F}(\Mon)$ the following hold true: 
\begin{enumerate}[(a),left=0em,labelsep=.7em]
\item (Existence) \ There exists  $c\models p$ such that $B\triangleleft^{\mathcal F}c$;
\item (Density) \ If $B\triangleleft^{\mathcal F} b$, then there is  $c$ such that $B\triangleleft^{\mathbf p} c\triangleleft^{\mathcal F} b$; 
\item (Symmetry) \ $a\ind_A b$ if and only if $b\ind_A a$; 
\item   \   
$a \triangleleft^{\mathcal F} b\Leftrightarrow b\triangleleft^{\mathcal F^*} a$ \ ($b$ is right $\mathcal F$-generic over $a$ if and only if $a$ is left $\mathcal F$-generic over $b$.) 
\item  ($\triangleleft^{\mathcal F}$-comparability) \ $a\ind_A B\Leftrightarrow (B \triangleleft^{\mathcal F} a\vee B\triangleleft^{\mathcal F^*} a)$
\ and \ 
$a\ind_A b\Leftrightarrow (a \triangleleft^{\mathcal F} b\vee b\triangleleft^{\mathcal F} a)$;
\item  (Transitivity) \ $B\triangleleft^{\mathcal F} a \triangleleft^{\mathcal F} b$ \ implies \ $B\triangleleft^{\mathcal F} b$;
 \item  $B\triangleleft^{\mathcal F} a \land a\dep_A b$ implies $B\triangleleft^{\mathcal F} b\land a\dep_{B}b$;
\item   $B\triangleleft^{\mathcal F} a<^{\mathcal F}b$ implies $B\triangleleft^{\mathcal F} b$;  
\item   $B\triangleleft^{\mathcal F} a<^{\mathcal F}b\land  \, a\ind_B b$ implies $Ba\triangleleft^{\mathcal F} b$; 
\item   $B\triangleleft^{\mathcal F}a<^{\mathcal F}b<^{\mathcal F}c\land a\dep_{B}c$ implies  $b\dep_{B} a \land b\dep_{B}c$. 
\end{enumerate}
\end{Theorem}
\begin{proof}
(a)-(f) are Theorem 6.11(i)-(vi) of \cite{MTwom} and (g)-(j) are Proposition 6.12(iii)-(vi) of \cite{MTwom}. 
\end{proof} 

Most applications of Theorem \ref{Theorem_triangle_mathcal F} will be in the case where $\mathcal F=\{\mathbf p,\mathbf q\}$ (with $\delta_A(\mathbf p,\mathbf q)$ assumed). This is summarized in the next corollary.

\begin{Corollary}\label{Cor_triangle_mathbf_pq_properties}
 $\delta_A(\mathbf p,\mathbf q)$ implies that for all $B$, $a\models p$ and $b\models q$:
\begin{enumerate}[(a),left=0em,labelsep=.7em]
\item (Existence) \ There exists  $a'\models p$ such that $B\triangleleft^{\mathbf p}a'$;
\item (Density) \ If $B\triangleleft^{\mathbf q} b$, then there is  $a'\models p$ such that $B\triangleleft^{\mathbf p} a'\triangleleft^{\mathbf q} b$; 
\item (Symmetry) \ $a\ind_A b$ if and only if $b\ind_A a$; 
\item   \   
$a \triangleleft^{\mathbf q} b\Leftrightarrow b\triangleleft^{\mathbf p^*} a$ \ ($b$ is right $\mathbf q$-generic over $a$ if and only if $a$ is left $\mathbf p$-generic over $b$.) 
\item  ($\triangleleft^{\mathbf p}$-comparability) \ $a\ind_A B\Leftrightarrow (B \triangleleft^{\mathbf p} a\vee B\triangleleft^{\mathbf p^*} a)$
\ and \ 
$a\ind_A b\Leftrightarrow (a \triangleleft^{\mathbf q} b\vee b\triangleleft^{\mathbf p} a)$;
\item  (Transitivity) \ $B\triangleleft^{\mathbf p} a \triangleleft^{\mathbf q} b$ \ implies \ $B\triangleleft^{\mathbf q} b$; 
 \item  $B\triangleleft^{\mathbf p} a \land a\dep_A b$ implies $B\triangleleft^{\mathbf q} b\land a\dep_{B}b$.
\end{enumerate}
\end{Corollary}

\begin{Corollary}\label{Corolary_a_1_equiv_a_2_AB}
Suppose that $p\in S(A)$ is a weakly o-minimal type and $a_1,a_2\models p$. Then for all $B$:  
\begin{center}
  $a_1\dep_A a_2 \ \land \ a_1\ind_A B$ \ implies \ $a_1\equiv a_2 \,(AB)$.   
\end{center}
\end{Corollary}
\begin{proof}
Assume $a_1\dep_A a_2 \ \land \ a_1\ind_A B$. Choose $\mathbf p=(p,<)$, a weakly o-minimal pair over $A$, such that $B\triangleleft^{\mathbf p} a_1$. Then by Corolary \ref {Cor_triangle_mathbf_pq_properties}(g), applied to $\mathbf p=\mathbf q$, $B\triangleleft^{\mathbf p} a_1$ and $a_1\dep_A a_2$ together imply $B\triangleleft^{\mathbf p} a_2$. Therefore, $a_1\equiv a_2\,(AB)$.  
\end{proof}

\begin{Fact}[{\cite[Corollary 6.14]{MTwom}}]\label{Fact_delta_implies_deltaB_fworB}
Suppose that $\mathbf p,\mathbf q$ are weakly o-minimal pairs over $A$ with $\delta_A(\mathbf p,\mathbf q)$. Let $B\supseteq A$. Define $p_B=(\mathbf p_r)_{\restriction B}$, $\mathbf p_B=(p_B,<_p)$ and analogously define $q_B$ and $\mathbf q_B$.   
    \begin{enumerate}[(a),left=0em,labelsep=.7em]
\item $\delta_B(\mathbf p_B,\mathbf q_B)$ holds;  in particular, $p_B\nwor q_B$ and $\mathbf p_r\nwor \mathbf q_r$.
\item If  $p\nfor q$, then $p_B\nfor q_B$ and, in particular, $\mathbf p_r\nfor\mathbf q_r$.  
\end{enumerate}
\end{Fact}

\begin{Fact}[{\cite[Proposition 7.4]{MTwom}}]
    Every weakly o-minimal type is dp-minimal, so NIP. 
\end{Fact}

We have already mentioned that, unlike weak o-minimality, weak quasi-o-minimality of (one-sorted) $T$ does not depend on the choice of a 0-definable order: if $T$ is weakly quasi-o-minimal with respect to one 0-definable order, then it is so with respect to any other 0-definable order. This can be deduced from the following fact.

\begin{Proposition}[{\cite[Proposition 4.7]{MTwom}}]\label{Proposition_wqom_iff_all_1types_wom}
A complete first-order theory $T$ with infinite models is weakly quasi-o-minimal if and only if every type $p\in S_1(T)$ is weakly o-minimal.    
\end{Proposition}

A ``local" context of weak quasi-o-minimality was introduced in \cite[Definition 4.8]{MTwom}: A partial type $\Pi(x)$ is weakly quasi-o-minimal over $A$ if $\Pi(x)$ is over $A$ and every type $p\in S_\Pi(A)$ is weakly o-minimal. In that case, we also say that the set $\Pi(\Mon)$ is weakly quasi-o-minimal over $A$. 
This definition was motivated (and justified) by the next proposition whose corollary is that the theory $T$ is weakly quasi-o-minimal if and only if the formula $x=x$ is weakly quasi-o-minimal.  

\begin{Proposition}[{\cite[Proposition 4.9]{MTwom}}] 
Let $P$ be type-definable over $A$. Then $P$ is weakly quasi-o-minimal over $A$ if and only if there exists a relatively $A$-definable linear order $<$ on $P$ such that every relatively definable subset of $P$ is a Boolean combination of $<$-convex and relatively $A$-definable sets. In that case, the latter is true for any relatively $A$-definable linear order $<$ on $P$.
\end{Proposition}

\section{Interval types, convex types, and almost weakly o-minimal theories}\label{Section_IT convex}

In this section, we study convex weakly o-minimal types. The main technical result is Proposition \ref{Proposition_convex_type_witness}, in which we prove that the convexity of a weakly o-minimal type $p\in S(A)$ does not essentially depend on the choice of a relatively $A$-definable order on $p(\Mon)$: for any such order, there exists an $A$-definable extension of $(p(\Mon),<)$ in which $p(\Mon)$ is a convex subset. We also introduce almost weakly o-minimal theories which nearly approximate weakly o-minimal theories within weakly quasi-o-minimal theories. 

Let $(D,<)$ be an $A$-definable linear order. A formula $\phi(x)$ is said to be {\em convex with respect to $(D,<)$} if it implies $x\in D$ and defines a convex subset of $D$; in general, the meaning of $(D,<)$ will be clear from the context, so we simply say that $\phi(x)$ is convex. For a convex formula $\phi(x)$, by $\phi^-(x)$ and $\phi^+(x)$ we will denote formulae $x\in D\land x<\phi(\Mon)$ and $x\in D\land x>\phi(\Mon)$ respectively. Note that both $\phi^-(x)$ and $\phi^+(x)$ are convex formulae given over the same parameters as $\phi(x)$, and that $\phi^-(\Mon)<\phi(\Mon)<\phi^+(\Mon)$ is a convex partition of $D$.

By an {\em interval type of $(D,<)$ over $B\supseteq A$} we mean a maximal partial type $\Pi(x)$ consisting of convex $L_B$-formulae. Since for each convex $L_B$-formula $\phi(x)$ exactly one of $\phi^-(x)$, $\phi(x)$ and $\phi^+(x)$ belongs to an interval type in $(D,<)$ over $B$, an interval type in $(D,<)$ over $B$ may be described as a consistent set of convex $L_B$-formulae $\Pi(x)$ such that for each convex $L_B$-formula $\phi(x)$ either $\phi(x)\in \Pi(x)$ or $\phi(x)$ is inconsistent with $\Pi(x)$.  
The set of all interval types in $(D,<)$ over $B\supseteq A$ is denoted by $IT(B)$. Note that, although we do not emphasize it, $IT(B)$ depends on $A$, $D$ and $<$; however, these will always be clear from the context. The set $IT(B)$ is endowed with a compact Hausdorff topology in the usual way. 

For $a\in D$, {\em the interval type of $a$ in $(D,<)$ over $B\supseteq A$}, denoted by $\itp(a/B)$ (we again suppress $A$, $D$ and $<$), is the set of all convex $L_B$-formulas that are satisfied by $a$. It is easy to see that $\itp(a/B)\in IT(B)$. For a type $p\in S_x(B)$ that implies $x\in D$, by $p^{\conv}$ we denote the subtype consisting of all convex $L_B$-formulae from $p$. Again, it is easy to see that $p^{\conv}\in IT(B)$.

In the following lemma, we list some well-known facts about interval types.

\begin{Lemma}\label{Lemma_interval_types}
Let $B\supseteq A$.
\begin{enumerate}[(a),left=0em,labelsep=.7em]
    \item For all $\Pi\in IT(B)$, $\Pi(\Mon)$ is a convex subset of $D$.
    \item For distinct $\Pi_1,\Pi_2\in IT(B)$, $\Pi_1(\Mon)$ and $\Pi_2(\Mon)$ are disjoint and $<$-comparable.
    \item The space $IT(B)$ is naturally linearly ordered by $<$.
    \item For all $p\in S_x(B)$ implying $x\in D$, $p^{\conv}(\Mon)=\conv (p(\Mon))$ (the convex hull of $p(\Mon)$).
    \item If $\Pi\in IT(B)$ and if $p\in S_x(B)$ is a completion of $\Pi$, then $\Pi=p^{\conv}$ and $\Pi(\Mon)=\conv(p(\Mon))$. 
    \item Let $p,q\in S_x(A)$ be distinct. Then: $p^{\conv}=q^{\conv}$ if and only if $\conv(p(\Mon))=\conv (q(\Mon))$ if and only if there are $a,a'\models p$ and $b\models q$ such that $a'<b<a''$. 
\end{enumerate}
\end{Lemma}
\begin{proof}
(a) Clearly, $\Pi(\Mon)$ is convex as an intersection of convex sets.
For (b), if $\phi(x)\in\Pi_1\smallsetminus\Pi_2$, then $\phi(x)$ is inconsistent with $\Pi_2$, so $\Pi_1(\Mon)\cap\Pi_2(\Mon)=\emptyset$.
Now, (c) follows directly from (a) and (b).

(d) It is clear that $p^{\conv}(\Mon)$ contains $\conv(p(\Mon))$. If $a\in D$ and $a\notin \conv(p(\Mon))$, then $a<p(\Mon)$ or $a>p(\Mon)$. Without loss, suppose that the former holds. By compactness, there is a $\theta(x)\in p(x)$ such that $a<\theta(\Mon)$. Denote by $\theta^{\conv}(x)$ the formula that describes the convex hull of $\theta(\Mon)\cap D$. Then $a<\theta^{\conv}(\Mon)$ and $\theta^{\conv}(x)\in p^{\conv}$, so $a\notin p^{\conv}(\Mon)$.

(e) Clearly, $\Pi\subseteq p^{\conv}$ holds, so $\Pi=p^{\conv}$ follows by the maximality of $\Pi$, and $\Pi(\Mon)=\conv(p(\Mon))$ follows by (d). 

(f) Easy.
\end{proof}

In the literature, the term ``convex type" is commonly used to denote an interval type, but here we will follow \cite{MT} and assign it a different meaning.

\begin{Definition}\label{Definition_convex}
A complete type $p\in S(A)$ is {\it convex} if there exists an $A$-definable order $(D,<)$ in which $p(\Mon)$ is convex. 
\end{Definition}

\begin{Lemma}\label{Lemma_D_p_rel_def_convex}
Let $\mathbf p=(p,<)$ be a weakly o-minimal pair over $A$ and let $(D,<)$ be an $A$-definable extension of $(p(\Mon),<)$. Then there is a family $\Phi(x,y)=\{\phi_i(x,y)\mid i\in I\}$ of $L_{A}$-formulae such that for all $a\models p$, the disjunction $\bigvee_{i\in I}\phi_i(x,a)$ relatively defines $\mathcal D_p(a)$ within $p(\Mon)$ and for all $i\in I$ the set $\phi_i(\Mon,a)$ is a convex subset of $D$. 
\end{Lemma}
\begin{proof}
Let $\Phi_0(x,a)$ denote the set of all relatively $p$-bounded $L_{Aa}$-formulae in variable $x$. 
By Remark \ref{Remark_pgeneric_first}(e), the disjunction $\bigvee \Phi_0(x,a)$ relatively defines $\mathcal D_p(a)$ within $p(\Mon)$. Define: \
\begin{center}
$\Phi(x,a)=\{\phi(x,a)\in \Phi_0\mid \mbox{ $\phi(\Mon,a)$ is a convex subset of $D$}\}$.
\end{center}
We claim that the family $\Phi(x,y)$ satisfies the conclusion of the lemma: $\bigvee \Phi(x,a)$ relatively defines $\mathcal D_p(a)$ within $p(\Mon)$. To prove this, it suffices to show that for each formula $\phi(x,a)\in \Phi_0(x,a)$ there exists a $\phi^*(x,a)\in \Phi(x,a)$ such that $\phi(p(\Mon),a)\subseteq \phi^*(p(\Mon),a)$. Let $\phi(x,a)\in \Phi_0$ and let $b_1,b_2\models p$ satisfy $b_1\triangleleft^{\mathbf p}a\triangleleft^{\mathbf p}b_2$. By Remark \ref{Remark_trianglep_binary_basicprops} we have $b_1<\mathcal D_p(a)<b_2$. 
Since the set $\phi(p(\Mon),a)$ is a subset of $\mathcal D_p(a)$ we have: $p(x)\cup\{\phi(x,a)\}\vdash b_1<x<b_2$. By compactness, there is a formula $\theta(x)\in p$, without loss implying $x\in D$, such that $\models \theta(x)\land \phi(x,a)\rightarrow b_1<x<b_2$. Let $\phi^*(x,a)$ be a formula that defines the convex hull of $\phi(\Mon,a)\cap \theta(\Mon)$ in $(D,<)$. Clearly, $b_1<\phi^*(\Mon,a)<b_2$, so $\phi^*(x,a)$ is a relatively $p$-bounded formula, and $\phi(x,a)\in \Phi(x,a)$.  As $\theta(x)\in p$, the formula $\theta(x)\land \phi(x,a)$ relatively defines the same subset of $p(\Mon)$ as $\phi(x,a)$ does; clearly, this implies the desired conclusion $\phi^*(p(\Mon),a) \supseteq \phi(p(\Mon),a)$ and completes the proof of the lemma.  
\end{proof}

\begin{Proposition}\label{Proposition_convex_type_witness} 
Suppose that $\mathbf p=(p,<)$ is a weakly o-minimal pair over $A$ and that $p$ is convex. Then there is a formula $\theta(x)\in p$ such that $(\theta(\Mon),<)$ is an $A$-definable extension of $(p(\Mon),<)$ and $p(\Mon)$ is a convex subset of $\theta(\Mon)$; we will say that $\theta(x)$ witnesses the convexity of $\mathbf p$.     
\end{Proposition}
\begin{proof}
Let $(D,<)$ be an $A$-definable extension of $(p(\Mon),<)$. Since $p$ is convex, there is a linear order $A$-definable $(D_p,<_p)$ that contains $p(\Mon)$ as a convex subset; by replacing the domains $D$ and $D_p$ by their intersection, we can assume $D=D_p$. Denote $\mathfrak p=(p,<_p)$. Reversing the order $<_p$ if necessary, we may assume that $<$ and $<_p$ have the same orientation; that is, $\mathbf p_r=\mathfrak p_r$. 

Let $a\models p$ and let $\Phi(x,y)$ be the set of formulas given by Lemma \ref{Lemma_D_p_rel_def_convex}:  $\bigvee\Phi(x,a)$ relatively defines $\mathcal D_p(a)$ within $p(\Mon)$ and for all $\phi(x,a)\in \Phi(x,a)$ the set $\phi(\Mon,a)$ is a convex subset of $D$. Then $\{\phi(\Mon,a)<x\mid \phi(x,a)\in \Phi(x,a)\}$ is a type, denote it by $\bigcup\Phi(D,a)<x$. Similarly, define the type $x<\bigcup\Phi(D,a)$. 
Since $\bigvee\Phi(x,a)$ relatively defines the set $\mathcal D_p(a)$ within $p(\Mon)$, we have $p(x)\cup (\bigcup\Phi(D,a)<x)\vdash (\mathbf p_r)_{\restriction Aa}(x)$ and, since $\mathbf p_r=\mathfrak p_r$, $p(x)\cup (\bigcup\Phi(D,a)<x)\vdash a<_px$. Similarly, for all $b\models p$ we have $p(x)\cup (x<\bigcup\Phi(D,b))\vdash x<_p b$. So:
\begin{equation*}p(x)\cup \left(\bigcup\Phi(D,a)< x\right)\vdash a<_p x \ \mbox{ and } \ p(x)\cup \left(x<\bigcup\Phi(D,b)\right)\vdash x<_p b .
\end{equation*}
By compactness, there is a formula $\theta(x)\in p(x)$ that implies $x\in D$ such that:
\begin{equation}\{\theta(x)\}\cup \left(\bigcup\Phi(D,a)< x\right)\vdash a<_p x \ \mbox{ and } \ \{\theta(x)\}\cup \left(x<\bigcup\Phi(D,b)\right)\vdash x<_p b .
\end{equation}
\setcounter{equation}{0} 

We claim that $\theta(x)$ satisfies the conclusion of the lemma: $p(\Mon)$ is convex in $(\theta(\Mon),<)$. To prove it, assume $a',b'\models p$ and $\models a'<c<b'\land \theta(c)$; we need to prove $c\models p$. Choose $a,b\models p$ satisfying $a\triangleleft^{\mathbf p}a'$ and $b'\triangleleft^{\mathbf p}b$. Since each formula from $\Phi(x,a)$ defines a convex subset of $D$, the locus of the type $\bigcup\Phi(D,a)<x$ is a final part of $(D,<)$; clearly, it contains $a'$, so $a'<c\in D$ implies  $c\models (\bigcup\Phi(D,a)<x)$, which, together with $\models \theta(c)$, by (1), implies $a<_pc$. Similarly, $c<_pb$. Since $p(\Mon)$ is convex in $(\theta(\Mon),<_p)$ and $a<_p c<_p b$, we conclude $c\models p$, as desired. 
\end{proof}
\setcounter{equation}{0}  

As a corollary, we obtain an important characterization of convexity for weakly o-minimal types. 

\begin{Corollary}\label{Corollary_convex_iff_isolated_in_IT} 
Let $p\in S_x(A)$ be weakly o-minimal. 
\begin{enumerate}[(a),left=0em,labelsep=.7em]
\item The convexity of $p$ does not depend on the choice of a relatively $A$-definable order on $p(\Mon)$, that is, if $p$ is convex then every relatively $A$-definable order on $p(\Mon)$ has an $A$-definable extension in which $p(\Mon)$ is a convex set.
\item  The following conditions are equivalent:
\begin{enumerate}[(1),left=0em,labelsep=.7em]
    \item $p$ is convex;
    \item For all $A$-definable linear orders $(D,<)$ with $p(\Mon)\subseteq D$, $p$ is an isolated point of $S_{p^{\conv}}$;
    \item There exists an $A$-definable linear order $(D,<)$ and an interval type $\Pi\in IT(A)$ such that $p$ is an isolated point of $S_{\Pi}$.
\end{enumerate}
\end{enumerate}
\end{Corollary}
\begin{proof} (a) is immediate, so we prove (b).
(1)$\Rightarrow$(2) Suppose that $p$ is convex and that $(D,<)$ is an $A$-definable linear order such that $p(\Mon)\subseteq D$; then $p^{\conv}\in IT(A)$. 
By Proposition \ref{Proposition_convex_type_witness} there is a formula $\theta(x)\in p$ such that $p(\Mon)$ is convex in $(\theta(\Mon),<)$. Replacing $\theta(x)$ with $\theta(x)\land x\in D$, we obtain $\theta(\Mon)\subseteq D$; note that $p(\Mon)$ is still convex in $(\theta(\Mon),<)$. By Lemma \ref{Lemma_interval_types}(d) we know that $p^{\conv}(\Mon)$ is the convex hull of $p(\Mon)$, so since $p(\Mon)$ is convex in $(\theta(\Mon),<)$, 
we conclude $p^{\conv}(\Mon)\cap\theta(\Mon)=p(\Mon)$. 
Therefore, $p^{\conv}(x)\cup \{\theta(x)\}\vdash p(x)$; $p$ is an isolated point of $S_{p^{\conv}}$. 

(2)$\Rightarrow$(3) is trivial, so we prove (3)$\Rightarrow$(1).  Assume that $\theta(x)\in p(x)$ implies $x\in D$ and isolates $p$ within $S_\Pi$; then $p(\Mon)=\Pi(\Mon)\cap\theta(\Mon)$. As $\Pi(\Mon)$ is convex in $(D,<)$, it follows that $p(\Mon)$ is convex in $(\theta(\Mon),<)$. Therefore, $p$ is a convex type.  
\end{proof}

The following example shows that the weak o-minimality of $p$ cannot be omitted in Proposition \ref{Proposition_convex_type_witness}. 

\begin{Example}\label{Example_convex_depends on order}(A convex type $p\in S_1(T)$ of a saturated linearly ordered structure $(M,\triangleleft, \dots)$ such that $p(M)$ is not a $\triangleleft$-convex subset of any 0-definable subset of $M$.)\\
Let $\mathcal K$ be the class of all finite structures $(A,<,\triangleleft,C_n)_{n\in\omega}$ such that:
\begin{enumerate}[label=$\bullet$,left=1em,labelsep=.7em]
    \item $<$ and $\triangleleft$ are linear orders;
    \item Each $C_n$ is a $<$-convex subset of $A$;
    \item $C_0\supseteq C_1\supseteq C_2\supseteq \dots$.
\end{enumerate}
Clearly, $\mathcal K$ is closed for substructures and has the amalgamation property. Let $\mathcal M=(M,<,\triangleleft,C_n)_{n\in\omega}$ be the Fra\"iss\'e limit of $\mathcal K$ and let $T=\Th(\mathcal M)$. It is routine to verify that $\mathcal M$ is saturated and that $T$ is model complete and eliminates quantifiers.  Clearly, $(M,<)$ and $(M,\triangleleft)$ are dense endless linear orders. 

By elimination, the set $\{C_n(x)\mid n\in\omega\}$ determines a complete type $p(x)\in S_1(T)$. Clearly, $p$ is convex, as witnessed by $(M,<)$. Let $D\subseteq M$ be any 0-definable set that contains $p(M)$. We claim that $p(M)$ is not convex in $(D,\triangleleft)$. By compactness, there is $n$ such that $C_n\subseteq D$, so it suffices to show that $p$ is not convex in $(C_n,\triangleleft)$. 
By construction, for each $b\in M$ the set $\{x\in p(M)\mid x\triangleleft b\}$ is dense and codense in each $(C_n,<)$ ($n\in\omega$), and consequently also in $(p(M),<)$; similarly, for $\{x\in p(M)\mid b\triangleleft x\}$. Thus, for each $b\in M$ there are $a_1,a_2\in p(M)$ such that $a_1\triangleleft b\triangleleft a_2$. Choose $b\in C_n\smallsetminus p(M)$ to conclude that $p(M)$ is not convex in $(C_n,\triangleleft)$. Therefore, $p$ is a convex type, but $p(M)$ is not a $\triangleleft$-convex subset of any 0-definable subset of $M$.
\end{Example}

\begin{Definition}
A complete first-order theory $T$ is {\it almost weakly o-minimal} if every type $p\in S_1(T)$ is weakly o-minimal and convex. 
\end{Definition}

Clearly, weakly o-minimal theories are almost weakly o-minimal. It is easy to see that almost weakly o-minimal theories are weakly quasi-o-minimal: by Proposition \ref{Proposition_wqom_iff_all_1types_wom}, weak o-minimality of all types from $S_1(T)$ implies weak quasi-o-minimality of $T$. Therefore:
$$\mbox{ $T$ weakly o-minimal   \ $\Rightarrow$    \ $T$ almost weakly o-minimal \ $\Rightarrow$ \ $T$ weakly quasi-o-minimal.}$$ 
In Examples \ref{Examle_almostwom_not_wom} and \ref{Example_wqom_not_almostwom}, we demonstrate that both implications are non-reversible in general. Prior to that, we provide a characterization of almost weak o-minimality.

\begin{Proposition}\label{Proposition_everywom convex iff It finite}
$T$ is almost weakly o-minimal if and only if it is weakly quasi-o-minimal and for some (equivalently, any) $0$-definable order $<$ every interval type $\Pi\in IT(\emptyset)$ has finitely many completions in $S_1(T)$.
\end{Proposition}
\begin{proof}
First, assume that $T$ is almost weakly o-minimal and we prove that for any $0$-definable order $<$ every interval type $\Pi\in IT(\emptyset)$ has finitely many completions in $S_1(T)$. Consider the space $S_{\Pi}$. It is compact as a closed subspace of $S_1(T)$ and discrete by Corollary \ref{Corollary_convex_iff_isolated_in_IT}.  Therefore, $S_{\Pi}$ is finite.

Now, assume that $T$ is weakly quasi-o-minimal and that for some $0$-definable order $<$ every interval type $\Pi\in IT(\emptyset)$ has finitely many completions in $S_1(T)$. Then every type $p\in S_1(T)$ is an isolated point of the space $S_{p^{\conv}}$, so by Corollary \ref{Corollary_convex_iff_isolated_in_IT}, $p$ is convex; $p$ is weakly o-minimal by the assumed weak quasi-o-minimality of $T$. Therefore, every type from $S_1(T)$ is weakly o-minimal and convex; $T$ is almost weakly o-minimal.  
\end{proof}

Let $T$ be weakly quasi-o-minimal and let $<$ be a 0-definable order. Then $T$ is weakly o-minimal with respect to $<$ if and only if the locus of every type $p\in S_1(T)$ is convex in $(\Mon,<)$, that is, every interval type has a unique completion. Therefore, the previous proposition suggests that the almost weak o-minimality is a close approximation of the weak o-minimality.

\begin{Example}\label{Examle_almostwom_not_wom}(Almost weakly o-minimal theory which is not weakly o-minimal with respect to any 0-definable order.)\\
Consider the colored order $(\mathbb Q,<,I_n,P_n,R_n)_{n\in\mathbb N}$ such that:
\begin{enumerate}[label=$\bullet$,left=1em,labelsep=.7em]
    \item $(\mathbb Q,<)$ is a usual linear order and $I_n$ is the open interval $(n,n+1)$ ($n\in \mathbb N$);
    \item $\{P_n,R_n\}$ is a partition of $I_n$ into (topologically) dense pieces for all $n\in\mathbb N$.
\end{enumerate}
It is not hard to see that after naming all natural numbers as constants, the underlying theory $T$ eliminates quantifiers and is almost weakly o-minimal.
It is also clear that $T$ is not weakly o-minimal with respect to $<$. However, $T$ is not weakly o-minimal with respect to any 0-definable order: Let  $\triangleleft$ be a 0-definable order and let $p\in S_1(T)$ be the type of infinitely large element. There is a unique complete 2-type extending $p(x)\cup p(y)\cup\{x<y\}$ because every automorphism, $f$, of $(p(\Mon),<)$ extends to an automorphism of $\Mon$ (defined by $f(x)=x$ for $x\notin p(\Mon)$). This implies that $\triangleleft$ agrees with one of $<$ and $>$ on $p(\Mon)$; without loss, assume that it agrees with $<$. 
By compactness, there is $n\in \mathbb N$ such that $\triangleleft$ and $<$ agree for all $x\geqslant n$. This implies that for any $m>n$, both $R_{m}$ and $P_m$ contain an infinite number of $\triangleleft$-convex components, so $T$ is not weakly o-minimal with respect to $\triangleleft$.
\end{Example}

\begin{Example}\label{Example_wqom_not_almostwom}(A weakly quasi-o-minimal theory that is not almost weakly o-minimal.)\\
Consider the colored order $\mathcal M=(\mathbb Q,<,P_n)_{n\in\mathbb N}$, in which $(P_n\mid n\in\mathbb N)$ is a partition of $\mathbb Q$ into topologically dense subsets. It is easy to see that the theory $\Th(\mathcal M)$ eliminates quantifiers and is weakly quasi-o-minimal. There is a unique interval type $\Pi$. Clearly, the space $S_{\Pi}=S_1(T)$ is infinite and compact, so contains a non-isolated type; that type is not convex by Corollary \ref{Corollary_convex_iff_isolated_in_IT}, so the theory $\Th(\mathcal M)$ is not almost weakly o-minimal.    
\end{Example}

The rest of this section contains a few technical facts that will be used in the proof of Theorem \ref{Theorem1_shift}.

 \begin{Lemma}\label{Lemma_wqom_interval_type}
Suppose that $\Th(\Mon,<,\dots)$ is weakly quasi-o-minimal.
\begin{enumerate}[(a),left=0em,labelsep=.7em] 
\item  $\delta_A(\mathbf p,\mathbf q)$ holds for all $p,q\in S_1(A)$ extending the same interval type ($\mathbf p=(p,<),\mathbf q=(q,<)$). 
\item If $\mathcal F$ is a $\delta_A$-class, then the order $<^{\mathcal F}$ in Theorem \ref{Theorem5} can be chosen so that for each $\Pi\in IT(A)$ with $\Pi(\Mon)\cap\mathcal F(\Mon)\neq\emptyset$ one of $(<^{\mathcal F})_{\restriction \Pi(\Mon)}=\ <$ and $(<^{\mathcal F})_{\restriction \Pi(\Mon)}=\ >$ holds.
\end{enumerate}
\end{Lemma}
\begin{proof}
(a)  Let $\Pi\in IT(A)$ and let $p,q\in S_{\Pi}$ be distinct completions of $\Pi$. First, we show $p\nwor q$.  We know $p^{\conv}=q^{\conv}=\Pi$, so by Lemma \ref{Lemma_interval_types}(d) the sets $p(\Mon)$ and $q(\Mon)$ have the same convex hull in $(\Mon,<)$. Thus, for any $a\models p$ there are $b_1,b_2\models q$ such that $b_1<a<b_2$; clearly, $\tp(ab_1)\neq \tp(ab_2)$ witnesses $p\nwor q$. 
Next, we claim that $a\triangleleft^{\mathbf q} b$ implies $a<b$. 
Assume that $b$ is right $\mathbf q$-generic over $a$. By Remark \ref{Remark_pgeneric_first}(c), the set of all right $\mathbf q$-generic elements over $a$, which is the locus of $\tp(b/Aa)$, is a final part of $(q(\Mon),<)$; since $q(\Mon)$ is included in the convex hull of $p(\Mon)$, some $\mathbf q$-generic element, say $b'$, satisfies $a<b'$;  $a<b$ follows, proving the claim. Similarly, we see that $b\triangleleft^{\mathbf p}a$ implies $b<a$.
Therefore, $a\triangleleft^{\mathbf q}b\triangleleft^{\mathbf p}a$ is impossible. By Fact \ref{Fact_notdelta} this implies $\delta_A(\mathbf p,\mathbf q)$. 

(b) Follows from part (a) and the ``moreover" part of Theorem \ref{Theorem5}.
\end{proof}

 \noindent
{\bf Convention.} Whenever the theory $T=\Th(\Mon,<,\dots)$ is weakly quasi-o-minimal, $\mathcal F$ is a part of a $\delta_A$-class, and $<^{\mathcal F}$ satisfies the conclusion of Theorem \ref{Theorem5}, we will assume that $<^{\mathcal F}$ also satisfies the conclusion of the previous lemma.

\begin{Lemma}\label{Lemma_forking_on_interval_type}
Suppose that $\Th(\Mon,<,\dots)$ is weakly quasi-o-minimal.
For each $\Pi\in IT(A)$ define:  
\begin{center}$S_{\Pi}=\{p\in S_1(A)\mid \Pi(x)\subseteq p(x)\}  \ \ \ \mbox{ and } \ \ \ \mathcal D_{\Pi}=\{(a,b)\in \Pi(\Mon)^2\mid a\dep_A b\}.$
\end{center}
\begin{enumerate}[(a),left=0em,labelsep=.7em] 
\item   $(a,b)\in\mathcal D_{\Pi}$ \ if and only if \ $a,b\in \Pi(\Mon)$ and $\tp(a/Ab)$ contains a $\Pi$-bounded formula.
\item $\mathcal D_{\Pi}$ is a convex equivalence relation on $\Pi(\Mon)$. 
\end{enumerate}
\end{Lemma}
\begin{proof} 
(a) Let $(a,b)\in \Pi(\Mon)$ and let $p=\tp(a/A)$. First, suppose that $(a,b)\in\mathcal D_{\Pi}$. Then $a\dep_A b$, so there is a relatively $p$-bounded formula $\phi(x)\in\tp(a/Ab)$ in $(p(\Mon),<)$; that is, there are $a_1,a_2\models p$ such that $p(x)\cup \{\phi(x)\}\vdash a_1<x<a_2$. 
By compactness, there is $\theta(x)\in p$ such that 
$\theta(x)\land\phi(x)\vdash a_1<x<a_2$. Clearly, $\theta(x)\land \phi(x)\in \tp(a/Ab)$ is a $\Pi$-bounded formula, proving the left-to-right implication. To prove the other, it suffices to note that, since $\Pi(\Mon)$ is the convex hull of $p(\Mon)$, every $\Pi$-bounded formula $\phi(x)\in\tp(a/Ab)$ is also $p$-bounded, so $\phi(x)$ witnesses $a\dep_Ab$. 

(b) Let $\mathcal F$ be the $\delta_A$-class of $(p,<)$, where $p$ is a completion of $\Pi$; by Lemma \ref{Lemma_wqom_interval_type}(a), $(q,<)\in\mathcal F$ for all $q\in S_\Pi$. Let $<^{\mathcal F}$ be given by Lemma \ref{Lemma_wqom_interval_type}(b). By Theorem \ref{Theorem5}(b), $\mathcal D_{\mathcal F}=\{(a,b)\in \mathcal F(\Mon)^2\mid a\dep_A b\}$ is a $<^{\mathcal F}$-convex equivalence relation on $\mathcal F(\Mon)$. Since $\mathcal D_{\mathcal F\restriction \Pi(\Mon)}=\mathcal D_{\Pi}$ and $(<^{\mathcal F})$ is equal to $<$ or $>$ on $\Pi(\Mon)$, we conclude that $\mathcal D_{\Pi}$ is $<$-convex on $\Pi(\Mon)$. 
\end{proof}

\section{Order-trivial types}\label{Section_trivial}

The notion of triviality for types in stable theories was introduced and studied by Baldwin and Harrington in \cite{BH}: a stationary type $p\in S_x(A)$ is trivial if, for all $B\supseteq A$, any pairwise independent triple of realizations of the nonforking extension of $p$  is independent (as a set) over $B$. They showed that triviality of the type $p$ is preserved in all nonforking extensions and restrictions, suggesting that triviality is a property of the global nonforking extension of $p$, which is also a unique $A$-invariant, global extension of $p$. Therefore, a global type $\mathfrak p$ is trivial if for some (equivalently all) sets $A$ over which $\mathfrak p$ is invariant, any pairwise independent over $A$ set of realizations of $\mathfrak p_{\restriction A}$ is independent over $A$. Now, if we want to consider an $A$-invariant type $\mathfrak p$ in an arbitrary first-order theory, then it is natural to consider Morley sequences in $\mathfrak p$ over $A$ as independent. 

\begin{Definition}\label{Def_trivial_global_type}
A global type  $\mathfrak p\in S_x(\Mon)$  is {\em trivial over $A$} if it is non-algebraic, $A$-invariant and whenever the members of the sequence $I=(a_i\mid i\in \omega)$ are such that $(a_i,a_j)$ is a Morley sequence in $\mathfrak p$ over $A$ for all $i<j\in\omega$, then $I$ is a Morley sequence in $\mathfrak p$ over $A$. 
\end{Definition}

\begin{Remark}\phantomsection\label{Remark_basic_trivial}
\begin{enumerate}[(a),left=0em,labelsep=.7em]
\item  Let $T$ be stable and let $p\in S_x(A)$ be stationary and non-algebraic. Denote by $\mathfrak p$ the global nonforking extension of $p$. Then $p$ is trivial in the Baldwin-Harrington sense if and only if $\mathfrak p$ is trivial over $A$ in the sense of Definition \ref{Def_trivial_global_type}.

\item  If $T$ is a binary theory, then any non-algebraic, $A$-invariant global type is trivial over $A$. Examples of binary theories are: the theory of random graphs, theories of colored orders, etc.

\item  If a global type $\mathfrak p$ is trivial over $A$, then so is every finite power $\mathfrak p^n$.
\end{enumerate}
\end{Remark}

\begin{Remark}\label{Remark_trivial_basic2}
Let $\mathfrak p\in S_x(\Mon)$ be a non-algebraic, $A$-invariant type.
\begin{enumerate}[(a),left=0em,labelsep=.7em]
\item  An equivalent way to state that $\mathfrak p$ is trivial over $A$ is: for every $n\in\mathbb N$, the type \ $\bigcup_{i<j< n}(\mathfrak p^2)_{\restriction A}(x_i,x_j)$ \ implies a unique complete type over $A$ in variables $x_0,\dots,x_{n-1}$ (in which case this type must be $(\mathfrak p^n)_{\restriction A}(x_0,\dots,x_{n-1})$).

\item  Similarly as in \cite{BH}, by a {\em $\mathfrak p$-triangle}  over $A$ we will mean a triplet $(a_0,a_1,a_2)$ of realizations of $\mathfrak p_{\restriction A}$ which is not a Morley sequence in $\mathfrak p$ over $A$, but each pair $(a_i,a_j)$, $i<j<3$, is. When $\mathfrak p$ is not trivial over $A$, we can find a finite (possibly empty) Morley sequence $\bar a$ in $\mathfrak p$ over $A$ such that $\mathfrak p$ is not trivial over $A\bar a$, as witnessed by a triangle. To do this, choose minimal $n$ such that some sequence of size $n+3$, say $(a_0,a_1,\ldots, a_{n+2})$, realizes the type $\bigcup_{i<j< n+3}(\mathfrak p^2)_{\restriction A}(x_i,x_j)$ but does not realize $(\mathfrak p^n)_{\restriction A}(x_0,\dots,x_{n+2})$. Then put $\bar a=(a_0,\ldots,a_{n-1})$ and observe that $(a_n,a_{n+1},a_{n+2})$ is a $\mathfrak p$-triangle over $\bar aA$. 
\end{enumerate}
\end{Remark}

 Poizat in \cite{Goode} introduced triviality in the context of stable theories. A stable theory $T$ is trivial if, for all parameter sets $A$, any family of pairwise independent tuples over $A$ is independent over $A$. 
 It is easy to see that every global non-algebraic type in a trivial, stable theory is trivial over any parameter set over which it is based.  
The following fact is worth noting, but since we will not use it later, we leave the proof to the reader.

\begin{Fact} 
A complete theory $T$ is stable and trivial if and only if every global non-algebraic type is trivial (in the sense of Definition \ref{Def_trivial_global_type}) over some small set of parameters.      
\end{Fact}

Below, we will introduce order-triviality, a strong form of triviality for global types. The motivation for this comes from the weakly o-minimal case: we will prove in Lemma \ref{Lemma wom left trivial iff right trivial} that for every weakly o-minimal pair $(p,<)$ over $A$, the type $\mathbf p_r$ is trivial over $A$ if and only if it is order-trivial over $A$.

\begin{Definition}
A global type  $\mathfrak p\in S_x(\Mon)$  is {\em order-trivial over $A$} if it is non-algebraic, $A$-invariant and whenever the sequence $I=(a_i\mid i\in \omega)$ is such that $(a_i,a_{i+1})$ is a Morley sequence in $\mathfrak p$ over $A$ for all $i\in\omega$, then $I$ is a Morley sequence in $\mathfrak p$ over $A$. 
\end{Definition}

\begin{Remark}\phantomsection\label{Remark_basic_ordertrivial}
\begin{enumerate}[(a),left=0em,labelsep=.7em]
\item  A global $A$-invariant type $\mathfrak p\in S_x(\Mon)$ is order-trivial over $A$ if and only if 
for every $n\in\mathbb N$, the type $\bigcup_{i< n}(\mathfrak p^2)_{\restriction A}(x_i,x_{i+1})$ implies a unique complete type over $A$ in variables $x_0,\dots,x_{n}$ (in which case this type must be $(\mathfrak p^{n+1})_{\restriction A}(x_0,\dots,x_{n})$).

\item  If $\mathfrak p$ is order-trivial over $A$, then $\mathfrak p$ is trivial over $A$ and every power $\mathfrak p^n$ is order-trivial over $A$.

\item  A non-algebraic symmetric global invariant type $\mathfrak p$ is {\it not} order-trivial over a small set of parameters (where $\mathfrak p$ is symmetric if $\mathfrak p^2(x,y)=\mathfrak p^2(y,x)$). 
Indeed, if $\mathfrak p$ is symmetric and $A$-invariant, take $(a_0,a_1)\models \mathfrak p^2_{\restriction A}$ and consider the sequence $a_0,a_1,a_0,a_1,\dots$; clearly, this is not a Morley sequence in $\mathfrak p$ over $A$, although each consecutive pair is. In particular, there are no order-trivial types in stable theories nor in the theory of random graphs. 
\end{enumerate}
\end{Remark}

All examples of indiscernible sequences of order-trivial types that we know of are strictly increasing with respect to some definable partial order; yet, it remains uncertain if this is always the case.

\begin{Question}
If $\mathfrak p(x)\in S_x(\Mon)$ is order-trivial over $A$, must there exist a $B\supseteq A$ and a $B$-definable partial order such that all Morley sequences in $\mathfrak p$ over $B$ are strictly increasing?    
\end{Question}

\begin{Proposition}\label{Prop_basic_order_trivial}
Suppose that $\mathfrak p\in S(\Mon)$ is order-trivial over $A$ and that the type $p=\mathfrak p_{\restriction A}$ is NIP. Let $I=(a_i\mid i\in \omega)$ be a Morley sequence in $\mathfrak p$ over $A$ and let $B\supseteq A$. Then $I$ is a Morley sequence in $\mathfrak p$ over $B$ if and only if $a_0\models \mathfrak p_{\restriction B}$.
\end{Proposition}
\begin{proof} The left-to-right implication is clear. For the converse, suppose on the contrary that $a_0\models \mathfrak p_{\restriction B}$ but there is $n\in\omega$ such that $\bar a=(a_0,\dots,a_{n-1})$ is not Morley in $\mathfrak p$ over $B$. Choose a formula $\phi(\bar x,b)\in \tp(\bar a/B)$ such that $\phi(\bar x,b)\notin \mathfrak p^n$. Define the sequence $J=(\bar c^i\mid i\in\omega)$, where $\bar c^i=(c_0^i,\dots,c_{n-1}^i)\models (\mathfrak p^n)_{\restriction A}$, in the following way: Let $\bar c^0=\bar a$. Suppose that $J_{<k}=(\bar c^i\mid i<k)$ has already been defined. First choose $c_0^k\models \mathfrak p_{\restriction Bc_{n-1}^{k-1}}$ and then $\bar c^k=(c_0^k,\dots,c_{n-1}^k)$ such that:
\begin{equation}\tag{$\dagger$}\label{eq prop 2.9}
\bar c^{k}\equiv \bar a\,(B)\mbox{ if $k$ is even \  \ and \ \ }\bar c^{k}\models \mathfrak p^n_{\restriction B}\mbox{ if $k$ is odd}.
\end{equation}
Note that this is possible as $c_0^k,a_0\models\mathfrak p_{\restriction B}$.
Consider the sequence: $$c_0^0,\dots,c_{n-1}^0,c_0^1,\dots,c_{n-1}^1,\dots, c_0^k,\dots,c_{n-1}^k,\dots.$$
Note that, by our choice of $c_0^k$ and $\bar c^k$,  any pair of consecutive elements of this sequence
realizes the type $(\mathfrak p^2)_{\restriction A}$, so, by order-triviality, the sequence is Morley in $\mathfrak p$ over $A$. It follows that $J$ is a Morley sequence in $\mathfrak p^n$ over $A$. 
By our choice of $\phi(\bar x,b)$, condition (\ref{eq prop 2.9}) implies: $\models \phi(\bar c^k,b)$ if and only if $k$ is even. Therefore, the indiscernible sequence $J$ and the formula $\phi(\bar x,b)$ show that the type $(\mathfrak p^n)_{\restriction A}$ is not NIP. Contradiction!
\end{proof}

\begin{Corollary}\label{Cor_otrivialoverA_otrivialoverB}
Suppose that $\mathfrak p\in S(\Mon)$ is order-trivial over $A$ and that the type $p=\mathfrak p_{\restriction A}$ is NIP. Then $\mathfrak p$ is order-trivial over any $B\supseteq A$.
\end{Corollary}
\begin{proof}
Let $I=(a_i\mid i\in\omega)$ be a sequence of realizations of $\mathfrak p_{\restriction B}$ such that $(a_i,a_{i+1})\models (\mathfrak p^2)_{\restriction B}$ for all $i\in\omega$. We need to show that $I$ is a Morley sequence in $\mathfrak p$ over $B$. Note that $(a_i,a_{i+1})\models(\mathfrak p^2)_{\restriction A}$ for each $i\in\omega$, so by order-triviality of $\mathfrak p$ over $A$, $I$ is a Morley sequence in $\mathfrak p$ over $A$. Since $a_0$ realizes $\mathfrak p_{\restriction B}$, by Proposition \ref{Prop_basic_order_trivial}, $I$ is a Morley sequence in $\mathfrak p$ over $B$. Therefore, $\mathfrak p$ is order-trivial over $B$. 
\end{proof}

\begin{Corollary}\label{Cor wom lifts to omega for o trivial}
Let $\mathfrak p,\mathfrak q\in S(\Mon)$ be order-trivial over $A$ such that $p=\mathfrak p_{\restriction A}$ and $q=\mathfrak q_{\restriction A}$ are NIP types. 
\begin{enumerate}[(a),left=0em,labelsep=.7em]
\item If $r\in S(A)$ and $p\wor r$, then $\mathfrak (\mathfrak p^{\omega})_{\restriction A}\wor r$.
\item  If $p\wor q$ then $(\mathfrak p^{\omega})_{\restriction A}\wor (\mathfrak q^{\omega})_{\restriction A}$.
\end{enumerate}
\end{Corollary}
\begin{proof}
(a) Let $c\models r$, $\bar a=(a_0,a_1,\dots)\models(\mathfrak p^\omega)_{\restriction Ac}$ and $\bar a'=(a_0',a_1',\dots)\models(\mathfrak p^\omega)_{\restriction A}$; we need to prove $\bar a\equiv\bar a'\,(Ac)$. Note that $p\wor r$ implies $a_0\equiv a_0'\,(Ac)$, so $a_0'\models \mathfrak p_{\restriction Ac}$, which together with $\bar a'\models(\mathfrak p^\omega)_{\restriction A}$, by Proposition \ref{Prop_basic_order_trivial}, implies $\bar a'\models(\mathfrak p^\omega)_{\restriction Ac}$. Therefore, $\bar a\equiv\bar a'\,(Ac)$, as desired. 

(b) By (a), $(\mathfrak p^\omega)_{\restriction A}\wor q$, so again by (a), $(\mathfrak p^\omega)_{\restriction A}\wor (\mathfrak q^\omega)_{\restriction A}$.
\end{proof}

\section{Trivial weakly o-minimal types}\label{Section_trivial_wom}

\begin{Lemma}\label{Lemma wom left trivial iff right trivial}
Let $\mathbf p=(p,<)$ be a weakly o-minimal pair over $A$. The following are equivalent:   
\begin{enumerate}[(1),left=0em,labelsep=.7em]
\item $\mathbf p_r$ is trivial over $A$;  
\item Every $\triangleleft^{\mathbf p}$-increasing sequence of realizations of $p$ is Morley in $\mathbf p_r$ over $A$; 
\item $\mathbf p_r$ is order-trivial over $A$;
\item Every $\triangleleft^{\mathbf p}$-decreasing sequence of realizations of $p$ is Morley in $\mathbf p_l$ over $A$; 
\item $\mathbf p_l$ is order-trivial over $A$.
\end{enumerate}
\end{Lemma}
\begin{proof}
Recall that $(p(\Mon), \triangleleft^{\mathbf p})$ is a strict partial order and that, by Remark \ref{Remark_trianglep_binary_basicprops}(b),  $a\triangleleft^{\mathbf p}b$, $(a,b)\models (\mathbf p_r^2)_{\restriction A}$, and $(b,a)\models (\mathbf p_l^2)_{\restriction A}$ are mutually equivalent for all $a,b\models p$; in particular, $(\mathbf p_l^2)_{\restriction A}(x,y)=(\mathbf p_r^2)_{\restriction A}(y,x)$.

(1)$\Rightarrow$(2) Suppose that $\mathbf p_r$ is trivial over $A$ and let $I=(a_i\mid i\in \omega)$ be a $\triangleleft^{\mathbf p}$-increasing sequence. Then for all $i<j<\omega$ we have $a_i\triangleleft^{\mathbf p} a_j$; hence, $(a_i,a_j)\models (\mathbf p_r^2)_{\restriction A}$. Due to the triviality of $\mathbf p_r$, $I$ is a Morley sequence in $\mathbf p_r$ over $A$. 

(2)$\Leftrightarrow$(3) Every $\triangleleft^\mathbf p$-increasing sequence is Morley if and only if:

\begin{enumerate}[label=(r),left=0em,labelsep=1em]
\item For all $n\in \mathbb N$ the type  $\bigcup_{i<n}(\mathbf p_r^2)_{\restriction A}(x_i,x_{i+1})$ has a unique completion over $A$.
\end{enumerate}
By Remark \ref{Remark_basic_ordertrivial}, this is equivalent to the order-triviality of $\mathbf p_r$ over $A$. 

As (3)$\Rightarrow$(1) is immediate from the definitions, we conclude (1)$\Leftrightarrow$(2)$\Leftrightarrow$(3). 

(4)$\Leftrightarrow$(5) is similar to (2)$\Leftrightarrow$(3). 

(2)$\Leftrightarrow$(4) Every $\triangleleft^{\mathbf p}$-increasing sequence of realizations of $p$ is Morley in $\mathbf p_r$ over $A$ if and only if condition (r) holds. 
As  $(\mathbf p_l^2)_{\restriction A}(x,y)=(\mathbf p_r^2)_{\restriction A}(y,x)$,  condition (r) is equivalent with:
\begin{enumerate}[label=(l),left=0em,labelsep=1em]
\item For all $n\in \mathbb N$ the type  $\bigcup_{i<n}(\mathbf p_l^2)_{\restriction A}(x_i,x_{i+1})$ has a unique completion over $A$,
\end{enumerate}
which holds if and only if every $\triangleleft^{\mathbf p}$-decreasing sequence of realizations of $p$ is Morley in $\mathbf p_l$ over $A$.
\end{proof}

\begin{Definition}
A weakly o-minimal type $p\in S(A)$ is {\em trivial} if one (equivalently, both) of its $A$-invariant globalizations is (order-) trivial over $A$.
\end{Definition}

\noindent{\bf Convention.}  \ For $\mathbf p=(p,<)$ a weakly o-minimal pair over $A$, by $A\triangleleft^{\mathbf p}x_1\triangleleft^{\mathbf p} x_2\triangleleft^{\mathbf p} \dots \triangleleft^{\mathbf p}x_{n}$ we will denote the partial type $\bigcup_{1\leqslant i< n} (\mathbf p_r^2)_{\restriction A}(x_i,x_{i+1})$.

\begin{Remark}\label{Remark trivial wom expressing via triangleleft}
Let $\mathbf p=(p,<)$ be a weakly o-minimal pair over $A$.  As a consequence of Lemma \ref{Lemma wom left trivial iff right trivial} we have the following: 

\begin{enumerate}[(a),left=0em,labelsep=.7em]
\item $p$ is trivial if and only if $A\triangleleft^{\mathbf p}x_1\triangleleft^{\mathbf p} x_2\triangleleft^{\mathbf p} \dots \triangleleft^{\mathbf p}x_{n}$ determines a complete type for all $n\in\mathbb N$. 

\item If $p$ is trivial, then a sequence $I=(a_i\mid i< \omega)$ of realizations of $p$ is Morley in a nonforking globalization of $p$ over $A$ if and only if $I$ is strictly $\triangleleft^{\mathbf p}$-monotone; equivalently, if the sequence $(\mathcal D_p(a_i)\mid i<\omega)$ is strictly $<$-monotone. 
\end{enumerate}
\end{Remark}

\begin{Lemma}\label{Lemma_triv_preserved_in_nonforking}
Let $\mathbf p=(p,<)$ be a weakly o-minimal pair over $A$. Suppose that $p$ is trivial.
\begin{enumerate}[(a),left=0em,labelsep=.7em]
 \item  Every nonforking extension of $p$ is trivial.
 \item The sequence $I=(a_i\mid i< \omega)$ of realizations of $p$ is a Morley sequence in $\mathbf p_r$ over $B\supseteq A$ if and only if \ $B\triangleleft^{\mathbf p} a_0\triangleleft^{\mathbf p}a_1\triangleleft^{\mathbf p} a_2\triangleleft^{\mathbf p}\dots$.
\end{enumerate}
\end{Lemma}
\begin{proof}(a) Let $q\in S(B)$ be a nonforking extension of $p$. 
Then $q=(\mathbf p_r)_{\restriction B}$ or $q=(\mathbf p_l)_{\restriction B}$. 
Without loss assume $q=(\mathbf p_r)_{\restriction B}$; then $\mathbf q_r=\mathbf p_r$ where $\mathbf q=(q,<)$. Since $p$ is trivial, according to Lemma \ref{Lemma wom left trivial iff right trivial} $\mathbf p_r$ is order-trivial over $A$, so by Corollary \ref{Cor_otrivialoverA_otrivialoverB}, $\mathbf p_r$ is order-trivial over $B$. Hence $\mathbf q_r$ is trivial, so $q$ is trivial by Lemma \ref{Lemma wom left trivial iff right trivial}. 

(b) First, assume \ $B\triangleleft^{\mathbf p} a_0\triangleleft^{\mathbf p}a_1\triangleleft^{\mathbf p} a_2\triangleleft^{\mathbf p}\dots$. Since $I$ is strictly $\triangleleft^{\mathbf p}$-increasing, by Remark \ref{Remark trivial wom expressing via triangleleft}, it is a Morley sequence in $\mathbf p_r$ over $A$. Since $a_0\models (\mathbf p_r)_{\restriction B}$, Proposition \ref{Prop_basic_order_trivial} applies: $I$ is a Morley sequence in $\mathbf p_r$ over $B$. This proves one direction of the equivalence. To prove the other, assume $I$ is Morley in $\mathbf p_r$ over $B$. Then $a_0\models (\mathbf p_r)_{\restriction B}$ implies $B\triangleleft^{\mathbf p} a_0$. Since $I$ is Morley over $A$,   $a_0\triangleleft^{\mathbf p}a_1\triangleleft^{\mathbf p} a_2\triangleleft^{\mathbf p}\dots$ follows. Therefore, \ $B\triangleleft^{\mathbf p}a_0\triangleleft^{\mathbf p}a_1\triangleleft^{\mathbf p} a_2\triangleleft^{\mathbf p}\dots$ 
\end{proof}

In the following lemma, we establish an important property of trivial types, which will be referred to as the strong transitivity property later in the text.  
 
\begin{Lemma}[Strong transitivity]\label{Lemma_trivial_strong_trans}
Suppose that $\mathbf p$ and $\mathbf q$ are directly nonorthogonal weakly o-minimal pairs over $A$ and that $p$ is trivial. Then $B\triangleleft^{\mathbf p}a\triangleleft^{\mathbf q} b$ implies $Ba\triangleleft^{\mathbf q} b$ for all $B\supseteq A$. 
\end{Lemma}
\begin{proof}
Assume $B\triangleleft^{\mathbf p}a\triangleleft^{\mathbf q}b$. Choose $a'$ with $a\triangleleft^{\mathbf p}a'\triangleleft^{\mathbf q}b$; that is possible by the density property (Corollary \ref {Cor_triangle_mathbf_pq_properties}(b)). By Lemma \ref{Lemma_triv_preserved_in_nonforking}(b), $B\triangleleft^{\mathbf  p}a\triangleleft^{\mathbf  p}a'$ implies that $(a,a')$ is a Morley sequence in $\mathbf p_r$ over $B$, so $Ba\triangleleft^{\mathbf  p} a'$ holds. By the transitivity property  (Corollary \ref {Cor_triangle_mathbf_pq_properties}(f))  $Ba\triangleleft^{\mathbf  p}a'\triangleleft^{\mathbf q}b$ implies  $Ba\triangleleft^{\mathbf q}b$.
\end{proof}

\begin{Lemma}\label{Lemma_trivial_iff_Dp=Dq}
For a weakly o-minimal type $p\in S(A)$ the following are equivalent:
\begin{enumerate}[(1),left=0em,labelsep=.7em]
    \item $p$ is trivial;
    \item $\mathcal D_q(a)=\mathcal D_p(a)$ for all nonforking extensions $q$ of $p$ and all $a\models q$;
    \item $\mathcal D_q(a)=\mathcal D_p(a)$ for all right (with respect to a fixed order) nonforking extensions $q$ of $p$ and all $a\models q$.
\end{enumerate}
\end{Lemma}
\begin{proof}
(1)$\Rightarrow$(2) Let $\mathbf p=(p,<)$ be a weakly o-minimal pair over $A$, and assume that $p$ is trivial. Let $q\in S(B)$ be a nonforking extension of $p$ and let $a\models q$; we need to prove $\mathcal D_p(a)=\mathcal D_q(a)$.  As $q$ is a nonforking extension of $p$, $a$ is left or right $\mathbf p$-generic over $B$; without loss, assume  $B\triangleleft^{\mathbf p} a$. 
By Theorem \ref{Theorem_triangle_mathcal F}(g) we know that for all $b\models p$, $B\triangleleft^{\mathbf p}a\land a\dep_A b$ implies $B\triangleleft^{\mathbf p}b\land a\dep_Bb$; hence,
$\mathcal D_p(a)\subseteq \mathcal D_q(a)$. If this inclusion were proper, then there would be some $a'\in\mathcal D_q(a)$ with $a\ind_A a'$; in particular, $a'\models q$ and thus $B\triangleleft^{\mathbf p}a'$. By the $\triangleleft^{\mathbf p}$-comparability property, Corollary \ref {Cor_triangle_mathbf_pq_properties}(e), $a\ind_A a'$ implies $a\triangleleft^{\mathbf p}a'$ or $a'\triangleleft^{\mathbf p}a$. In the first case, we have $B\triangleleft^{\mathbf p}a\triangleleft^{\mathbf p}a'$, which,
by  Lemma \ref{Lemma_triv_preserved_in_nonforking}(b), implies that $(a,a')$ is Morley over $B$ and, in particular,  $a'\ind_B a$; this contradicts $a'\in \mathcal D_q(a)$. The second case is dealt with similarly.  Therefore, $\mathcal D_p(a)=\mathcal D_q(a)$, as desired. 

(2)$\Rightarrow$(3) is obvious, so we prove (3)$\Rightarrow$(1). 
Suppose that $\mathbf p=(p,<)$ is a weakly o-minimal pair over $A$ such that $\mathcal D_p(a)=\mathcal D_q(a)$ holds for all right nonforking extensions $q$ and all $a\models q$. 
We will prove that every $\triangleleft^{\mathbf p}$-increasing sequence $I=(a_i\mid i\in \omega)$ of realizations of $p$ is Morley in $\mathbf p_r$ over $A$; by Lemma \ref{Lemma wom left trivial iff right trivial}, this implies that $p$ is trivial.
By induction on $n>0$, we prove that $(a_0,a_1,\dots,a_n)$ is Morley (in $\mathbf p_r$ over $A$). The case $n=1$ is clear, so assume that $(a_0,a_1,\dots,a_n)$ is Morley for some $n>0$; in particular, $\bar a_{<n}\triangleleft^{\mathbf p} a_n\triangleleft^{\mathbf p}a_{n+1}\triangleleft^{\mathbf p}\ldots$.
Put $q=(\mathbf p_r)_{\restriction A\bar a_{<n}}$ and $\mathbf q=(q,<)$. Then $q$ is a right nonforking extension of $p$, $\mathbf p_r=\mathbf q_r$, and $a_n,a_{n+1}\models q$. By the  assumption, we have $\mathcal D_p(a_n)=\mathcal D_q(a_{n})$ and $\mathcal D_p(a_{n+1})=\mathcal D_q(a_{n+1})$. As $I$ is $\triangleleft^{\mathbf p}$-increasing, we get $\mathcal D_p(a_n)<\mathcal D_p(a_{n+1})$; hence $\mathcal D_q(a_n)<\mathcal D_q(a_{n+1})$; the latter implies that $(a_n,a_{n+1})$ is Morley in $\mathbf q_r$ over $A\bar a_{<n}$, so $(a_0,\ldots,a_n,a_{n+1})$ is Morley in $\mathbf p_r$ over $A$, finishing the proof. 
\end{proof}

\begin{Lemma}\label{Lemma trivial wom morley in nonforking is left or right morley}
Let $\mathbf p=(p,<)$ be a weakly o-minimal pair over $A$ such that $p$ is trivial, and let $q\in S(B)$ be a nonforking extension of $p$. Suppose that $\mathfrak q$ is a global nonforking extension of $q$. Then every
Morley sequence in $\mathfrak q$ over $B$ is a Morley sequence in $\mathbf p_r$ or $\mathbf p_l$ over $A$.
\end{Lemma}
\begin{proof} First, note that $\mathfrak q\notin \{\mathbf p_r,\mathbf p_l\}$ is possible: For example, let $a\triangleleft^{\mathbf p} b$ realize $p$, let $q=\tp(b/a)$, and let $\mathfrak q$ be the left globalization of $(q,<)$. 

By Lemma \ref{Lemma_triv_preserved_in_nonforking}, the triviality of $p$ is preserved in nonforking extensions, so $q$ is trivial. 
Suppose that $I=(a_i\mid i\in\omega)$ is a Morley sequence in $\mathfrak q$ over $B$. By Remark \ref{Remark trivial wom expressing via triangleleft}, the sequence $(\mathcal D_q(a_i)\mid i\in\omega)$ is strictly $<$-monotone, so by
Lemma \ref{Lemma_trivial_iff_Dp=Dq} the sequence $(\mathcal D_p(a_i)\mid i\in\omega)$
is also strictly $<$-monotone; by Remark \ref{Remark trivial wom expressing via triangleleft} again, $I$ is a Morley sequence in $\mathbf p_r$ or $\mathbf p_l$ over $A$.
\end{proof}

\begin{Remark}\label{Remark_notrivial_wom_witness}Let $\mathbf p=(p,<)$ be a weakly o-minimal pair over $A$. 
\begin{enumerate}[(a),left=0em,labelsep=.7em]
\item Recall that a sequence $(a_1,a_2,a_3)$ of realizations of $p$ is a $\mathbf p_r$-triangle over $A$ if it is not Morley in $\mathbf p_r$ over $A$, but every pair $(a_i,a_j) \ (i<j)$ is so; equivalently, $A\triangleleft^{\mathbf p}a_1 \triangleleft^{\mathbf p}a_2\triangleleft^{\mathbf p} a_3$ and $a_2\dep_{Aa_1}a_3$ (since by Theorem \ref{Theorem_triangle_mathcal F}(i), $a_1 \triangleleft^{\mathbf p}a_2\triangleleft^{\mathbf p} a_3$ and $a_2\ind _{Aa_1} a_3$ together imply $a_1a_2\triangleleft^{\mathbf p}a_3$).
The latter can happen if and only if the partial type $A\triangleleft^{\mathbf p}x\triangleleft^{\mathbf p} y\triangleleft^{\mathbf p}z$ is incomplete, in which case, by symmetry, there exists a $\mathbf p_l$-triangle over $A$. In particular, $\mathbf p_r$-triangles and $\mathbf p_l$-triangles over $A$ simultaneously exist.
\item Let $\mathfrak p=(p,<_p)$ be another weakly o-minimal pair over $A$. Then $\triangleleft^{\mathfrak p}\in\{\triangleleft^{\mathbf p}, \triangleleft^{\mathbf p^*}\}$, so
$A\triangleleft^{\mathbf p}x\triangleleft^{\mathfrak p} y\triangleleft^{\mathfrak p}z$  equals $A\triangleleft^{\mathbf p}x\triangleleft^{\mathbf p} y\triangleleft^{\mathbf p}z$ or $A\triangleleft^{\mathbf p}z\triangleleft^{\mathbf p} y\triangleleft^{\mathbf p}x$. By (a), we conclude that $\mathfrak p_r$-triangles and $\mathbf p_r$-triangles exist simultaneously (for all $\mathfrak p=(p,<_p)$).
\item By Remark \ref{Remark_trivial_basic2} the non-triviality of $\mathbf p_r$ over $A$ can be witnessed by a finite (possibly empty) Morley sequence $\bar b$ in $\mathbf p_r$ over $A$ and elements $a_1,a_2,a_3$ that form the $\mathbf p_r$-triangle over $A\bar b$. 
This can be expressed by: \ $\bar b\triangleleft^{\mathbf p}  a_1 \triangleleft^{\mathbf p_{\bar b}}a_2\triangleleft^{\mathbf p_{\bar b}} a_3 \ \ \ \mbox{and} \ \ \  a_2\dep_{A\bar ba_1} a_3$ \ (where $\mathbf p_{\bar b}=((\mathbf p_r)_{\restriction A\bar b},<)$).
\end{enumerate}
\end{Remark} 

\begin{Lemma}\label{Lemma_3trivial}
The following conditions are equivalent for any weakly o-minimal type $p\in S(A)$. 
\begin{enumerate}[(1),left=0em,labelsep=.7em]
    \item The type $A\triangleleft^{\mathbf p}x\triangleleft^{\mathbf p} y\triangleleft^{\mathbf p} z$ has a unique completion over $A$;
    \item There are no $\mathbf p_r$-triangles over $A$ for some/all weakly o-minimal pairs $\mathbf p=(p,<)$ over $A$;
    \item $(\mathbf p_r)_{\restriction Aa}\wor (\mathbf p_l)_{\restriction Aa}$ for all $a\models p$ and some/all weakly o-minimal pairs $\mathbf p=(p,<)$ over $A$;
    \item $(\mathbf p_r^3)_{\restriction A}(x,y,z)=(\mathbf p_l^3)_{\restriction A}(z,y,x)$ for some/all weakly o-minimal pairs $\mathbf p=(p,<)$ over $A$.
\end{enumerate}
\end{Lemma}
\begin{proof}
(1)$\Leftrightarrow$(2) follows by Remark \ref{Remark_notrivial_wom_witness}.  

(1) $\Leftrightarrow$(3) Let $\mathbf p=(p,<)$ be a weakly o-minimal pair over $A$, and let $a\models p$. The type $A\triangleleft^{\mathbf p}x\triangleleft^{\mathbf p}y\triangleleft^{\mathbf p}z$ has at least two completions in $S_{xyz}(A)$ if and only if the type $A\triangleleft^{\mathbf p}x\triangleleft^{\mathbf p}a\triangleleft^{\mathbf p}z$ has at least two completions in $S_{xz}(Aa)$. The latter is equivalent to $(\mathbf p_l)_{\restriction Aa}\nwor (\mathbf p_r)_{\restriction Aa}$, because $p(x)\cup \{x\triangleleft^{\mathbf p}a\}$ is equivalent to $(\mathbf p_l)_{\restriction Aa}(x)$ and $p(z)\cup \{a\triangleleft^{\mathbf p}z\}$ to $(\mathbf p_r)_{\restriction Aa}(z)$. 

(1)$\Rightarrow$(4) is immediate, so it remains to prove $\lnot{\text{(1)}}\Rightarrow \lnot{\text{(4)}}$. Assume that $A\triangleleft^{\mathbf p}x\triangleleft^{\mathbf p} y\triangleleft^{\mathbf p} z$ has at least two completions. 
Let $(a,b,c)$ be a Morley sequence in $\mathbf p_r$ over $A$, and let $(c,b',a)$ be a Morley sequence in $\mathbf p_l$ over $A$; 
it suffices to prove $b\nequiv b'\,(Aac)$. Denote $p_a=(\mathbf p_r)_{\restriction Aa}$ and $\mathbf p_a=(p_a,<)$, and consider $\Pi(x)=a\triangleleft^{\mathbf p}x\triangleleft^{\mathbf p}c$ as a partial type over $Aac$. By Remark \ref{Remark_pgeneric_first}(c), the set $p_a(\Mon)$, which is defined by $a\triangleleft^{\mathbf p}x$, is a final part of $(p(\Mon),<)$. Similarly,  $x\triangleleft^{\mathbf p}c$ defines an initial part of $(p(\Mon),<)$, so $\Pi(\Mon)$ is a convex subset of $(p(\Mon),<)$. 
If $q\in S_x(Aac)$ is a completion of $\Pi(x)$, then $q$ is an extension of $p$, so by the weak o-minimality of $p$, $q(\Mon)$ is convex within $(p(\Mon),<)$ and therefore also within $(\Pi(\Mon),<)$; since $\Pi(x)$ is incomplete by $\lnot(1)$, $q(\Mon)$ is a proper subset of $\Pi(\Mon)$.
Now, we claim that the locus of $\tp(b/Aac)$, call it $P$, is an initial part of $\Pi(\Mon)$.  
To prove this, note that the sequence $(b,c)$ is Morley in $\mathbf p_r$ over $Aa$, so $b\triangleleft^{\mathbf p_a}c$, that is, $b$ is left $\mathbf p_a$-generic over $c$. We conclude that $P$ consists of all left $(\mathbf p_a)_l$-generic elements over $c$, so, by Remark \ref{Remark_pgeneric_first}(c), $P$ is an initial part of $p_a(\Mon)$. Combining this with the facts that $\Pi(\Mon)$ is an initial part of $p_a(\Mon)$ and that $P\subseteq \Pi(\Mon)$, we conclude that $P$ is an initial part of $\Pi(\Mon)$.
Analogously, we see that the locus of $\tp(b'/Aac)$ is a proper final part of $\Pi(\Mon)$. Therefore, $b\not\equiv b'\,(Aac)$.
\end{proof}

\begin{Theorem}\label{Theorem_nwor preserves triviality}
\begin{enumerate}[(a),left=0em,labelsep=.7em]
\item Nonforking extensions of a trivial weakly o-minimal type are all trivial, but the triviality of some nonforking extension does not guarantee the triviality of the type.  

\item  The triviality is preserved under $\nwor$ of weakly o-minimal types over the same domain.

\item  The weak orthogonality of trivial weakly o-minimal types transfers to their nonforking extensions. 
\end{enumerate}
\end{Theorem}
\begin{proof}(a) The first clause is Lemma \ref{Lemma_triv_preserved_in_nonforking}(a); the second will be justified in Example \ref{Example_trivover0}. 

(b) Let $p$ and $q$ be weakly o-minimal types over $A$ with $p\nwor q$. Assuming that $p$ is non-trivial, we will prove that $q$ is also non-trivial. Choose relatively $A$-definable orders $<_p$ and $<_q$ such that the weakly o-minimal pairs $\mathbf p=(p,<_p)$ and $\mathbf q=(q,<_q)$ are directly nonorthogonal. 
By Remark \ref{Remark_notrivial_wom_witness}(c) there are $D\supseteq A$ and $a_1,a_2,a_3$ such that $D\triangleleft^{\mathbf p}a_1\triangleleft^{\mathbf p_D}a_2\triangleleft^{\mathbf p_D}a_3$ and $a_2\dep_{Da_1} a_3$, where $\mathbf p_D=((\mathbf p_r)_{\strok D},<)$; note that in particular $Da_1\triangleleft^{\mathbf p}a_2\triangleleft^{\mathbf p}a_3$ holds.
By  density  (Corollary \ref {Cor_triangle_mathbf_pq_properties}(b)) there are $b,b'$ satisfying $a_2\triangleleft^{\mathbf q}b\triangleleft^{\mathbf q}b'\triangleleft^{\mathbf p}a_3$. By Theorem \ref{Theorem_triangle_mathcal F}(j),
$Da_1\triangleleft^{\mathbf p}a_2\triangleleft^{\mathbf q} b\triangleleft^{\mathbf q}b'\triangleleft^{\mathbf p} a_3$ and $a_2\dep_{Da_1} a_3$ together imply
$b\dep_{Da_1} b'$. 
If $q$ were trivial, then by  strong transitivity  (Lemma \ref{Lemma_trivial_strong_trans}) $Da_1\triangleleft^{\mathbf q}b\triangleleft^{\mathbf q}b'$ would imply $Da_1b\triangleleft^{\mathbf q}b'$, which contradicts $b\dep_{Da_1} b'$. Therefore, $q$ is non-trivial, proving (b).

(c)  Suppose that $p,q\in S(A)$ are trivial weakly o-minimal types with $p\wor q$ and that $p_1,q_1\in S(B)$ are their nonforking extensions. We need to prove $p_1\wor q_1$. 
Suppose on the contrary that $p_1\nwor q_1$. By part (a) of the theorem, both $p_1$ and $q_1$ are trivial, since $p$ and $q$ are. 
Choose relatively $A$-definable orders $<_p$ and $<_q$ such that the weakly o-minimal pairs $\mathbf p_1=(p_1,<_p)$ and $\mathbf q_1=(q_1,<_q)$ are directly nonorthogonal. Let $\mathbf p=(p,<_p)$ and $\mathbf q=(q,<_q)$.
Choose $a_0,a_1\models p_1$ and $b_0,b_1\models q_1$ such that \ $a_0\triangleleft^{\mathbf q_1}b_0\triangleleft^{\mathbf q_1}b_1\triangleleft^{\mathbf p_1}a_1$. Then \ $\tp(a_0b_0/B)\neq \tp(a_1b_1/B)$. To see this, first observe that $p_1\nwor q_1$ implies $\tp(a_0b_0/B)\neq \tp(a_1b_0/B)$ as $a_0$ is left and $a_1$ is right $\mathbf p_1$-generic over $b_0$; also observe that 
$b_0\triangleleft^{\mathbf q_1}b_1\triangleleft^{\mathbf p_1}a_1$ implies $\tp(a_1b_0/B)=\tp(a_1b_1/B)$; therefore, $\tp(a_0b_0/B)\neq \tp(a_1b_1/B)$. 
Witness this by a formula $\phi(x,y,c)\in \tp(a_0b_0/B)$ such that $\phi(x,y,c)\notin \tp(a_1b_1/B)$. 
Define sequences $I=(a_n\mid n<\omega)$ and $J=(b_n\mid n<\omega)$ such that the following holds for all $n\geqslant 1$:
$$a_{2n-1}\triangleleft^{\mathbf p_1}a_{2n}\ \ \ \ \mbox{ and }\ \ \ \ a_0b_0a_1b_1\equiv a_{2n}b_{2n}a_{2n+1}b_{2n+1}\, (B).$$
Note that for all $n<\omega$, $a_{2n}\triangleleft^{\mathbf q_1}b_{2n}\triangleleft^{\mathbf q_1}b_{2n+1}\triangleleft^{\mathbf p_1}a_{2n+1}\triangleleft^{\mathbf p_1}a_{2n+2}$ holds.
Hence, $I$ is strictly $\triangleleft^{\mathbf p_1}$-increasing, so, by Remark \ref{Remark trivial wom expressing via triangleleft}, $I$ is a Morley sequence in $(\mathbf p_1)_r$ over $B$. Similarly, $J$ is a Morley sequence in $(\mathbf q_1)_r$ over $B$. By Lemma \ref{Lemma trivial wom morley in nonforking is left or right morley}, $I$ is a Morley sequence in a nonforking globalization of $p$ over $A$; as it is $<_p$-increasing, $I$ is a Morley sequence in $\mathbf p_r$ over $A$. Similarly, $J$ is a Morley sequence in $\mathbf q_r$ over $A$. By Corollary \ref{Cor wom lifts to omega for o trivial}(b),  $\tp(I/A)\wor \tp(J/A)$ as $p\wor q$. In particular, $I$ and $J$ are mutually indiscernible over $A$ and the sequence $(a_nb_n\mid n<\omega)$ is indiscernible over $A$. Since $a_0b_0a_1b_1\equiv a_{2n}b_{2n}a_{2n+1}b_{2n+1}\,(B)$ implies $\models\phi(a_{2n},b_{2n},c)\land\lnot \phi(a_{2n+1},b_{2n+1},c)$, the type $\tp(a_0b_0/A)$ is not NIP. Contradiction.
\end{proof}

For directly non-orthogonal weakly o-minimal pairs we can strengthen part (c) of the theorem: the transfer of weak and forking non-orthogonality occurs in both directions, from and towards non-forking extensions.

\begin{Corollary}
Suppose that $\mathbf p=(p,<_p)$ and $\mathbf q=(q,<_q)$ are weakly o-minimal pairs over $A$ such that $p$ and $q$ are trivial. Let $B\supseteq A$, $p_B=(\mathbf p_r)_{\restriction B}$, and $q_B=(\mathbf q_r)_{\restriction B}$.
\begin{enumerate}[(a),left=0em,labelsep=.7em]
\item  $p\wor q \Rightarrow p_B\wor q_B$ \ and  \ $\delta_A(\mathbf p,\mathbf q) \Rightarrow p_B\nwor q_B$.

\item  $\delta_A(\mathbf p,\mathbf q)$ implies  \ $p\fwor q \Leftrightarrow p_B\fwor q_B$. 
\end{enumerate}
\end{Corollary}
\begin{proof}
(a)  The first claim is Theorem \ref{Theorem_nwor preserves triviality}(c) and the second follows by Fact \ref{Fact_delta_implies_deltaB_fworB}(a). 

(b) Assume $\delta_A(\mathbf p,\mathbf q)$. By Fact \ref{Fact_delta_implies_deltaB_fworB}(b) we have $p\nfor q \Rightarrow p_B\nfor q_B$. 
To prove the other direction, assume $p\fwor q$ and let $a\models p_B$ and $b\models q_B$; we need to prove $a\ind_B b$. $p\fwor q$ implies $a\ind_A b$, 
so $a\triangleleft^{\mathbf q} b$ or $b\triangleleft^{\mathbf p} a$. Without loss, assume $a\triangleleft^{\mathbf q} b$. By  strong transitivity (Lemma \ref{Lemma_trivial_strong_trans}), $B\triangleleft^{\mathbf p}a\triangleleft^{\mathbf q}b$ implies $Ba\triangleleft^{\mathbf q}b$; hence, $b\ind_B a$, as desired.  
\end{proof}

We have already shown that every $\triangleleft^{\mathbf p}$-increasing sequence of realizations of a trivial type is Morley in $\mathbf p_r$. This generalizes to realizations of trivial types from the same $\delta_A$-class in the following way: 

\begin{Definition}
Let $\mathcal F$ be a $\delta_A$-class of weakly o-minimal pairs over $A$ and let $(I,<_I)$ be a linear order. A sequence $(a_i\mid i\in I)$  of realizations of types from $\mathcal F$ is {\em $\mathcal F$-independent over $A$} if $(a_{j}\mid j<i)\triangleleft^{\mathcal F} a_i$ holds for all $i\in I$. 
\end{Definition}

\begin{Proposition}
Let $\mathcal F$ be a $\delta_A$-class such that the types appearing in $\mathcal F$ are trivial. Then every $\triangleleft^{\mathcal F}$-increasing sequence of realizations of types from $\mathcal F$ is $\mathcal F$-independent over $A$. 
\end{Proposition} 
\begin{proof}
It suffices to prove the claim for finite sequences. We proceed by induction on the length of a sequence. The claim is trivial for sequences of length $\leq 2$, so suppose that we have a $\triangleleft^{\mathcal F}$-increasing sequence $(a_1,\dots,a_n,a_{n+1})$ for $n\geqslant 2$. By the induction hypothesis, $(a_j\mid j<i)\triangleleft^{\mathcal F}a_i$ holds for all $i\leqslant n$. In particular, we have $(a_j\mid j<n)\triangleleft^{\mathcal{F}} a_n\triangleleft^{\mathcal{F}}a_{n+1}$. By strong transitivity (Lemma \ref{Lemma_trivial_strong_trans}) we conclude $(a_j\mid j\leqslant n)\triangleleft^{\mathcal F}a_{n+1}$, and we are done.
\end{proof}

\subsection{Trivial 1-types in o-minimal theories}

Throughout this subsection, $T$ denotes an o-minimal theory (with respect to some dense linear order $<$). Recall that an element $a$ of an o-minimal structure is said to be non-trivial if there is an open interval $I$ containing $a$ and a definable function $f:I\times I\to \Mon$ that is strictly increasing in both coordinates; otherwise, $a$ is trivial. By the Trichotomy Theorem (\cite{PeS}), $a$ is non-trivial if and only if there is a definable group structure on some open, convex neighborhood of $a$. 
On the other hand, as
shown in \cite{MRS}, all elements of $\Mon$ are trivial if and only if $(\Mon,\dcl)$ is a degenerate pregeometry if and only if $T$ is binary. In fact, 
if $a$ is trivial, then it follows from Lemmas 2.1 and 2.2 of \cite{MRS} that $\tp(a)$ is trivial. In general, the converse is not true, as illustrated by the following example. 
 
\begin{Example}\phantomsection\label{Example_trivial_real_additive}
Consider the ordered group of reals $(\mathbb R,+,<,0)$; clearly, it is an o-minimal structure. There is a unique type $p\in S_1(T)$ that contains  $0<x$ and the pair $\mathbf p=(p,<)$ is weakly o-minimal. 
For all $A$, the type $(\mathbf p_{r})_{\restriction A}$ is determined by $Gp(A)<x$, where $Gp(A)$ is the subgroup generated by $A$. Using this, it is easy to see that $\mathbf p_r$ is order-trivial over $\emptyset$:
an increasing sequence of realizations of $p$, $(a_i\mid i\in\omega)$, is Morley in $\mathbf p_r$ over $\emptyset$ if and only if $Gp(a_i)<a_{i+1}$ for all $i\in \omega$, or equivalently,  if $a_i\triangleleft^{\mathbf p} a_{i+1}$  for all $i\in\omega$. Therefore, $\mathbf p_r$ is trivial over $\emptyset$, although the $\mathrm{dcl}$-pregeometry on $p(\Mon)$ is obviously not degenerate.
\end{Example}
 
Therefore, the triviality of the type is not equivalent to the triviality of its realizations. In the proof of Proposition \ref{Proposition_Ramakrishnan}, we will see that the triviality of $\tp(+\infty)$ is related to the existence of ``uniform bounds on growth" of functions, studied by Friedman and Miller in \cite{FM} and Ramakrishnan in \cite{Ramakrishnan}; we adapt the existence of uniform bounds to the context of 1-types as follows:

\begin{Definition}\label{Definition_p_germs_bdd}
Let $\mathfrak p\in S_1(\Mon)$. We say that {\it $\mathfrak p$-germs are $A$-bounded} if $\mathfrak p$ is $A$-invariant and for all $B\supseteq A$ and all relatively $B$-definable functions $f:\mathfrak p_{\restriction B}(\Mon)\to \mathfrak p_{\restriction B}(\Mon)$, there exists a relatively $A$-definable function $g:p(\Mon)\to p(\Mon)$ such that $f(x)\leqslant g(x)$ for all $x\in \mathfrak p_{\restriction B}(\Mon)$. 
\end{Definition} 

Equivalently, $\mathfrak p$-germs are $A$-bounded if every relatively definable function that maps the locus of $\mathfrak p$ in a larger monster into itself is bounded by a relatively $A$-definable such function.

\begin{Remark}\label{Remark_boundedgerms}
If $\mathfrak p$-germs are $A$-bounded and $B,f$ are as in Definition \ref{Definition_p_germs_bdd}, then there are relatively $A$-definable functions $g,g' :\mathfrak p_{\restriction B}(\Mon)\to \mathfrak p_{\restriction B}(\Mon)$ with $g'(x)\leqslant f(x)\leqslant g(x)$. By the Monotonicity Theorem, $f^{-1}$ is well defined and relatively $B$-definable, so there is a relatively $A$-definable $g_1$ with $f^{-1}(x)\leqslant g_1(x)$. Since $g_1$ is strictly increasing on $p(\Mon)$, $g_1^{-1}(f^{-1}(x))\leqslant x$ for all $x\in \mathfrak p_{\restriction B}(\Mon)$.
Then $g_1^{-1}(x)\leqslant f(x)$ holds for all $x\in \mathfrak p_{\restriction B}(\Mon)$, so  $g'=g_1^{-1}$ works.    
\end{Remark}
 
Let $p\in S_1(A)$ and let $a\models p$. Recall that the set $\mathcal D_p(a)$ is the union of all relatively $Aa$-definable subsets of $p(\Mon)$ that are bounded in $(p(\Mon),<)$; since $p(\Mon)$ is a convex subset of the monster, each of these bounded subsets is included in a segment with endpoints in $\dcl(Aa)\cap p(\Mon)$. Therefore, the set $\mathcal D_p(a)$ equals the union of all segments $[f(a),g(a)]$, where $f,g:p(\Mon)\to p(\Mon)$ are relatively $A$-definable functions; $\mathcal D_p(a)$ can also be described as the convex hull of $\dcl(Aa)\cap p(\Mon)$.   

\begin{Lemma}\label{Lemma_ominimal}
A type $p\in S_1(A)$ (in an o-minimal theory) is trivial if and only if $\mathbf p_r$-germs are $A$-bounded if and only if $\mathbf p_l$-germs are $A$-bounded.
\end{Lemma} 
\begin{proof}
We will prove that $p\in S_1(A)$ is trivial if and only if $\mathbf p_r$-germs are $A$-bounded; a similar proof establishes the other equivalence. By Lemma \ref{Lemma_trivial_iff_Dp=Dq} it suffices to prove that $\mathcal D_q(a)=\mathcal D_p(a)$ holds for all right nonforking extensions $q$ of $p$ and all $a\models q$, if and only if $\mathbf p_r$-germs are $A$-bounded.

For $(\Rightarrow)$, let $B\supseteq A$ and $f:q(\Mon)\to q(\Mon)$ be relatively $B$-definable. Denote $q=(\mathbf p_r)_{\restriction B}$;
by the assumption, for any $a\models q$ we have $\mathcal D_q(a)=\mathcal D_p(a)$. Then $f(a)\in \mathcal D_p(a)$, so there exists $b\in \mathcal D_p(a)\cap\dcl(Aa)$ such that $f(a)\leqslant b$; clearly, $b=g(a)$ for some relatively $A$-definable function $g$.

For $(\Leftarrow)$, suppose that $\mathbf p_r$-germs are $A$-bounded. We prove that $\mathcal D_{(\mathbf p_r)_{\restriction B}}(a)=\mathcal D_p(a)$ holds for all $B\supseteq A$ and all $a\models (\mathbf p_r)_{\restriction B}$.  Let $[f_1(a),f_2(a)]\subseteq \mathcal D_{(\mathbf p_r)_{\restriction B}}(a)$, where $f_1(a),f_2(a)\in \dcl(Ba)$. By Remark \ref{Remark_boundedgerms}, there are 
$g_1(a),g_2(a)\in \dcl(Aa)\cap p(\Mon)$ such that $g_1(a)\leqslant f_1(a)<f_2(a)\leqslant g_2(a)$. Then $[f_1(a),f_2(a)]\subseteq [g_1(a),g_2(a)]\subseteq \mathcal D_p(a)$. This finishes the proof.
\end{proof}
 
\begin{Proposition}\label{Proposition_Ramakrishnan} 
If $(M,<, \dots)$ is a densely ordered o-minimal structure, then every definable type $p\in S_1(M)$ is trivial.  
\end{Proposition}
\begin{proof}  Suppose that $p$ is definable and nonalgebraic. 
After possibly modifying the order $<$, while preserving the o-minimality, we may assume that $p=\tp(+\infty/M)$. 
We will show that $\mathbf p_r$-germs are $M$-bounded; consequently, by Lemma \ref{Lemma_ominimal}, $p$ is trivial. Let
$B\supseteq M$ and let $f:(\mathbf p_r)_{\restriction B}(\Mon)\to (\mathbf p_r)_{\restriction B}(\Mon)$ be a relatively $b$-definable function ($b\in B^n$); this is also expressed by (the formula) $\displaystyle\lim_{x\to +\infty}f(x)=+\infty$. By compactness, there are a $M$-definable set $D\subseteq \Mon^n$ and a $M$-definable function $F:D\times \Mon\to \Mon$ such that $b\in D$, $F(b,x)=f(x)$, and
$\displaystyle\lim_{x\to +\infty}F(d,x)=+\infty$ (for all $d\in D$). By \cite[Theorem 1.2]{Ramakrishnan}, there exists an $M$-definable function $g:\Mon\to\Mon$ such that for all $d\in D$: $F(d,x)\leqslant g(x)$ holds for all sufficiently large $x\in M$. In particular, $f(x)=F(b,x)\leqslant g(x)$ holds for some (equivalently, all) $x\in (\mathbf p_r)_{\restriction B}(\Mon)$, as desired.
 \end{proof}

In \cite{MTwqom2}, we will generalize the previous proposition to the context of nonisolated definable weakly o-minimal types. 

\begin{Example}[$\emptyset$-invariant type which is trivial over $\{0\}$, but not over $\emptyset$]\label{Example_trivover0}\

\noindent
Consider the affine reduct of the additive ordered group of reals $(\mathbb R,m,<)$, where $m(x,y,z):=x-y+z$. This structure is o-minimal with a unique type $p\in S_1(\emptyset)$ and exactly three complete 2-types determined by $x<y$, $y<x$, and $x=y$, respectively;  for example, if $a<b$, then $x\longmapsto \frac{x-a}{b-a}$ is an automorphism mapping $(a,b)$ to $(0,1)$, so $\tp(a,b)=\tp(0,1)$. In particular, every pair of distinct elements is independent. However, $p$ is non-trivial, since the triple $(0,1,2)$ is not a Morley sequence in $\mathbf p_r$ due to $m(1,0,1)=2$.  
Let $p_0(x)$ be the complete type over $\{0\}$ determined by $0<x$; clearly, $p_0(x)=(\mathbf p_r)_{\restriction \{0\}}$.    In a similar way as in Example \ref{Example_trivial_real_additive}, we may conclude that $p_0$ is trivial; this is because the structures
$(\mathbb R,m,<,0)$ and $(\mathbb R,+,<,0)$ are interdefinable. 
\end{Example}

\begin{Example}\label{Example non evetuallly trivial}
Consider the ordered group of rationales $(\mathbb Q,+,<,0)$, elementarily embedded in the group of reals. By Proposition \ref{Proposition_Ramakrishnan} we know that all definable types belonging to $S_1(\mathbb Q)$ are trivial; we will show that these are the only trivial types there. Let $p\in S_1(\mathbb Q)$ be non-definable. Then there is an irrational number $\alpha$ such that $p=\tp(\alpha/\mathbb Q)$. Since $\alpha$ is the unique realization of $p$ in $\mathbb R$, for each $a\in p(\Mon)$ we have $\dcl(\mathbb Qa)\cap p(\Mon)=\{a\}$. This implies that there are exactly three extensions of $p(x)\cup p(y)$ in $S_2(\mathbb Q)$, determined by $x<y$, $x=y$, and $y<x$, respectively. In particular, every two distinct realizations of $p$ are independent over $\mathbb Q$. It follows that if $a<b$ are distinct realizations of $p$, then $(a,\frac{a+b}{2},b)$ is a triangle of realizations of $p$, so $p$ is non-trivial.
\end{Example}

\subsection{Omitting trivial weakly o-minimal types}

In this subsection, we will prove that in a countable theory with few countable models, every trivial weakly o-minimal type over a finite domain, say $p\in S(A)$, is both convex and simple; in fact, we prove that if $p$ is either non-simple or non-convex, then an arbitrary countable endless linear order-type can be represented as $((\mathcal D_p(x)\mid x\in M),<)$ for some countable $M\models T$. This generalizes and is motivated by \cite[Proposition 7.8]{MT}, where we prove that in a binary theory with few countable models every so-type, and hence every weakly o-minimal type, over a finite domain is convex and simple; recall that every weakly o-minimal type in a binary theory is trivial.   
Also, recall that a weakly o-minimal type $p\in S(A)$ is {\it simple} if $x\dep_Ay$ is a relatively definable (equivalence) relation on $p(\Mon)$ (in which case, by $\Aut_{A}(\Mon)$-invariance, it is relatively $A$-definable); that is, $\mathcal D_p=\{(x,y)\in p(\Mon)^2\mid x\dep_A y\}$ is relatively definable within $p(\Mon)^2$ or, equivalently, $\mathcal D_p(a)$ is relatively definable for some/all $a\models p$. 
In the next proposition, we characterize weakly o-minimal types that are both convex and simple; again, by $a\triangleleft^{\mathbf p}x\triangleleft^{\mathbf p}b$ we denote the partial type $p(x)\cup(\mathcal D_p(a)<x)\cup (x<\mathcal D_p(b))$ (where $\mathbf p=(p,<)$ is a weakly o-minimal pair over $A$).

\begin{Proposition}\label{Prop middle definable iff simple and convex}
Suppose that $\mathbf p=(p,<)$ is a weakly o-minimal pair over $A$ and $A\triangleleft^{\mathbf p}a\triangleleft^{\mathbf p}b$. Then $p$ is both simple and convex if and only if the type $a\triangleleft^{\mathbf p}x\triangleleft^{\mathbf p}b$ has a definable locus.
\end{Proposition}
\begin{proof}
Let $(D,<)$ be a definable extension of $(p(\Mon),<)$.
Let $P$ denote the locus of $a\triangleleft^{\mathbf p}x\triangleleft^{\mathbf p}b$. Notice that $P$ is a $\mathcal D_p$-closed subset of $p(\Mon)$: if $c\in P$ then $\mathcal D_p(c)\subseteq P$; this holds since by Remark \ref{Remark_trianglep_binary_basicprops}(b) $a\triangleleft^{\mathbf p}c\triangleleft^{\mathbf p}b$ is equivalent with $a<\mathcal D_p(c)<b$.

$(\Leftarrow)$ Suppose that $P$ is definable. Then $P$ is $Aab$-definable, since it is clearly $Aab$-invariant. Let $\theta(y,x,z)$ be an $L_A$-formula such that $\theta(a,x,b)$ defines $P$. First, we prove that $p$ is simple, that is, $\mathcal D_p(e)$ is relatively definable within $p(\Mon)$ for some/all  $e\models p$. In fact, we prove more: $\mathcal D_p(e)$ is definable. Without loss, assume $e\in P$. 
We {\it claim} that $\mathcal D_p(e)$ is defined by the following formula:
$$\sigma(x)=\theta(a,x,b)\land\lnot \theta(a,x,e)\land\lnot\theta(e,x,b).$$
Indeed, note that $\sigma(x)$ is equivalent with $a\triangleleft^{\mathbf p}x\triangleleft^{\mathbf p}b\land x\ntriangleleft^{\mathbf p}e \land e\ntriangleleft^{\mathbf p}x$. For all $x\models p$,  $x\ntriangleleft^{\mathbf p}e \land e\ntriangleleft^{\mathbf p}x$ expresses the $\triangleleft^\mathbf{p}$-incompatibility of $x$ and $e$, so by Theorem \ref{Theorem4}(c) is equivalent with $x\in \mathcal D_p(e)$. As $\mathcal D_p(e)\subset P$, we conclude that $\sigma(x)$ is equivalent with $x\in \mathcal D_p(e)$. This proves the claim, so $p$ is simple. 
It remains to prove that $p$ is convex. Note that:
$$\{\theta(a,x_1,b),\theta(a,x_2,b),x_1<x<x_2\}\cup p(x)\vdash \theta(a,x,b).$$
Indeed, if $e_1,e,e_2$ satisfy the left-hand side (for $x_1,x,x_2$ respectively), then $e_1,e,e_2\models p$ and $a\triangleleft^{\mathbf p}e_1<e<e_2\triangleleft^{\mathbf p}b$, so $a\triangleleft^{\mathbf p}e\triangleleft^{\mathbf p}b$, and therefore $\models\theta(a,e,b)$. By compactness, find $\pi(x)\in p(x)$ such that:
\begin{equation}\label{eq convex}\tag{$\dagger$}
 (\forall x_1,x,x_2)(\theta(y,x_1,z)\land\theta(y,x_2,z)\land x_1<x<x_2\land\pi(x)\rightarrow \theta(y,x,z))\in\tp_{y,z}(a,b/A).
\end{equation}
Moreover, we can assume that $\pi(\Mon)\subseteq D$.
We now {\em claim} that $\pi(x)$ witnesses the convexity of $p$: $p(\Mon)$ is a $<$-convex subset of $(\pi(\Mon),<)$. Take $e_1,e_2\models p$ and $d\in\pi(\Mon)$ such that $e_1<d<e_2$; we need to show $d\in p(\Mon)$. 
Choose $a',b'\models p$ so that $a'\triangleleft^{\mathbf p}e_1$ and $e_2\triangleleft^{\mathbf p}b'$; then $a'\triangleleft^{\mathbf p}e_1<d<e_2\triangleleft^{\mathbf p}b'$. This implies $a'\triangleleft^{\mathbf p}e_1\triangleleft^{\mathbf p}b'$  and $a'\triangleleft^{\mathbf p}e_2\triangleleft^{\mathbf p}b'$, so $\models\theta(a',e_1,b')\land\theta(a',e_2,b')$ is valid. Also $a'\triangleleft^{\mathbf p}b'$, so $a'b'\equiv ab\ (A)$, hence from (\ref{eq convex}) we have:
$$\models (\theta(a',e_1,b')\land\theta(a',e_2,b')\land e_1<d<e_2\land\pi(d))\rightarrow \theta(a',d,b')\,.$$
The left side of this implication is valid, so $\models \theta(a',d,b')$. In particular, $d\models p$, and we are done.

$(\Rightarrow)$ Suppose now that $p$ is simple and convex. By Lemma \ref{Proposition_convex_type_witness}, there is $\pi(x)\in p(x)$ witnessing the convexity: $(\pi(\Mon),<)$ is a definable extension of $(p(\Mon),<)$ and $p(\Mon)$ is convex in $\pi(\Mon)$. 
By the simplicity of $p$, for $e\models p$, $\mathcal D_p(e)$ is a relatively $Ae$-definable subset of $p(\Mon)$, defined by $\phi(x,e)$ say.
Then the locus of $a\triangleleft^{\mathbf p}x\triangleleft^{\mathbf p}b$ is clearly defined by $\pi(x)\land \phi(\Mon,a)<x<\phi(\Mon,b)$. 
\end{proof}

\begin{Theorem}\label{Theorem trivial implies simple and convex}
Suppose that $T$ is countable and $I(\aleph_0,T)<2^{\aleph_0}$. Then every trivial weakly o-minimal type over a finite domain is simple and convex. 
\end{Theorem}
\begin{proof}
Assuming that $p$ is a trivial weakly o-minimal type over a finite domain that is not both convex and simple, we will prove $I(\aleph_0,T)=2^{\aleph_0}$. We may assume that $\mathbf p=(p,<)$ is a weakly o-minimal pair over $\emptyset$ (as for finite $A$, $I(\aleph_0,T_A)=2^{\aleph_0}$ implies $I(\aleph_0,T)=2^{\aleph_0}$).  
For every countable endless linear order $\mathbb I=(I,<_I)$ we will find a countable $M_{\mathbb I}\models T$ such that $(p(M_{\mathbb I})/\mathcal D_p,<)\cong \mathbb I$; clearly, for nonisomorphic orders $\mathbb I$ and $\mathbb J$ the corresponding models $M_{\mathbb I}$ and $M_{\mathbb J}$ will be nonisomorphic. As there are $2^{\aleph_0}$ many pairwise nonisomorphic such orders, the desired conclusion follows.     

Fix $\mathbb I=(I,<_I)$ and let $A_I=(a_i\mid i\in I)$ be a $\triangleleft^{\mathbf p}$-increasing sequence of realizations of $p$; since $p$ is trivial, $A_I$ is a Morley sequence in $\mathbf p_r$ over $\emptyset$. Put \ $\Pi_I(x):=p(x)\cup\bigcup_{i\in I}x\notin\mathcal D_p(a_i)$. 
 
\begin{Claim*} 
There exists a countable model $M_{\mathbb I}\supseteq A_I$ that omits $\Pi_I(x)$.
\end{Claim*}

\begin{proofC}
Suppose not. Then, by the Omitting Types Theorem there is an $L$-formula $\theta(x,\bar y)$ and a finite subsequence of $A_I$, $\bar a=(a_{i_1},\dots,a_{i_n})$, such that $\theta(x,\bar a)$ is consistent and $\theta(x,\bar a)\vdash \Pi_I(x)$. 
Moreover, we can assume that $\theta(\Mon,\bar a)$ is a convex subset of $(p(\Mon),<)$; indeed, $\theta(x,\bar a)\vdash p(x)$ holds, so according to the weak o-minimality of $p$,  $\theta(\Mon,\bar a)\subseteq  p(\Mon) $ has finitely many convex components in $(p(\Mon),<)$, and we can replace $\theta(x,\bar a)$ with a formula that defines one of them so that $\theta(x,\bar a)\vdash \Pi_I(x) $ still holds. 
Therefore, $D:=\theta(\Mon,\bar a)$ is a convex set that is disjoint from each of the (convex) sets $(\mathcal D_p(a_i)\mid i\in I)$. This implies that for each $i\in I$ either $\mathcal D_p(a_i)<D$ or $D<\mathcal D_p(a_i)$ holds, so every element $e\in D$ is $\triangleleft^{\mathbf p}$-comparable to $a_i$ and we can arrange $A_{\mathbb I}\cup \{e\}$ into a $\triangleleft^{\mathbf p}$-increasing sequence. 
Then the triviality of $p$ implies that this sequence is Morley in $\mathbf p_r$ over $\emptyset$, and that it remains such after replacing $e$ with any other $e'\in D$; in particular, $e\equiv e'\,(A_{\mathbb I})$ holds, so $\theta(x,\bar a)$ isolates a complete type $q\in S(A_{\mathbb I})$. Since $\theta(x,\bar a)$ uses only parameters $\bar a$, the type $q_{\restriction \bar a}\in S(\bar a)$ is isolated by $\theta(x,\bar a)$ ($q_{\restriction \bar a}(\Mon)=D$), and $q_{\restriction \bar a}\wor\tp(A_{\mathbb I}/\bar a)$ holds. Now we prove:
\begin{equation}
    \mbox{ $\mathcal D_p(a_{i_k})<D<\mathcal D_p(a_{i_{k+1}})$ \ holds for some $k<n$.  }
\end{equation}\setcounter{equation}{0}
We have a sequence of convex sets $\mathcal D_p(a_{i_0})<\mathcal D_p(a_{i_1})<\dots <\mathcal D_p(a_{i_n})$ that are disjoint from $D$ (which is also convex). Therefore, to prove (1), it suffices to exclude $D<\mathcal D_p(a_{i_0})$ and $\mathcal D_p(a_{i_n})<D$. 
Since $\mathbb I$ is endless, the type $\mathbf p_{r\restriction \bar a}$ is realized in $A_{\mathbb I}$, so $q_{\restriction \bar a}\wor\mathbf p_{r\restriction \bar a}$. 
We conclude that the locus $q_{\restriction \bar a}(\Mon)=D$ does not contain any elements above $\mathcal D_p(a_{i_n})$, as all realize $\mathbf p_{r\restriction \bar a}$. This excludes $\mathcal D_p(a_{i_n})<D$; similarly, $D<\mathcal D_p(a_{i_0})$ is ruled out, proving (1). 

Note that (1) implies $\theta(x,\bar a)\vdash a_{i_k}\triangleleft^{\mathbf p}x\triangleleft^{\mathbf p}a_{i_{k+1}}$. 
Also note that $a_{i_k}\triangleleft^{\mathbf p}x\triangleleft^{\mathbf p}a_{i_{k+1}}$ determines a complete type in $S(\bar a)$, as $a_{i_k}\triangleleft^{\mathbf p}d\triangleleft^{\mathbf p}a_{i_{k+1}}$ implies that $(a_{i_0},\ldots,a_{i_k},d,a_{i_{k+1}},\ldots,a_{i_n})$ is a Morley sequence in $\mathbf p_r$ over $\emptyset$; this type is $q_{\restriction \bar a}(x)$, which is isolated. In particular, the locus of $a_{i_k}\triangleleft^{\mathbf p}x\triangleleft^{\mathbf p}a_{i_{k+1}}$ is definable, so by Proposition \ref{Prop middle definable iff simple and convex}, $p$ is convex and simple. This contradicts the initial assumptions and completes the proof of Claim.
\end{proofC}

\smallskip
Let the model $M_{\mathbb I}$ be given by the claim. To complete the proof of the theorem, it suffices to note that since $M_{\mathbb I}$ omits $\Pi_{\mathbb I}(x)$ and since $A_I\subseteq M_{\mathbb I}$, the classes $(\mathcal D_p(a_i)\mid i\in I)$ are the only $\mathcal D_p$-classes represented in $p(M_{\mathbb I})$, so $(p(M_{\mathbb I})/\mathcal D_p,<)\cong \mathbb I$.
\end{proof}

\section{Semiintervals and shifts}\label{Section_shifts}
  
In this section, we first introduce the notions of semiintervals and shifts and then prove Theorem \ref{Theorem1_shift}. The motivation for these comes from the work of Baizhanov, Kulpeshov, and others  \cite{Baizhanov2006}, \cite{Kulpeshov2007}, and \cite{Alibek}.  

Let $(D,<)$ be a linear order. We say that $S\subseteq D$ is a {\it semiinterval of $(D,<)$} if $S$ is convex in $(D,<)$ and $a=\min S$ exists; in that case, we also say that $S$ is a semiinterval with minimum $a$, and denote it by $S_a$ to emphasize $a=\min S_a$. In this paper, we will be dealing with linear orders $(D,<)$ and families of semiintervals $\mathcal S=(S_a\mid a\in D)$ that are definable over some small set of parameters.

\begin{Definition}\label{Definition_family_semiintervals}
Let $(D,<)$ be a linear order and let $\mathcal S=(S_x\mid x\in D)$ be a family of semiintervals. 
\begin{enumerate}[(a),left=0em,labelsep=.7em]
\item $\mathcal S$ is an {\it $A$-definable family of semiintervals of $(D,<)$} if $(D,<)$ is $A$-definable and there exists an $L_A$-formula $S(x,y)$ such that $S_a=S(a,\Mon)$ holds for all $a\in D$.
\item The family $\mathcal S$ is {\it monotone} if $\supin (S_x)\subseteq \supin (S_y)$ for all $x\leqslant y$ from $D$. (Recall that $\supin(S_x)=\{d\in D\mid (\exists t\in S_x)\ d\leqslant t\}$ is the smallest inital part of $D$ that contains $S_x$.)
\end{enumerate}
\end{Definition}  

Suppose that $\mathcal S=(S_x\mid x\in D)$ is an $A$-definable family of semiintervals of $(D,<)$, defined by $S(x,y)$. Recursively, define: 
\begin{center}
 $S^{(1)}(x,y):=S(x,y)$  \ \ and \ \ $S^{(n+1)}(x,y):=(\exists t)(S^{(n)}(x,t)\land S(t,y))$ \  ($n\in\mathbb N$).     
\end{center}
Further in the text, we will use the following notation:
\begin{center}$S_x^n:=S^{(n)}(x,D)$,\ \ \ $\mathcal S^{n}:=(S^n_x\mid x\in D)$ \ \ \ and \ \ \ $S^{\omega}_x:=\bigcup_{n\in\mathbb N} S_x^n$.  
\end{center}
 
\begin{Observation}\label{Lemma_semiint_verybasic1}
Suppose that $\mathcal S=(S_x\mid x\in D)$ is an $A$-definable family of semiintervals of $(D,<)$, defined by the $L_A$-formula $S(x,y)$.  
\begin{enumerate}[(a),left=0em,labelsep=.7em]
\item $\mathcal S^{n}$ is an $A$-definable family of semiintervals of $(D,<)$; if $\mathcal S$ is monotone, then so is $\mathcal S^n$ for all $n\in \mathbb N$.  
\item  $(S_x^n\mid n\in\mathbb N)$ is a $\subseteq$-increasing sequence for all $x\in D$. If $S_x^n=S_x^{n+1}$ for some $n\in \mathbb N$, then $S_x^n=S_x^m$ for all $m\geqslant n$, so $S_x^n=S_x^\omega$. Hence, the sequence is strictly increasing or eventually constant.
\item $(S^{\omega}_x\mid x\in D)$ is a family of semiintervals of $(D,<)$; it is monotone if $\mathcal S$ is monotone.
\item $\models S^{(n+m)}(x,y)\leftrightarrow (\exists t)(S^{(n)}(x,t)\land S^{(m)}(t,y))$, so 
$S^{n+m}_a=\bigcup_{t\in S_a^n}S_t^m$ for all $m\in\mathbb N$.
\end{enumerate}
\end{Observation}

\begin{Definition}\label{Definition_shift} 
Let $\mathcal S=(S_x\mid x\in D)$ be a family of semiintervals of $(D,<)$. We say that the semiinterval $S_a\in\mathcal S$ is an {\em $\mathcal S$-shift} if the sequence $(S_a^n\mid n\in\mathbb N)$ is strictly increasing. 
\end{Definition}

\begin{Remark}\label{Remark_shift} 
Let $\mathcal S$ be a family of semiintervals of $(D,<)$ and let $S_a\in \mathcal S$. By Observation \ref{Lemma_semiint_verybasic1}(b) we have two options: either $S_a$ is an $\mathcal S$-shift, or $S_a^n=S^{\omega}_a$ holds for some $n\in\mathbb N$. 
\end{Remark}

\begin{Example}
(a) Consider the theory $T=\Th(\mathbb Q,<,S)$ where $<$ is the usual order on the rationals and $S$ is defined by: \ $(x,y)\in S$ if and only if $x\leqslant y<x+\sqrt{2}$. Here, the formula $S(x,y)$ defines a monotone family of semiintervals with $S_0^n=\{y\mid 0\leqslant y<n\sqrt 2\}$, so $S_0$ is a shift. Since $T$ is weakly o-minimal, $I(T,\aleph_0)=2^{\aleph_0}$ follows by Theorem \ref{Theorem1_shift}; it also follows by Theorem \ref{Theorem trivial implies simple and convex}, since the unique type $p\in S_1(T)$ is trivial and non-simple.

(b) A typical example of a monotone family without shifts is as follows. Let $(\Mon,<,E)$ be a linear order with a convex equivalence relation $E$. Define $S(x,y):=x\leqslant y\land E(x,y)$; $S_x$ is a final part of the class $[x]_E$ that begins at $x$. Here $S_x^n=S_x$ holds for all $n$, so $S_x$ is not a shift. 
\end{Example}

\begin{Lemma}\label{Lemma_shift_properties}
Suppose that $\mathcal S$ is a monotone $A$-definable family of semiintervals of $(D,<)$, $a\in D$, and $b\in S^{\omega}_a$. Then $S^{\omega}_b$ is a final part of $S^{\omega}_a$, and if $S_a$ is an $\mathcal S$-shift, then so is $S_b$. 
\end{Lemma}
\begin{proof}
Choose $n\in\mathbb N$ such that $b\in  S^n_a$. As we observed in \ref{Lemma_semiint_verybasic1}(d), $S^{n+m}_a=\bigcup_{x\in S_a^n}S_x^m$ holds for all $m\in\mathbb N$. This, together with $b\in S_a^n$ implies $S^{m+n}_a\supseteq S^m_b$, so $S_a^\omega\supseteq S_b^\omega$ holds. 
Since the family $\mathcal S^m$ is monotone for all $m\in\mathbb N$, $a\leqslant b$ implies $\supin (S_a^m)\subseteq \supin (S_b^m)$ and thus $\supin (S_a^\omega)\subseteq \supin (S_b^\omega)$. This, together with $S_a^\omega\supseteq S_b^\omega$ implies that $S_b^\omega$ is a final part of $S_a^\omega$ and proves the first claim. 
To prove the second, assume that $S_b$ is not an $\mathcal S$-shift; we prove that $S_a$ is also not an $\mathcal S$-shift. By Remark \ref{Remark_shift}, there is $k\in\mathbb N$ with $S_b^k=S_b^\omega$. By the above proved, we have $S^k_b\subseteq S_a^{k+n}$; since $S_b^\omega=S_b^k$ is a final part of $S_a^\omega$, we conclude that the set $S_a^{k+n}$ contains a final part of $S_a^\omega$. Clearly, this implies $S_a^{k+n}=S_a^\omega$, so $S_a$ is not an $\mathcal S$-shift by Remark \ref{Remark_shift}. This proves the second claim.     
\end{proof}

In the next lemma, we prove that every definable family of semiintervals can be slightly modified to become monotone. 

\begin{Lemma}\label{Prop_exists_monotonefamily}
Suppose that $\mathcal S=(S_x\mid x\in D)$ is an $A$-definable family of semiintervals of $(D,<)$ and $a\in D$. 
Then there is an $Aa$-definable monotone family of semiintervals $\mathcal R=(R_x\mid x\in D)$ such that $R_a=S_a$ and $R_a^{\omega}=S_a^{\omega}$. Moreover, if $S_a$ is an $\mathcal S$-shift, then $R_b$ is an $\mathcal R$-shift for all $b\in R_a^{\omega}$.    
\end{Lemma}
\begin{proof}
Using $a$ as an additional parameter, define the family $\mathcal R$ by letting $R_x=\{x\}$ for $x<a$, and $R_x=[x,+\infty)_D\cap\bigcup_{a\leqslant t\leqslant x}S_t$ for $x\geqslant a$. Clearly, $\mathcal R$ is a monotone $Aa$-definable family, $R_a=S_a$,  $S_x\subseteq R_x$ for all $x\in D$, and $S_a^k\subseteq R_a^k$ for all $k\in\mathbb N$. 

\begin{Claim*}
   $b\in S_a^n$ \ implies \ $R_b^k\subseteq S_a^{n+k}$ \ for all $k,n\in\mathbb N$.
\end{Claim*}
\begin{proofC}
We proceed by induction on $k$. For $k=1$ we prove that for all $n\in\mathbb N$, $b\in S_a^n$ implies $R_b\subseteq S_a^{n+1}$. We have \ $R_b\subseteq \bigcup_{a\leqslant t\leqslant b}S_t\subseteq \bigcup_{t\in S_a^n}S_t=S_a^{n+1}$;  here, the first inclusion follows by the definition of $R_b$, the second by $b\in S_a^n$, and the equality by the definition of $S_a^{n+1}$. Assume that the claim holds for some $k\in\mathbb N$: for all $n\in\mathbb N$, $b\in S_a^n$ implies $R_b^k\subseteq S_a^{n+k}$. Fix $n\in\mathbb N$ and $b\in S_a^n$, and we prove $R_b^{k+1}\subseteq S_a^{n+k+1}$. We have:
\begin{center}
    $R_b^{k+1}=\bigcup_{t\in R_b^k}R_t\subseteq \bigcup_{t\in S_a^{n+k}}R_t  \subseteq S_a^{n+k+1}$;
\end{center}
Here, the equality holds by the definition, the first inclusion follows from the induction hypothesis as $b\in S_a^n$, and the second holds as $R_t\subseteq S_a^{n+k+1}$ for all $t\in S_a^{n+k}$ by the base case (the case $k=1$). This proves the claim. 
\end{proofC}

By the claim, as $a\in S_a$, $S_a^k\subseteq R_a^k\subseteq  S_a^{k+1}$ holds for all $k\in\mathbb N$. Clearly, this implies $R_a^{\omega}=S_a^{\omega}$, proving the first assertion.
To prove the second, assume that $S_a$ is an $\mathcal S$-shift, that is, the family $(S_a^k\mid k\in\mathbb N)$ is strictly increasing. Then $S_a^k\subseteq R_a^k\subseteq  S_a^{k+1}$ implies that the family $(R_a^k\mid k\in\mathbb N)$ is also strictly increasing, so $R_a$ is an $\mathcal R$-shift. By Lemma \ref{Lemma_shift_properties}, $R_b$ is an $\mathcal R$-shift for all $b\in  R^{\omega}_a$, proving the proposition. 
\end{proof}

\subsection{Proof of Theorem \ref{Theorem1_shift}}

In this subsection, we will prove Theorem \ref{Theorem1_shift}. 
Throughout, assume that $T$ is a countable weakly quasi-o-minimal theory with respect to $<$, $\mathcal S=(S_x\mid x\in \Mon)$ is a definable family of semiintervals of $(\Mon,<)$, $c_0\in \Mon$, and $S_{c_0}$ is an $\mathcal S$-shift. For any type $p\in S_1(A)$, $\mathbf p$ will denote the pair $(p,<)$. We need to prove $I(T,\aleph_0)=2^{\aleph_0}$.

Absorb in the language the parameter $c_0$ and a finite set of parameters needed to define the family $\mathcal S$; note that this does not affect the desired conclusion $I(T,\aleph_0)=2^{\aleph_0}$.
By Lemma \ref{Prop_exists_monotonefamily}, after possibly modifying $\mathcal S$, we can also assume that $\mathcal S$ is monotone. Next, redefine $\mathcal S$ by setting $S_x=\{x\}$ for all $x<c_0$, and modify the order $<$ by moving the interval $(-\infty, c_0)$ to the right end. Notice that the redefined family is still monotone and $0$-definable, and that $c_0=\min \Mon$. 
Therefore, for the rest of this section, we will assume:  \begin{enumerate}[label=$\bullet$,left=1em,labelsep=.7em]
\item  $\mathcal S=(S_x\mid x\in\Mon)$ \ is a $0$-definable, monotone family of semiintervals of $(\Mon,<)$;
\item $c_0=\min \Mon$ \ and \ $S_{c_0}$  is an $\mathcal S$-shift; $S^{1}_{c_0}\subset S^{2}_{c_0}\subset S^{3}_{c_0}\subset\dots$ is a family of initial parts of $\Mon$. 
\end{enumerate}
The following notation is convenient:
\begin{enumerate}[label=$\bullet$,left=1em,labelsep=.7em]
    \item $\mathcal C:=S^{\omega}_{c_0} =\bigcup_{n\in\mathbb N} S^{n}_{c_0}$; 
    \item  $\Pi_A(x):=\{\phi(x)\in L_A\mid \mbox{$\phi(\Mon)$ is   convex and contains a final part of $\mathcal C$}\}$; $\Pi(x):=\Pi_{\emptyset}(x)$; \ \ 
    \item $S_{\Pi}(A):=\{p\in S_1(A)\mid \Pi_A(x)\subseteq p(x)\}$. 
\end{enumerate}

\begin{Remark}\label{Remark_mathcal C basic}
\begin{enumerate}[(a),left=0em,labelsep=.7em] 
\item $\mathcal C$ is an initial part of $(\Mon,<)$, and $\mathcal C<x$ is $0$-type-definable. For all $n\in \mathbb N$, the formula $S^n_{c_0}<x$ belongs to $\Pi_A(x)$.
 \item  Lemma \ref{Lemma_shift_properties} implies that $S_c$ is an $\mathcal S$-shift and $S^\omega_c$ is a final part of $\mathcal C$ for all $c\in \mathcal C$. 
 \item $\Pi_A(x)$ is a partial type, $(\mathcal C<x)\subseteq \Pi_A(x)$, and $\Pi_A(x)$ is finitely satisfiable in $\mathcal C$.
\end{enumerate}
\end{Remark}

We will say that $p\in S_1(A)$ is a {\it $\mathcal C$-type} if $p$ is finitely satisfiable but not realized in $\mathcal C$. $\mathcal C$-types have extension property: if $p\in S_1(A)$ is a $\mathcal C$-type and $A\subseteq B$, then there is a $\mathcal C$-type $q\in S_1(B)$ with $p\subseteq q$.
Most of our efforts are aimed at proving Corollary \ref{Corollary_Ctype_trivial}, which states that every $\mathcal C$-type over a finite domain is trivial. Using this fact and Theorem \ref{Theorem trivial implies simple and convex}, we will be able to quickly complete the proof of Theorem \ref{Theorem1_shift}. 

\begin{Lemma}\label{Lemma_Pi_basic}
\begin{enumerate}[(a),left=0em,labelsep=.7em]
\item  For all $p\in S_1(A)$ the following conditions are equivalent:\\
\begin{enumerate*}[label=(\arabic*),itemjoin={\hspace{3em}}]
    \item  \ $p\in S_{\Pi}(A)$; 
    \item  \ $p$ is a $\mathcal C$-type;
    \item  \ $\mathcal C=\{x\in\Mon\mid x< p(\Mon)\}$. 
\end{enumerate*}
\item  $\mathcal C\cup\Pi_A(\Mon)$ is an initial part of $\Mon$.
\item  $\Pi_A(x)$ is the minimal interval type from $IT(A)$ consistent with $\mathcal C<x$. 
\item For all $A\subseteq B$, $S_{\Pi}(B)=\{(\mathbf p_l)_{\restriction B}\mid p\in S_{\Pi}(A)\}$.  
\end{enumerate}
\end{Lemma}
\begin{proof} (a) (1)$\Rightarrow$(2) Assume $p\in S_\Pi(A)$. Then $(\mathcal C<x)\subseteq \Pi_A(x)\subseteq p(x)$, so $p$ is not realized in $\mathcal C$. To prove that $p$ is a $\mathcal C$-type, it remains to show that $p(x)$ is finitely satisfiable in $\mathcal C$. Suppose that an $L_A$-formula $\phi(x)$ is not realized in $\mathcal C$. Then the set of realizations of (a formula) $x<\phi(\Mon)$ is convex in $\Mon$ and contains $\mathcal C$; hence $(x<\phi(\Mon))\in \Pi_A(x)\subseteq p(x)$ and thus $p(\Mon)<\phi(\Mon)$. In particular, $\phi(x)\notin p$. Therefore, $p$ contains only formulae that are realized in $\mathcal C$; $p$ is a $\mathcal C$-type. 

(2)$\Rightarrow$(3) Suppose that $p$ is a $\mathcal C$-type. We need to prove that $\mathcal C$ is the set of all lower bounds of $p(\Mon)$. Since $p$ is not realized in $\mathcal C$, which is an initial part of $\Mon$, we conclude $\mathcal C< p(\Mon)$, so any $c\in\mathcal C$ is a lower bound of $p(\Mon)$. To prove (3), it remains to show that there are no other lower bounds; assuming $\mathcal C<a$ we will prove $a\nless p(\Mon)$. Since $p$ is a $\mathcal C$-type, by the extension property there is a $b\in p(\Mon)$ such that $\tp(b/aA)$ is a $\mathcal C$-type. Then, since $(x<a)\in \tp(c/aA)$ for all $c\in \mathcal C$, we have $(x<a)\in \tp(b/aA)$. 
Now, $b<a$ and $b\in p(\Mon)$ together imply $a\nless p(\Mon)$, so $a$ is not a lower bound of $p(\Mon)$.

(3)$\Rightarrow$(1) Suppose that $\mathcal C$ is the set of all lower bounds of $p(\Mon)$. To prove $p\in S_{\Pi}(A)$, i.e. $\Pi_A(x)\subseteq p(x)$, assume $\phi(x)\in\Pi_A(x)$; we need to prove $\phi(x)\in p$.
Since the set $\phi(\Mon)$ contains a final part of $\mathcal C$, 
there is a $k\in\mathbb N$ such that $\mathcal C\subseteq S^k_{c_0}\cup \phi(\Mon)$. 
Put $\psi(x):=(\forall y)(y\leqslant x\rightarrow y\in S^k_{c_0}\cup \phi(\Mon))$; then $\psi(\Mon)=in(\phi(\Mon))$ is the smallest initial part of $D$ that contains $\phi(\Mon)$. Clearly, 
$\mathcal C\subseteq \psi(\Mon)$, so by compactness  $\{\psi(x)\land\ S^n_{c_0}<x\mid n\in\mathbb N\}$ is consistent; let $b$ realize it. Then $\mathcal C<b$, so, by (3), $b$ is not a lower bound of $p(\Mon)$. Hence, there is $a\models p$ such that $a\leqslant b$. 
Then $\models \psi(b)$ implies $\models \phi(a)$, so $\phi(x)\in p$. Therefore, $\Pi_A(x)\subseteq p(x)$.

(b) Assuming $b<a$ and $a\in \mathcal C\cup \Pi_A(\Mon)$ we need to prove $b\in \mathcal C\cup \Pi_A(\Mon)$. This is immediate if $a\in \mathcal C$ or $b\in\mathcal C$. So assume $a,b\notin\mathcal C$ and let $p=\tp(a/A), q=\tp(b/A)$. Then $p\in S_{\Pi}(A)$ as $a\in \Pi_A(\Mon)$. By (1)$\Leftrightarrow$(3) from part (a), $\mathcal C$ is the set of all lower bounds of $p(\Mon)$. In particular, $b$ is not a lower bound of $p(\Mon)$, so $a'<b<a$ holds for some $a'\models p$. It follows that $\conv(p(\Mon))=\conv(q(\Mon))$ by Lemma \ref{Lemma_interval_types}(f), so $p(\Mon)$ and $q(\Mon)$ have the same set of lower bounds, hence $q\in S_{\Pi}(A)$ follows by (1)$\Leftrightarrow$(3) from part (a). Therefore, $b\in\Pi_A(\Mon)$, as desired.   

(c) We prove that $\Pi_A(x)$ is an interval type, i.e. a maximal consistent set of convex $L_A$-formulae. Let $\phi(x)$ be a convex $L_A$-formula such that $\phi(x)\notin \Pi_A(x)$; we show that $\phi(x)$ is inconsistent with $\Pi_A(x)$. As $\phi(x)$ is convex and $\phi(x)\notin \Pi_A(x)$, the set $\phi(\Mon)$ does not contain a final part of $\mathcal C$. Then one of the convex components of $\lnot\phi(\Mon)$, say $D$, contains a final part of $\mathcal C$. This implies that the formula expressing $x\in D$ belongs to $\Pi_A(x)$. Since $x\in D$ is inconsistent with $\phi(x)$, it follows that $\phi(x)$ is inconsistent with $\Pi_A(x)$, as desired. Therefore, $\Pi_A(x)$ is an interval type. The minimality of $\Pi_A(x)$ follows from (b).

(d) Let $p\in S_{\Pi}(A)$. By (1)$\Leftrightarrow$(3) from part (a), $\mathcal C$ is the set of all the lower bounds of $p(\Mon)$. Let $p_B=(\mathbf p_l)_{\restriction B}$.   
By the definition of $\mathbf p_l$, the set $p_B(\Mon)$ is an initial part of $p(\Mon)$, so $p(\Mon)$ and $p_B(\Mon)$ share the same set, $\mathcal C$, of lower bounds; $p_B\in S_{\Pi}(B)$ follows by part (a) of the lemma. Therefore, $\{(\mathbf p_l)_{\restriction B}\mid p\in S_{\Pi}(A)\}\subseteq S_{\Pi}(B)$. 
To prove the converse, suppose that $q_B\in S_{\Pi}(B)$ and let $q=(q_B)_{\restriction A}$; clearly, $q\in S_{\Pi}(A)$. Now, part (a) of the lemma applies to both $q$ and $q_B$, so the sets $q(\Mon)$ and $q_B(\Mon)$ have the same set, $\mathcal C$, of lower bounds. This, together with the fact that $q_B(\Mon)$ is a convex subset of $q(\Mon)$, implies that $q_B(\Mon)$ is an initial part of $q(\Mon)$, so $q_B=(\mathbf q_l)_{\restriction B}$, proving the converse. Therefore, $S_{\Pi}(B)=\{(\mathbf p_l)_{\restriction B}\mid p\in S_{\Pi}(A)\}$.
\end{proof}

Next, we consider forking as a binary relation on $\Pi_A(\Mon)$: 
\begin{enumerate}[label=$\bullet$,left=1em,labelsep=.7em]
    \item $\mathcal D_{\Pi_A}:=\{(x,y)\in\Pi_A(\Mon)^2\mid x\dep_A y\}$; \ \ \ \ $\mathcal D_{\Pi_A}(a):=\{x\in \Pi_A(\Mon)\mid (x,a)\in\mathcal D_{\Pi_{A}}\}$.
\end{enumerate}
Recall that a set $D\subseteq \Mon$ is $\Pi_A$-bounded if there are $a_1,a_2\in \Pi_A(\Mon)$ such that $a_1<D<a_2$.

\begin{Lemma}\label{Lemma_forking_on_D_Pi}
\begin{enumerate}[(a),left=0em,labelsep=.7em]
\item  $(a,b)\in \mathcal D_{\Pi_{A}}$ \ if and only if \ $\tp(a/Ab)$ contains a $\Pi_A$-bounded formula.
\item  $\mathcal D_{\Pi_A}$ is a convex equivalence relation on $\Pi_A(\Mon)$.
\item  $\mathcal S^{\omega}_a\subseteq \mathcal D_{\Pi_A}(a)$ \ holds for all $a\in\Pi_A(\Mon)$.
\item $S_a$ is an $\mathcal S$-shift for all $a\in \Pi(\Mon)$.
\end{enumerate}
\end{Lemma}
\begin{proof}
(a) and (b) follow from Lemma \ref{Lemma_forking_on_interval_type}. 

(c) Let $a\in \Pi_A(\Mon)$ and we prove that $S_a$ is $\Pi_A$-bounded. 
Put $p=\tp(a/A)$. Choose $b\in p(\Mon)$ such that $a$ is left $\mathbf p$-generic over $b$. By Lemma \ref{Lemma_Pi_basic}(a,d), $\tp(a/bA)$ is a $\mathcal C$-type. Note that for all $c\in\mathcal C$ we have $S_c<b$, so the formula (expressing) $S_x<b$ belongs to $\tp_x(a/bA)$. In particular, $S_a<b$ and thus $a\leqslant S_a<b$. Therefore, $x\in S_a$ is a $\Pi_A$-bounded formula, so $S_a\subseteq \mathcal D_{\Pi_A}(a)$; $\mathcal S^{n}_a\subseteq \mathcal D_{\Pi_A}(a)$ follows by induction, hence $\mathcal S^{\omega}_a\subseteq \mathcal D_{\Pi_A}(a)$.

(d) Let $a\in \Pi(\Mon)$. Consider the partial type $\Sigma=\{S_x^n\subset S_x^{n+1}\mid n\in\mathbb N\}$. By Remark \ref{Remark_mathcal C basic}(b), this partial type is realized by any element of $\mathcal C$; since, by Lemma \ref{Lemma_Pi_basic}(a), $\tp(a)$ is a $\mathcal C$-type, $a$ realizes $\Sigma$. Therefore, $S_a$ is an $\mathcal S$-shift.        
\end{proof}

\begin{Definition}
Let $D\subseteq \Mon$ and $a\in D$. We say that $a$ is an {\it $\mathcal S$-minimal element of $D$} if $a=c_0$ or if there exists  $b< D$ such that $a\in S_b$.
\end{Definition}
Usually, when we say that $a\in D$ is $\mathcal S$-minimal, the minimality is with respect to $D$. Note that if $D$ is definable, then ``$x$ is an $\mathcal S$-minimal element of $D$'' is also definable.  

\begin{Lemma}\label{Lemma_S_minimal_exist}

\begin{enumerate}[(a),left=0em,labelsep=.7em]
\item  Every definable set $D$ that intersects $\mathcal C\cup \Pi(\Mon)$ contains an $\mathcal S$-minimal element. 

\item ($T$ small) \ If $D$ is $\bar b$-definable and intersects $\mathcal C\cup \Pi(\Mon)$, then there is an $\mathcal S$-minimal element $a\in D$ such that $\tp(a/\bar b)$ is isolated. 

\item  Suppose that $D\subseteq \mathcal C\cup \Pi(\Mon)$ is the locus of some isolated type, say $q\in S_1(A)$. Then every element of $D$ is $\mathcal S$-minimal. Moreover, whenever $D$ intersects some $\mathcal D_{\Pi_A}$-class, it is contained entirely within that class; in particular, $D$ is bounded in $(\mathcal C\cup\Pi(\Mon),<)$.  
\end{enumerate}
\end{Lemma}
\begin{proof} 
(a) First, consider the case $D\cap \mathcal C\neq 0$. If $c_0\in D$, then $c_0$ is $\mathcal S$-minimal in $D$. Assume $c_0\notin D$ and let $n\in\mathbb N$ be minimal with $S^n_{c_0}\cap D\neq 0$. We claim that every element $a\in S^n_{c_0}\cap D$ is $\mathcal S$-minimal in $D$.  If $n=1$  then $c_0$ witnesses the $\mathcal S$-minimality of $a$ in $D$, as $c_0<D$ and $a\in S_{c_0}$. If $n>1$, then $a\in S^n_{c_0}$ implies that $a\in S_d$ holds for some $d\in S^{n-1}_{c_0}$;  the assumed minimality of $n$ implies $d<D$, so $d$ witnesses the $\mathcal S$-minimality of $a$ in $D$.  

The second case is $D\cap\Pi(\Mon)\neq 0$. Let $\phi(x,\bar b)$ be a formula that defines $D$ and let $d\in D\cap \Pi(\Mon)$. 
Let $\sigma(y,\bar z)$ be an $L$-formula that says ``$y$ is an $\mathcal S$-minimal element of $\phi(\Mon,\bar z)$''. 
By the first case, for each $\bar b'\in \Mon^{|\bar b|}$ the following holds: if the set $\phi(\Mon,\bar b')$
meets $\mathcal C$ then $\phi(\Mon,\bar b')$ has an $\mathcal S$-minimal element. Therefore, 
for all $c\in\mathcal C$ 
we have: \ $\models  (\forall \bar z)(\phi(c,\bar z)  \rightarrow (\exists y)\, \sigma(y,\bar z))$. Since $\tp(d)$ is finitely satisfiable in $\mathcal C$, the same holds with $c$ replaced by $d$; $D$ contains an $\mathcal S$-minimal element. 

(b) Suppose that $D$ is $\bar b$-definable and that it intersects $\mathcal C\cup\Pi(\Mon)$. Let $\phi(x)$ be an $L_{\bar b}$-formula that expresses ``$x$ is an $\mathcal S$-minimal element of $D$''; it is consistent by part (a), so, since $T$ is small, $\phi(x)$ belongs to some isolated type of $S_1(\bar b)$.

(c) Suppose that $\phi(x)$ defines $D$ and isolates $q$. By (a), we know that some $x\in D$ satisfies the formula ``$x$ is an $\mathcal S$-minimal element of $\phi(\Mon)$''. Since $D$ is $A$-invariant, all elements of $D$ satisfy that formula. This proves the first assertion. To prove the second, assume that $a\in D\cap \Pi_A(\Mon)$; we need to show $D\subseteq \mathcal D_{\Pi_A}(a)$. By the first claim, $a$ is $\mathcal S$-minimal in $D$, so there is a $b\in \Pi(\Mon)$ witnessing that: $b<D$ and $a\in S_b$. By Lemma \ref{Lemma_forking_on_D_Pi}(c), we have $\mathcal S^{\omega}_b\subseteq \mathcal D_{\Pi_A}(b)$, so $a\in S_b$ implies $\mathcal D_{\Pi_A}(a)=\mathcal D_{\Pi_A}(b)$. Therefore, to complete the proof, it suffices to prove $D\subseteq S_a\cup S_b$.  

Suppose on the contrary that there is $d\in D\smallsetminus (S_a\cup S_b)$. 
Note that $S_a\cup S_b$ is a convex set, so $b\in S_b$ and $b<D$ together imply $(S_b\cup S_a)<d$. By the above proved, $d\in D$ is $\mathcal S$-minimal; let $b'\in \Pi_A(\Mon)$ be a witness for that: $b'<D$ and $d\in S_{b'}$. Then $b'<a$ as $b'<D\ni a$, so the monotonicity of $\mathcal S$ implies $\supin (S_{b'})\subseteq\supin (S_a)$. On the other hand,  $S_a<d$ and $d\in S_{b'}$ imply $\supin (S_{a})\subset\supin (S_{b'})$. Contradiction.
\end{proof}

\begin{Lemma}\label{Lemma_every_convexshift_istrivial}  ($T$ small) \ 
Every convex $\mathcal C$-type over a finite domain is trivial.  
\end{Lemma}
\begin{proof}
Toward contradiction assume that $A_0$ is finite and $r\in S_1(A_0)$ is a non-trivial convex $\mathcal C$-type.   
By Remark \ref{Remark_notrivial_wom_witness}(c), there is a finite Morley sequence, $\bar b$, in $\mathbf{r}_l$ over $A_0$  such that the type $p=(\mathbf{r}_l)_{\restriction A_0\bar b}$ is non-trivial, as witnessed by a triangle; put $A=A_0\bar b$. Note that Lemma \ref{Lemma_Pi_basic}(a,d) guarantees:

(1) \ $p\in S_1(A)$ is a $\mathcal C$-type and $p\in S_{\Pi}(A)$.

\noindent
Also note that $p$ is convex, as a complete extension of a convex weakly o-minimal type $r$. Since $\Pi_A(x)$ is an interval type and since $p\in S_{\Pi}(A)$ is convex, $p$ is an isolated point of $S_{\Pi}(A)$. Then there is a formula $\theta(x)$ such that:  

(2) \ $\theta(x)\in p$ \ and \ $\Pi_A(x)\cup\{\theta(x)\}\vdash p(x)$.

\noindent 
Note that this guarantees that $p(\Mon)$ is convex in $(\theta(\Mon),<)$. 
Let $(a_3,a_2,a_1)$ be a $\mathbf p_l$-triangle over $A$: 

(3) \  $\mathcal D_p(a_1)<\mathcal D_p(a_2) <\mathcal D_p(a_3)$,  \ and   \ $a_1\dep_{Aa_3} a_2$.

\noindent 
By Lemma \ref{Lemma_forking_on_D_Pi}, the classes $\mathcal D_{\Pi_A}(a_1),\mathcal D_{\Pi_A}(a_2)$, and  $\mathcal D_{\Pi_A}(a_3)$ are convex; they are disjoint because $a_1,a_2$, and $a_3$ are pairwise independent over $A$. Therefore, $\mathcal D_{\Pi_A}(a_1)<\mathcal D_{\Pi_A}(a_2) <\mathcal D_{\Pi_A}(a_3)$.
Now, we claim that a triple $(a_1,a_2,a_3)$ can be found to also satisfy the following:

(4) \ $\tp(a_1/Aa_2a_3)$ is isolated.

\noindent
To prove (4), start with a triple $(a_1',a_2,a_3)$ satisfying $\mathcal D_{\Pi_A}(a_1')<\mathcal D_{\Pi_A}(a_2) <\mathcal D_{\Pi_A}(a_3)$ and $a_1'\dep_{Aa_3} a_2$. Put $q=(\mathbf p_l)_{\restriction a_3A}$; clearly, $q(\Mon)$ is an initial part of $p(\Mon)$ and $a_1',a_2\in q(\Mon)$. Note that $q(\Mon)$ is a convex subset of $(\theta(\Mon),<)$, as $p(\Mon)$ is. Choose a formula $\phi(x,y)\in\tp(a_1'a_2/Aa_3)$ that implies $x<y\land \theta(x)\land \theta(y)$ and witnesses $a_1'\dep_{Aa_3} a_2$: 
$\phi(x,a_2)$ is $q$-bounded. Then $\phi(\Mon,a_2)\subseteq \theta(\Mon)$ implies that $\phi(\Mon,a_2)$ is a $q$-bounded subset of $(\theta(\Mon), <)$, so since $q(\Mon)$ is a convex subset of $(\theta(\Mon),<)$, we derive $\phi(\Mon,a_2)\subseteq q(\Mon)$.

Let $a_1$ be an $\mathcal S$-minimal element of $\phi(\Mon,a_2)$ with $\tp(a_1/Aa_2a_3)$ isolated; it exists by Lemma \ref{Lemma_S_minimal_exist}. Clearly, $a_1\models q$ and the triple $(a_1,a_2,a_3)$ satisfy condition (4). We will prove that (3) is also satisfied. 
 Choose $b\in\Pi_A(\Mon)$ to witness the $\mathcal S$-minimality of $a_1$ in $\phi(\Mon,a_2)$: $b< \phi(\Mon,a_2)$ and $a_1\in S_b$.
By Lemma \ref{Lemma_forking_on_D_Pi}(b), $a_1\in S_b$ implies $\mathcal D_{\Pi_A}(a_1)=\mathcal D_{\Pi_A}(b)$. In addition, $b<\phi(\Mon,a_2)$ and $\models \phi(a_1',a_2)$ imply $b\leqslant a_1'$, which together with $\mathcal D_{\Pi_A}(a_1')<\mathcal D_{\Pi_A}(a_2)$ implies $\mathcal D_{\Pi_A}(b)<\mathcal D_{\Pi_A}(a_2)$. Hence $\mathcal D_{\Pi_A}(a_1)<\mathcal D_{\Pi_A}(a_2)<\mathcal D_{\Pi_A}(a_3)$. As $\mathcal D_p(a_i)\subseteq \mathcal D_{\Pi_A}(a_i)$ clearly holds for $i=1,2,3$, we conclude $\mathcal D_{p}(a_1)<\mathcal D_{p}(a_2)<\mathcal D_{p}(a_3)$. Since $a_1\in\phi(\Mon,a_2)$ and since the formula $\phi(x,a_2)$ has been chosen to witness $a_1'\dep_{Aa_3} a_2$, we conclude $a_1\dep_{Aa_3} a_2$. Therefore, the triple $(a_1,a_2,a_3)$ satisfies conditions (3) and (4).

\smallskip
From now on assume that $(a_1,a_2,a_3)$ satisfies conditions (3) and (4). Before stating the next claim, recall from the proof of (4) that the type $q=\tp(a_1/Aa_3)=\tp(a_2/Aa_3)$ is convex and that $q(\Mon)$ is convex in $(\theta(\Mon), <)$. We claim:

(5) \ $\tp(a_2/Aa_1a_3)$ is isolated.

\noindent 
We know  that $a_2\dep_{Aa_3} a_1$ holds  and that the type $\tp(a_1/Aa_2a_3)$ is isolated,  so there is a formula  $\phi'(x,y)\in\tp(a_1,a_2/Aa_3)$ implying $x<y\land \theta(x)\land \theta(y)$ such that  $\phi'(a_1,x)$ is $q$-bounded and $\phi'(x,a_2)$ isolates $\tp(a_1/Aa_2a_3)$. 
As $\phi'(a_1,\Mon)$ is a $q$-bounded subset of $\theta(\Mon)$ and $q(\Mon)$ is convex in $(\theta(\Mon),<)$, we conclude $\phi'(a_1,\Mon)\subseteq q(\Mon)$. 
Choose $a_2'\in\phi'(a_1,\Mon)$ such that $\tp(a_2'/Aa_1a_3)$ is isolated; then $a_2\in q(\Mon)$. Now, both $a_2$ and $a_2'$ realize the type $q\in S_1(Aa_3)$, so $a_2\equiv a_2'\,(Aa_3)$. Since the formula $\phi'(x,a_2)$ isolates a complete type in $S_1(Aa_2a_3)$, the formula $\phi'(x,a_2')$ isolates the $a_3A$-conjugate of that type in $S_1(Aa_2'a_3)$. Then $\models \phi'(a_1,a_2)\land \phi'(a_1,a_2')$ implies $a_1a_2\equiv a_1a_2'\,(Aa_3)$. Since $\tp(a_2'/Aa_1a_3)$ is isolated, so is $\tp(a_2/Aa_1a_3)$; this proves (5).

\smallskip 
Let $\psi(a_1,y,a_3)$ be a formula that isolates $\tp(a_2/Aa_1a_3)$ and let $M$ be a countable model prime over $Aa_1a_3$. We claim:

(6) \ The set $\{\mathcal D_p(x)\mid x\in p(M)\}$ is densely ordered by $<$.

\noindent
We need to prove that for all $b_1,b_3\in p(M)$ satisfying $\mathcal D_p(b_1)<\mathcal D_p(b_3)$ there exists a $b_2\in M$ with $\mathcal D_p(b_1)<\mathcal D_p(b_2)<\mathcal D_p(b_3)$.  Note that $b_1b_3\equiv a_1a_3\,(A)$, so the formula $\psi(b_1,y,b_3)$ isolates a complete type in $S_1(Ab_1b_3)$; that type is realized by some $b_2\in M$. Then $a_1a_2a_3\equiv b_1b_2b_3\, (A)$ and, in particular, $\mathcal D_p(b_1)<\mathcal D_p(b_2)<\mathcal D_p(b_3)$. This proves (6). 

\smallskip 
Let  $(b_i\mid i\in\mathbb Q)$ be an increasing sequence of elements of $p(M)$ with $\mathcal D_p(b_i)<\mathcal D_p(b_j)$ for all rational numbers $i<j$. As $M$ is prime over $Aa_1a_3$, for each $i\in\mathbb Q$ the type $\tp(b_i/Aa_1a_3)$ is isolated, by $\phi_i(x)$ say. By Lemma \ref{Lemma_S_minimal_exist}(c) we know $\phi_i(\Mon)\subseteq \mathcal D_p(b_i)$. This, together with $\mathcal D_p(b_i)<\mathcal D_p(b_j)$, implies $\phi_i(\Mon)<\phi_j(\Mon)$. It is easy to see that for each irrational number $r$ the following set:
$$p_r(x):=\{\phi_i(\Mon)<x\mid i\in\mathbb Q \land i<r\}\cup \{x<\phi_j(\Mon)\mid j\in\mathbb Q\land r<j\}$$
is consistent. Our final claim in this proof is:

(7) \ $p_r(x)\cup p_s(x)$ is inconsistent for all irrational numbers $r\neq s$.

\noindent
To prove the claim, suppose that $r<s$ are irrational. Choose $i\in \mathbb Q$ with $r<i<s$. Then $(x<\phi_i(\Mon))\in p_r$ and $(\phi_i(\Mon)<x)\in p_s$, so $p_r(x)\cup p_s(x)$ is inconsistent.
We have found $2^{\aleph_0}$ pairwise inconsistent partial types over $Aa_1a_3$, so $T$ is not small. Contradiction.
\end{proof}

\begin{Corollary}\label{Corollary_Ctype_trivial} ($T$ small) \
Every $\mathcal C$-type over a finite domain is trivial.
\end{Corollary}
\begin{proof}
Suppose that $A$ is finite and that $p\in S_1(A)$ is a $\mathcal C$-type. Since $T$ is small, there is a convex type $q$ that extends the same interval type as $p$; clearly, $p\nwor q$. By Lemma \ref{Lemma_every_convexshift_istrivial}, $q$ is trivial, so by Theorem \ref{Theorem_nwor preserves triviality}(a), $p\nwor q$ implies that $p$ is also trivial. 
\end{proof}

\begin{Lemma}
\label{Lema_S_Pi_isfinite} ($I(\aleph_0,T)<2^{\aleph_0}$)
Every type in $S_{\Pi}(\emptyset)$ is convex and simple. $S_{\Pi}(\emptyset)$ is a finite set.
\end{Lemma}
\begin{proof}Clearly, $T$ is small.
By Lemma \ref{Lemma_Pi_basic} the set $S_{\Pi}(\emptyset)$ consists of all completions of the interval type $\Pi(x)$, so $S_{\Pi}(\emptyset)$ is a closed subspace of $S_1(\emptyset)$. Corollary \ref{Corollary_Ctype_trivial} guarantees that each $p\in S_{\Pi}(\emptyset)$ is trivial, and then Theorem \ref{Theorem trivial implies simple and convex} implies that such $p$ is both convex and simple. 
By Corollary \ref{Corollary_convex_iff_isolated_in_IT}, convex types are isolated points of $S_{\Pi}(\emptyset)$, so $S_{\Pi}(\emptyset)$ is a discrete closed subspace of $S_1(\emptyset)$; $S_{\Pi}(\emptyset)$ is finite. 
\end{proof}

\begin{Lemma}[$I(\aleph_0,T)<2^{\aleph_0}$]\label{Lema_forking_reldef} 
$\mathcal D_{\Pi}$ is a relatively $0$-definable equivalence relation on $\Pi(\Mon)$. 
\end{Lemma}
\begin{proof}
We need to prove that the set $\mathcal D_{\Pi}(a)=\{x\in\Pi(\Mon)\mid x\dep a\}$ is $a$-definable for all $a\in \Pi(\Mon)$.
By Lemma \ref{Lema_S_Pi_isfinite} every $p\in S_{\Pi}(\emptyset)$ is simple, so  there exists a definable convex equivalence relation $E_p$ on $\Mon$ such that $E_{p\restriction p(\Mon)}=\mathcal D_p$. For each $q\in S_{\Pi}(\emptyset)$  with $q\nfor \tp(a)$ choose an element $a_q\in q(\Mon)$ with $a_q\dep a$. Note that the class $[a_q]_{E_q}$ contains exactly those realizations of $q$ that fork with $a$, so $\mathcal D_{\Pi}(a)=\bigcup \{[a_q]_{E_q}\mid q\nfor \tp(a), q\in S_{\Pi}(\emptyset)\}$; given that $S_{\Pi}(\emptyset)$ is finite, as proved in Lemma \ref{Lema_S_Pi_isfinite}, it follows that $D_{\Pi}(a)$ is definable; since it is $\Aut_a(\Mon)$-invariant, it is $a$-definable. 
\end{proof}

\begin{proof}[Proof of Theorem \ref{Theorem1_shift}]
We need to prove $I(\aleph_0,T)=2^{\aleph_0}$. By way of contradiction,  assume $I(\aleph_0,T)<2^{\aleph_0}$. Then Lemma \ref{Lema_forking_reldef} applies: $\mathcal D_{\Pi}$ is a convex relatively $0$-deinable equivalence relation on $\Pi(\Mon)$.  
Since $\Pi(\Mon)$ is a convex subset of $\Mon$, there is a $0$-deinable convex equivalence relation $E$ on $\Mon$ such that $E_{\restriction \Pi(\Mon)}=\mathcal D_{\Pi}$. Let $a\in \Pi(\Mon)$. Then $[a]_E=\mathcal D_{\Pi}(a)$ because $\mathcal D_{\Pi}(a)$ is a $\Pi$-bounded subset of $\Pi(\Mon)$.
By Lemma \ref{Lemma_forking_on_D_Pi}, we know that
$S_x\subset \mathcal D_{\Pi}(x)$ holds for all $x\in [a]_E$, so \ $\models \forall x(x\in [a]_E \rightarrow  S_x \subseteq [a]_E)$. 

Since $\tp(a)$ is a $\mathcal C$-type, \ $\models (\forall x)(x\in [c]_E \rightarrow  S_x \subseteq [c]_E)$ holds for some $c\in\mathcal C$. 
Then $S^{\omega}_c\subseteq [c]_E$. By Lemma \ref{Lemma_shift_properties},  $S_c$ is an $\mathcal S$-shift, so the sequence $(S_c^n\mid n\in\mathbb N\}$ is strictly increasing. By compactness there exists an $b\in[c]_E$  such that $S_c^n<b$ holds for all $n\in\mathbb N$, that is, $S^{\omega}_c<b$. By Lemma \ref{Lemma_shift_properties} the set $S^{\omega}_c$ is a final part of $\mathcal C$, so we conclude $\mathcal C<b$. Then, since the class $[c]_E$ is convex, there exists $b'<b$ with $b'\in[c]_E\cap \Pi(\Mon)$. By our choice of $E$, the class $[b']_E=\mathcal D_\Pi(b')$ is disjoint from $\mathcal C$. Hence $c\notin[b']_E$.
Contradiction. The proof of Theorem \ref{Theorem1_shift} is now complete. 
\end{proof}

It is routine to generalize Theorem \ref{Theorem1_shift} to the local weakly quasi-o-minimal context as follows.

\begin{Theorem}
Suppose that for some finite set of parameters $A$, there exist a definable weakly quasi-o-minimal order over $A$ possessing an $A$-definable family of semiintervals that includes a shift. Then $I(T,\aleph_0)=2^{\aleph_0}$. 
\end{Theorem}

\section{Examples}\label{Section_Examples}

In \cite[Section 6]{Herwig}, the authors
introduce a construction that, for a given relational structure $\mathcal P$ expanding the ordered rationales, produces a densely ordered structure $\mathcal N[\mathcal P]$ (in a different language). They show that if $\mathcal P$ is $\aleph_0$-categorical weakly o-minimal and 2-indiscernible, then $\mathcal N[\mathcal P]$ also has these properties and is isomorphic to some expansion of $\mathcal P$. Starting with $(\mathbb Q,<)$ and iterating the process of obtaining expansions in that way $\omega$ times, they produce an $\aleph_0$-categorical weakly o-minimal structure whose theory fails to be $n$-ary for any $n\in\mathbb N$. 

A careful inspection of the basic construction leads to the conclusion that for any weakly o-minimal $\mathcal P$, the structure $\mathcal N[\mathcal P]$ remains weakly o-minimal and 1-indiscernible, yet it is also atomic over a copy of $\mathcal P$ that can be defined in $\mathcal N[\mathcal P]^{eq}$; note that this implies
$I(\Th(\mathcal N[\mathcal P]),\aleph_0)=I(\Th(\mathcal P),\aleph_0)$. 
Using this, we find examples of weakly o-minimal theories $T_1$ and $T_2$ with 3 countable models that satisfy conditions \ref{C3} and \ref{C4}, respectively; recall that each of them implies $I(T,\aleph_0)=2^{\aleph_0}$ if $T$ is binary and weakly quasi-o-minimal.  
Our examples are $T_1=\Th(\mathcal N[\mathcal P_1])$ and $T_2=\Th(\mathcal N[\mathcal P_2])$, where $\mathcal P_1$ and $\mathcal P_2$ are Ehrenfeucht's examples, obtained by adding an increasing or decreasing $\omega$-sequence of constants to $(\mathbb Q,<)$. $T_1$ and $T_2$ are also examples of weakly o-minimal theories with 3 countable models that are not almost $\aleph_0$-categorical; recall that $T$ is almost $\aleph_0$-categorical if for all $n$ and all types $p_1(x_1),\ldots, p_n(x_n)\in S_1(T)$ the type $\bigcup_{1\leqslant i\leqslant n}p_i(x_i)$ has finitely many extensions in $S_n(T)$.  
 
The domain of $\mathcal N[\mathcal P]$ can be viewed as a set of branches of a dense meet-tree, which is also 0-definable in $\mathcal N[\mathcal P]^{eq}$ and whose levels are 0-definably identified with elements of the copy of $\mathcal P$ in $\mathcal N[\mathcal P]^{eq}$ by the level function. 
We will reinterpret the original construction by considering 3-sorted structures, with different sorts for levels, elements, and branches of the tree. That environment allows for a relatively simple description of types of $\Th(\mathcal N[\mathcal P])$ in terms of those of $\Th(\mathcal P)$. Using this, we show that $T_1$ and $T_2$ have a unique 1-type over $\emptyset$ that is nontrivial, but all nonalgebraic 1-types over a finite nonempty domain are trivial.

By a {\it meet-tree} we mean a strictly ordered structure $(P,\triangleleft)$ such that the set of predecessors of each element is totally ordered by $\triangleleft$, and any pair of elements $a,b$ of $P$ has the greatest lower bound which is also called {\it the meet} of $a$ and $b$; we will denote it by $\mm(a,b)$. A {\it subtree} of $P$ is any subset that is closed for meets; a {\it branch} of $P$ is any maximal totally ordered subset. Maximal elements are called {\it leaves}. 
For $a\in P$  
the set $C_a=\{x\in P\mid a\trianglelefteq x\}$ is {\it the closed cone} centered at $a$. The relation $\sim_a$ defined by $x\sim_a y\Leftrightarrow a\triangleleft \mm(x,y)$ is an equivalence relation on the set $C_a\smallsetminus \{a\}$; $\sim_a$-classes are called {\it open cones} centered at $a$, and $\mathcal C_a$ denotes the set of all of them. Note that each cone, open or closed, is a subtree.  

A {\it dense meet-tree} has two additional properties: every branch is a dense endless linear order, and for all $a\in P$ there are infinitely many open cones centered at $a$. 
$(P,\triangleleft,<)$ is an {\it ordered tree} if $(P,\triangleleft)$ is a meet-tree and $<$ is a linear order that extends $\triangleleft$ such that every closed cone is convex in $(P,<)$; in that case, every open cone is also convex.

The basic structure that we will consider is the following:
$$\mathcal M=(\mathcal L(M), \mathcal T(M),\mathcal B(M), \triangleleft,\mm,<,<_{\mathcal L},\Lev, \ell,R),\ \text{ where:}$$
\begin{enumerate}[$\bullet$,left=1em,labelsep=.7em] 
\item $\mathcal B(M)$ is the set of all functions $a:\mathbb Q\to \mathbb Q$ with finite support ($\{x\in\mathbb Q\mid a(x)\neq 0\}$ is finite); 
\item $\mathcal T(M)$ is the set of all functions with finite support that map $(-\infty,r)$ to $\mathbb Q$ for some $r\in \mathbb Q$, that is the set of all $a_{\restriction(-\infty,r)}$ for $a\in\mathcal B(M)$ and $r\in\mathbb Q$; 
\item $\triangleleft$ denotes the strict inclusion on pairs of elements of $\mathcal B(M)\cup\mathcal T(M)$. 

  \item Notice that $(\mathcal B(M)\cup\mathcal T(M),\triangleleft)$ is a meet-tree; where the meet is defined by 
  \[\mm(a,b)=\bigcup 
\{a_{\restriction(-\infty,r)}\colon\ a_{\restriction(-\infty,r)}=b_{\restriction(-\infty,r)}\ \land\ r\in\mathbb Q\};\]
\item $<$ is the lexicographic order on $\mathcal B(M)\cup\mathcal T(M)$ that extends $\triangleleft$: $a<b$ if and only if $a\triangleleft b$ or $a(r)<b(r)$ where $r\in\mathbb Q$ is the least number such that $a(r)\neq b(r)$;
\item $\mathcal L(M)=\mathbb Q\cup\{\infty_{\mathcal L}\}$ and $<_{\mathcal L}$ is a natural order on $\mathbb Q$ with added infinitely large element $\infty_{\mathcal L}$; 
\item $\Lev:\mathcal B(M)\cup\mathcal T(M)\to\mathcal L(M)$ is the level function: $\Lev(a)=\sup(\mathrm{dom}(a))$ for $a\in \mathcal T(M)$ and $\Lev(a)=\infty_{\mathcal L}$ for $a\in \mathcal B(M)$.
\end{enumerate}
The binary function $\ell$ and the quaternary relation $R$ are defined on $\mathcal B(M)$ in terms of the already defined relations: 
\begin{enumerate}[$\bullet$,left=1em,labelsep=.7em] 
\item $\ell:\mathcal B(M)^2\to \mathcal L(M)$ is defined by $\ell(x,y):=\Lev(\mm(x,y))$;
\item $R\subset \mathcal B(M)^{4}$ is defined by $R(x,y;z,t)\Leftrightarrow \ell(x,y)<_{\mathcal L}\ell(z,t)$. 
\end{enumerate}

\begin{Observation}\label{Obs1}
\begin{enumerate}[(a),left=0em,labelsep=.7em] 
\item $(\mathcal T(M),\triangleleft)$ is a dense meet-tree.
    \item $\mathcal B(M)$ is a $<$-dense set of leafs of the tree $(\mathcal B(M)\cup\mathcal T(M),\triangleleft)$: for all $a_1< a_2$ there is a $b\in\mathcal B(M)$ such that $a_1<b<a_2$; if, in addition, $a_1\triangleleft a_2$, then $a_1\triangleleft \mm(a_2,b)\triangleleft a_2$.
    \item $(\mathcal B(M)\cup\mathcal T(M),\triangleleft,<)$ is an ordered tree: It is easy to see that each closed cone $C_a=\{x\mid a\trianglelefteq x\}$ is a $<$-convex subset of $\mathcal B(M)\cup\mathcal T(M)$ with $a=\min C_a$. If $a\in \mathcal T(M)$, then $(C_a\smallsetminus \{a\},<)$ is a dense endless linear order and $C_a^{\mathcal B}$ its dense subset, where here and later, $C_a^{\mathcal B}$ denotes $C_a\cap \mathcal B(M)$.
    
    \item For all $a,b\in \mathcal B(M)\cup\mathcal T(M)$ we have: 
        \begin{enumerate}[$\bullet$,left=0em,labelsep=.7em]
        \item  $C_a\subset C_b$ \ if and only if $C_a^{\mathcal B}\subset C_b^{\mathcal B}$ \ \ and \ \ $C_a= C_b$ \ if and only if $C_a^{\mathcal B}= C_b^{\mathcal B}$;
        \item   $C_a<C_b$ \ if and only if $C_a^{\mathcal B}< C_b^{\mathcal B}$;
        \item  $a\triangleleft b$ \ if and only if \ $C_a^{\mathcal B}\supset C_b^{\mathcal B}$.
        \end{enumerate} 
    \item For all $a\in \mathcal T(M)$, each open cone centered at $a$ is of the form $C_a(q):=\{x\in\mathcal T(M)\cup\mathcal B(M)\mid a\triangleleft x \land x(r)=q\}$, where $\dom(a)=(-\infty,r)_{\mathbb Q}$; note that each $C_a(q)$ is $<$-convex and: $C_a(q)< C_a(q')$ if and only if $q<q'$. Thus, the set $\mathcal C_a=\{C_a(q)\mid q\in\mathbb Q\}$ is ordered by $<$, and $q\mapsto C_a(q)$ is an isomorphism of ordered rationales and $(\mathcal C_a,<)$. In particular,  $(\mathcal C_a,<)$ is a dense endless order. 
\end{enumerate}
\end{Observation}

Now, consider the structure $\mathcal N=(\mathcal B(M),<_{\mathcal B},R)$, where $<_{\mathcal B}$ is the restriction of $<$ to $\mathcal B(M)$. Note that $M= \dcl_{\mathcal M}(\mathcal B(M))$. In the next lemma, we show that $\mathcal M$ is quantifier-free interpretable in $\mathcal N$, which means that each sort of $\mathcal M$ is identified with $D/_\sim$, where $D$ and $\sim$ are a quantifier-free definable set and an equivalence relation of $\mathcal N$, as is any basic relation defined of $\mathcal M$.  

\begin{Lemma}\label{Lemma_example_MisoNeq} 
$\mathcal M$ is isomorphic to a structure that is 0-definable in $\mathcal N^{eq}$ and, moreover, is defined by quantifier-free formulas in the language of $\mathcal N$.
\end{Lemma}
\begin{proof} Note that in $\mathcal M$, the formulae $\ell(x,y)\ \epsilon \ \ell(z,t)$, where $\epsilon\in\{<_{\mathcal L},\leqslant_{\mathcal L},=\}$, and where $x,y,z,t$ range over $\mathcal B(M)$, are quantifier-free expressible in the language of $\mathcal N$.
So are the following:
\begin{enumerate}[$\bullet$,left=1em,labelsep=.7em]
    \item  $z\in C_{\mm(x,y)}^{\mathcal B} \Leftrightarrow \ell(x,y)\leqslant_{\mathcal L} \ell(x,z)$;
    \item  $C_{\mm(x,y)}^{\mathcal B}<_{\mathcal B}z\Leftrightarrow (x<_{\mathcal B}z\land \ell(x,z)<_{\mathcal L}\ell(x,y)$;
    \item  $C_{\mm(x,y)}^{\mathcal B}\subseteq C_{\mm(z,t)}^{\mathcal B} \Leftrightarrow (x\in C_{\mm(z,t)}^{\mathcal B}\land y\in C_{\mm(z,t)}^{\mathcal B})$;
    \item  $C_{\mm(x,y)}^{\mathcal B}\cap C_{\mm(z,t)}^{\mathcal B}=\emptyset \Leftrightarrow \lnot (C_{\mm(x,y)}^{\mathcal B}\subseteq C_{\mm(z,t)}^{\mathcal B}\vee C_{\mm(x,y)}^{\mathcal B}\supseteq C_{\mm(z,t)}^{\mathcal B})$;
    \item  $C_{\mm(x,y)}^{\mathcal B}<_{\mathcal B}C_{\mm(z,t)}^{\mathcal B}\Leftrightarrow (C_{\mm(x,y)}^{\mathcal B}\cap C_{\mm(z,t)}^{\mathcal B}=\emptyset \land C_{\mm(x,y)}^{\mathcal B}<_{\mathcal B}z)$.
\end{enumerate}
Now, we find an adequate copy of $\mathcal L(M)$ in $\mathcal N^{eq}$. 
Let $\sim$ be the relation defined by $\ell(x_1,y_1)=\ell(x_2,y_2)$. We can naturally identify $\mathcal B(M)^2/_\sim$ with $\mathcal L(M)$ and the quotient map $\mathcal B(M)^2\to\mathcal B(M)^2/_\sim$ with the function $l$, so that $<_{\mathcal L}$ is defined by $[(x_1,y_1)]_{\sim}<_{\mathcal L}[(x_2,y_2)]_{\sim}\Leftrightarrow R(x_1,y_1;x_2,y_2)$. 

$\mathcal T(M)$ is naturally identified with the set $\{(x,y)\in \mathcal B(M)^2\mid x\neq  y\}/_{\sim_1}$, where $\sim_1$ is defined by: 
$$(x,y)\sim_1(z,t)  \ \mbox{ if and only if } \ C_{\mm(x,y)}^{\mathcal B}= C_{\mm(z,t)}^{\mathcal B}.$$ 
$\triangleleft$ is defined as follows: For $z\in \mathcal B(M)$,  $[(x,y)]_{\sim_1}\triangleleft z$ is defined by $z\in C_{\mm(x,y)}^{\mathcal B}$, and  
\begin{center}$[(x,y)]_{\sim_1} \triangleleft [(z,t)]_{\sim_1}$ if and only if  $C_{\mm(x,y)}^{\mathcal B} \supset C_{\mm(z,t)}^{\mathcal B}$.
\end{center}
The function $m$ is defined as follows:  For distinct $x,y\in\mathcal B(M)$ and distinct $z,t\in\mathcal B(M)$, define:  $\mm(x,x)= x$, $\mm(x,y)=[(x,y)]_{\sim_1}$, and:
\[ \mm([(x,y)]_{\sim_1},z)= \mm(z,[(x,y)]_{\sim_1}):= 
    \begin{cases*}
      [(x,y)]_{\sim_1} & if $[(x,y)]_{\sim_1} \triangleleft z$ \\
      [(x,z)]_{\sim_1} & if $[(x,y)]_{\sim_1}$ and $z$ are $\triangleleft$-incomparable; 
    \end{cases*}
\]
\[ 
    \mm([(x,y)]_{\sim_1},[(z,t)]_{\sim_1}) =
    \begin{cases*}
      [(x,y)]_{\sim_1} & if $[(x,y)]_{\sim_1} \trianglelefteq [(z,t)]_{\sim_1}$ \\
       [(z,t)]_{\sim_1} & if $[(z,t)]_{\sim_1} \triangleleft [(x,y)]_{\sim_1}$ \\
      [(x,z)]_{\sim_1} & if $[(x,y)]_{\sim_1}$ and $ [(z,t)]_{\sim_1}$ are $\triangleleft$-incomparable. 
    \end{cases*}
\]
$\Lev$ is defined by: $\Lev(x)=[(x,x)]_\sim$ for $x\in \mathcal B(M)$ and $\Lev([(x,y)]_{\sim_1})=[(x,y)]_{\sim}$ for distinct $x,y\in\mathcal B(M)$. 
To find an adequate copy of the order $<$, it suffices to observe that for all $u,v\in \mathcal T(M)\cup\mathcal B(M)$:
\begin{center} $u<v$ \ is and only if \  $u\neq v\land (u\triangleleft v\vee C_u^{\mathcal B}<_{\mathcal B} C_v^{\mathcal B})$.
\end{center}
This completes the proof.
\end{proof}

\noindent{\bf Convention.} \  For the rest of this section, $\mathcal P$ denotes a 
{\em saturated} expansion of $(\mathcal L(M), <_{\mathcal L},\infty_{\mathcal L})$ obtained by adding a countable set of new relations.   
$\mathcal M[\mathcal P]$ is the expansion of $\mathcal M$ obtained by adding all the new relations of $\mathcal P$ to the $\mathcal L$-sort.

\begin{Lemma}\label{Lemma_Aut_L(M)_properties}  
\begin{enumerate}[(a),left=0em,labelsep=.7em]
    \item If $a,b\in \mathcal B(M)\cup \mathcal T(M)$ are such that $\Lev(a)=\Lev(b)$, then there is an automorphism $f\in \Aut_{\mathcal L(M)}(\mathcal M[\mathcal P])$ with $f(a)=b$. 
    \item Suppose that $C$ is a cone (open or closed) centered at $a$,  $f\in\Aut_{\mathcal L(M)}(\mathcal M[\mathcal P])$, and $f(C)=C$. Then the mapping defined by $F(x)=f(x)$ for $x\in C$ and $F(x)=x$ for $x\notin C$ belongs to $\Aut_{\mathcal L(M)}(\mathcal M[\mathcal P])$.
    \item For $a\in\mathcal T(M)$, every automorphism of $(\mathcal C_a,<)$ has an extension in $\Aut_{\mathcal L(M)\cup\{a\}}(\mathcal M[\mathcal P])$.
    \item Every automorphism of $\mathcal P$ has an extension in $\Aut(\mathcal M[\mathcal P])$.
\end{enumerate}
\end{Lemma}
\begin{proof}
(a) Suppose that $\Lev(a)=\Lev(b)=r$. For each $x\in \mathcal B(M)\cup \mathcal T(M)$ let $f(x):\dom(x)\to \mathbb Q$ be defined by: $f(x)(t)=x(t)$ for $t\geqslant r$ and $f(x)(t)=x(t)+b(t)-a(t)$ for $t<r$. It is easy to see that $f\in\Aut_{\mathcal L(M)}(\mathcal M[\mathcal P])$ and $f(a)=b$.

(b) Easy.

(c) Let $f$ be an automorphism of $(\mathcal C_a,<)$ and let $r=\Lev(a)$. Observation \ref{Obs1}(e) gives $\mathcal C_a=\{C_a(q)\mid q\in\mathbb Q\}$, where $x\in C_a({x(r)})$ for all $x\triangleright a$, and the isomorphism $g$ between the ordered rationales and $(\mathcal C_a,<)$ defined by $q\mapsto C_a(q)$. For each $x\triangleright a$ let $F(x)$ denote the function $x$ with the value at $r$ redefined by $F(x)(r)=g^{-1}fg(x(r))$; for all other $x$ let $F(x)=x$. Then $F\in \Aut_{\mathcal L(M)}(\mathcal M[\mathcal P])$ and $F(C_q)=f(C_q)$ for all $q\in\mathbb Q$. 

(d) Let $f\in \Aut_{\mathcal L(M)}(\mathcal P)$. It is then fairly easy to check that the map $\hat f$ extending $f$ by setting $\dom(\hat f(a))=f[\dom(a)]$ and $\hat f(a)(f(q))=a(q)$ for all $a\in\mathcal B(M)\cup\mathcal T(M)$ and $q\in\dom(a)$ is an automorphism of $\mathcal M[\mathcal P]$.
\end{proof}

For a tuple $\bar a=(a_0,\dots,a_n)$ of elements of $\mathcal B(M)\cup \mathcal T(M)$ define: 
\begin{center}$Tr(\bar a)=( \mm(a_i,a_j) \mid 0\leqslant i\leqslant j\leqslant n)$ \ 
 and  \ $\ell(\bar a)=(\ell(a_i,a_j)\mid 0\leqslant i\leqslant j\leqslant n)$.
\end{center}
Note that $Tr(\bar a)$ is the subtree of $\mathcal B(M)\cup \mathcal T(M)$ generated by $\bar a$.
Let $\bar x=(x_0,\dots,x_n)$, and let $\bar y=(y_{ij}\mid 0\leqslant i\leqslant j\leqslant n)$. Let $\Delta_{\bar a}(\bar x,\bar y)$ be the conjunction of the following formulae:
\begin{enumerate}[label=$\bullet$,left=1em,labelsep=0.7em]
    \item $\ell(x_i,x_j)=y_{ij}$, \ for all $0\leqslant i\leqslant j\leqslant n$,
    \item $\mm(x_i,x_j)=\mm(x_k,x_l)$, \ for all $0\leqslant i,j,k,l\leqslant n$ such that $\mm(a_i,a_j)=\mm(a_k,a_l)$, and
    \item $\mm(x_i,x_j)<\mm(x_k,x_l)$, \ for all $0\leqslant i,j,k,l\leqslant n$ such that $\mm(a_i,a_j)<\mm(a_k,a_l)$.
\end{enumerate}
Note that the second and third item describe the isomorphism type of $(Tr(\bar a),\mm,<)$ viewed as an ordered subtree of $(\mathcal B(M)\cup\mathcal T(M),\mm,<)$; in particular, there are only finitely many formulas $\Delta_{\bar a}(\bar x,\bar y)$ for a fixed $n=|\bar a|$. Also note that $\mathcal M\models\Delta_{\bar a}(\bar b,\bar c)$ implies $\ell(\bar b)=\bar c$ and $\Delta_{\bar a}=\Delta_{\bar b}$. 
 
\begin{Lemma}\label{Lemma_example_qfmain}
\begin{enumerate}[(a),left=0em,labelsep=0.7em]
    \item Tuples $\bar a$ and $\bar b$ from $\mathcal B(M)\cup\mathcal T(M)$ are conjugated by some element of $\Aut_{\mathcal L(M)}(\mathcal M[\mathcal P])$ if and only if $\mathcal M[\mathcal P]\models\Delta_{\bar a}(\bar b,\ell(\bar a))$.
    \item For all $\mathcal M'\equiv \mathcal M[\mathcal P]$ and all $\bar a\in \mathcal B(M')\cup\mathcal T(M')$, $\Delta_{\bar a}(\bar x,\ell(\bar a))\vdash  \tp^{\mathcal M'}(\bar a/\mathcal L(M'))$. In particular, $\mathcal M'$ is atomic over $\mathcal L(M')$.
\item For all $\mathcal M'\equiv \mathcal M$ and all $\bar c\in \mathcal L(M')$, $\tp^{(\mathcal L(M'),<_{\mathcal L})}(\bar c)\vdash \tp^{\mathcal M'}(\bar c)$. 

    \item Let $\mathcal M_1,\mathcal M_2\equiv \mathcal M[\mathcal P]$ be countable. Then every isomorphism of $\mathcal L(M_1)$ and $\mathcal L(M_2)$, viewed as models of $\Th(\mathcal P)$, extends to an isomorphism of $\mathcal M_1$ and $\mathcal M_2$. 
    \item $\mathcal M[\mathcal P]$ is saturated.
    \item $\tp^{\mathcal P}_{\bar y,\bar z}(\ell(\bar a),\bar c)\cup\{\Delta_{\bar a}(\bar x,\bar y)\}\vdash \tp^{\mathcal M[\mathcal P]}_{\bar x,\bar z}(\bar a,\bar c)$ \  for all $\bar a\in\mathcal B(M)\cup\mathcal T(M)$ and all $\bar c\in \mathcal L(M)$.
    \item $\tp^{\mathcal P}(\bar c/\ell(\bar a))\vdash \tp^{\mathcal M[\mathcal P]}(\bar c/\bar a)$ \  for all $\bar a\in\mathcal B(M)\cup\mathcal T(M)$ and all $\bar c\in \mathcal L(M)$.
\end{enumerate} 
\end{Lemma} 
\begin{proof}
(a) 
It suffices to prove the claim in the case where $\bar a$ is a subtree of $\mathcal B(M)\cup\mathcal T(M)$. The left-to-right implication is clear and we prove the reverse by induction on $|\bar a|$. 
For $|\bar a|=1$, the conclusion follows by Lemma \ref{Lemma_Aut_L(M)_properties}(a). 

Assume that (a) holds for all subtrees of $\mathcal B(M)\cup\mathcal T(M)$ with fewer than $n$ elements. Let $\bar a$ be a subtree of size $n$ and suppose that $\mathcal M\models \Delta_{\bar a}(\bar b,\ell(\bar a))$.  
Let $a$ and $b$ be the roots of the trees $\bar a$ and $\bar b$, respectively. By Lemma \ref{Lemma_Aut_L(M)_properties}(a) some member of $\Aut_{\mathcal L(M)}(\mathcal M)$ moves $b$ to $a$, so replacing $\bar b$ by its image we can assume $a=b$. 

Let $C_1<C_2<\dots<C_k$ ($C_1'<C_2'<\dots<C_k'$) list all the open cones from $\mathcal C_{a}$ that contain elements of $\bar a$ ($\bar b$). By Lemma \ref{Lemma_Aut_L(M)_properties}(c) some member of $\Aut_{\mathcal L(M)\cup\{a\}}(\mathcal M)$ moves $(C_1,C_2,..,C_k)$ to $(C_1',C_2',\dots,C_k')$. Therefore, we can assume $C_i=C_i'$ ($1\leqslant i\leqslant k$). 
Notice that $C_i\cap \bar a$ and $C_i\cap\bar b$ are subtrees of $C_i$ of size less than $n$, 
and that $\mathcal M\models\Delta_{C_i\cap \bar a}(C_i\cap \bar b,\ell(c_i\cap \bar a))$. 
By the induction hypothesis, some $f_i\in\Aut_{\mathcal L(M)}(\mathcal M)$ moves $C_i\cap \bar b$ to $C_i\cap\bar a$. Then $f_i(a)=a$: To justify this, suppose that $a_0$ and $b_0$ are the roots of $C_i\cap \bar a$ and $C_i\cap\bar b$, respectively. Note that $f_i(a_0)=b_0$. Also note that $a\in \dcl(a_0,\Lev(a))$, as evidenced by the formula $x\triangleleft a_0\land \Lev(x)=\Lev(a)$. Then $\models f_i(a)\triangleleft b_0\land \Lev(f_i(a))=\Lev(a)$. Since the formula $ x\triangleleft b_0\land \Lev(x)=\Lev(a)$ witnesses $a\in \dcl(b_0,\Lev(a))$, we conclude $f_i(a)=a$. This together with $f_i(C_i\cap \bar a)=C_i\cap\bar b$ implies $f(C_i)=C_i$, so by Lemma \ref{Lemma_Aut_L(M)_properties}(b) we can assume that $f_i$ is the identity map outside of $C_i$. Then the composition $f_1\circ f_2\circ\dots\circ f_k$ belongs to $\Aut_{\mathcal L(M)}(\mathcal M)$ and moves $\bar a$ to $\bar b$, completing the proof.

(b) From (a) we immediately derive $\Delta_{\bar a}(\bar x,\ell(\bar a))\vdash  \tp^{\mathcal M[\mathcal P]}(\bar a/\mathcal L(M))$. This is expressible by a set of sentences of $\Th(\mathcal M[\mathcal P])$: For a fixed $n$, there are only finitely many distinct formulae $\Delta_{\bar a}(\bar x,\bar y)$, where $|\bar a|=n$; say, $\Delta_{\bar a_1},\dots,\Delta_{\bar a_n}$, hence $\mathcal M[\mathcal P]\models \forall \bar x
\,\bigvee_{i=1}^n\exists \bar y\,\Delta_{\bar a_i}(\bar x,\bar y)$. For each $i\leqslant n$ and for each formula $\psi(\bar x,\bar y,\bar z)$, where $\bar z$ are variables for the $\mathcal L$-sort, the sentence
\[\forall \bar x,\bar x',\bar y,\bar z\,(\Delta_{\bar a_i}(\bar x,\bar y)\land \Delta_{\bar a_i}(\bar x',\bar y) \rightarrow (\psi(\bar x,\bar y, \bar z) \leftrightarrow \psi(\bar x',\bar y,\bar z))\]
is satisfied in $\mathcal M[\mathcal P]$ and guaranties that for all $\bar a'\in \mathcal M[\mathcal P]$ realizing $\exists \bar y\, \Delta_{\bar a_i}(\bar x,\bar y)$, all realizations of $\Delta_{\bar a_i}(\bar x,\ell(\bar a'))$ have the same $\psi$-type over $\mathcal L(M)$ as $\bar a'$. The same holds for $\mathcal M'$. Since each  $n$-tuple of $\mathcal B(M')\cup \mathcal T(M')$ satisfies $\exists \bar y\, \Delta_{\bar a_i}(\bar x,\bar y)$ for some $i\leqslant n$, the desired conclusion follows.

(c) Suppose that $\bar c \equiv^{(\mathcal L(M),<_{\mathcal L})} \bar c'$ and let $f\in \Aut(\mathcal L(M),<_{\mathcal L})$ move $\bar c$ to $\bar c'$. By Lemma \ref{Lemma_Aut_L(M)_properties}(d), $f$ extends to an automorphism of $\mathcal M$. Therefore, $\tp^{(\mathcal L(M),<_{\mathcal L})}(\bar c)\vdash \tp^{\mathcal M}(\bar c)$. Since  $\tp^{(\mathcal L(M),<_{\mathcal L})}(\bar c)$ can be expressed by a single formula,  $\tp^{(\mathcal L(M'),<_{\mathcal L})}(\bar c)\vdash \tp^{\mathcal M'}(\bar c)$ holds for all $\mathcal M'\equiv \mathcal M$ and all $\bar c\in \mathcal L(M')$.

(d) Suppose that $\mathcal M_1,\mathcal M_2\models  \Th(\mathcal M[\mathcal P])$ and that $f:\mathcal L(M_1)\to \mathcal L(M_2)$ is an isomorphism of models of $\Th(\mathcal P)$. Now, consider the reducts of $\mathcal M_1$ and $\mathcal M_2$ to the language of $\mathcal M$. By (c), $f$ is a partial elementary mapping between the $\mathcal L$-sorts of these structures.  By (b), each of these structures is atomic over its $\mathcal L$-sort, so there is an isomorphism $\hat f:M_1\to M_2$ that extends $f$.
Since $\hat f$ preserves all the relations named by the language of $\mathcal P$, and all the relations named by the language of $\mathcal M$, $\hat f$ is an isomorphism of $\mathcal M_1$ and $\mathcal M_2$.

(e) First, we see that $\mathcal M[\mathcal P]$ is universal: every countable $\mathcal M'\equiv \mathcal M[\mathcal P]$ embeds into $\mathcal M[\mathcal P]$. Since $\mathcal P$ is saturated, it is routine to find a countable $\mathcal M''\succ\mathcal M'$ whose $\mathcal L$-sort, viewed as a model of $\Th(\mathcal P)$, is isomorphic to $\mathcal P$; let $f:\mathcal L(M)\to \mathcal L(M'')$ be an isomorphism. By (d), $f$ lifts to an isomorphism of $\mathcal M[\mathcal P]$ and $\mathcal M''$, so $\mathcal M'$ can be embedded into $\mathcal M[\mathcal P]$; $\mathcal M[\mathcal P]$ is universal.
In particular, the theory $\Th(\mathcal M[\mathcal P])$ is small, so it has a countable saturated model $\mathcal M_1$. Note that the $\mathcal L$-sort of $\mathcal M_1$, viewed as a model of $\Th(\mathcal P)$, is isomorphic to $\mathcal P$. By (d), this isomorphism extends to an isomorphism of $\mathcal M[\mathcal P]$ and $\mathcal M_1$. Therefore, $\mathcal M[\mathcal P]$ is saturated. 

(f)  Suppose that $\ell(\bar a')\bar c'\equiv^{\mathcal P} \ell(\bar a)\bar c$ and $\mathcal M[\mathcal P]\models \Delta_{\bar a}(\bar a',\bar \ell(\bar a'))$. We need to prove $\bar a\bar c\equiv^{\mathcal M[\mathcal P]} \bar a'\bar c'$. Let $f\in\Aut(\mathcal P)$ be such that $f(\ell(\bar a'),\bar c')=(\ell(\bar a),\bar c)$. By (d), $f$ extends to $\hat f\in \Aut(\mathcal M[\mathcal P])$. 
Denote $\bar a''=\hat f(\bar a')$. Then $\mathcal M[\mathcal P]\models \Delta_{\bar a}(\bar a'',\ell(\bar a))$, so by (b) we have 
$\bar a''\equiv \bar a\,(\mathcal L(M))$. In particular, $\bar a\bar c\equiv^{\mathcal M[\mathcal P]}\bar a''\bar c$; combining this with $\hat f(\bar a',\bar c')=(\bar a'',\bar c)$, we deduce $\bar a\bar c\equiv^{\mathcal M[\mathcal P]} \bar a'\bar c'$.

(g) Follows from (f).
\end{proof}

\begin{Corollary}\label{Corollary_M_alleph0_categorical}
 \begin{enumerate}[(a),left=0em,labelsep=0.7em]
    \item $\Th(\mathcal M)$ is $\aleph_0$-categorical and after naming $\infty_{\mathcal L}$ eliminates quantifiers.
    \item $\Th(\mathcal N)$ is $\aleph_0$-categorical, weakly o-minimal, and eliminates quantifiers. There is a unique type $p\in S_1(\Th(\mathcal N))$.
\end{enumerate}
 \end{Corollary}
\begin{proof}
(a) It suffices to show that for all $n$ there are finitely many types 
$\tp(\bar a,\bar c)$, where $\bar a\in\mathcal B(M)\cup\mathcal T(M)$, $\bar c\in \mathcal L(M)$, and $|\bar a\bar c|=n$,  and that each of these types is isolated by a quantifier-free formula.
Applying Lemma \ref{Lemma_example_qfmain}(f) to $\mathcal P= (\mathcal L(M),<_{\mathcal L},\infty_{\mathcal L})$, we have $\tp^{\mathcal P}_{\bar y,\bar z}(\ell(\bar a),\bar c)\cup\{\Delta_{\bar a}(\bar x,\bar y)\}\vdash \tp_{\bar x,\bar z}(\bar a,\bar c)$. Here, $\tp^{\mathcal P}_{\bar y,\bar z}(\ell(\bar a),\bar c)$ is isolated by a quantifier free formula in the language $\{<_{\mathcal L},\infty_{\mathcal L}\}$, and $\Delta_{\bar a}(\bar x,\bar y)$ is quantifier free. Since there are only finitely many possibilities (nonequivalent modulo $\Th(\mathcal M[\mathcal P])$) for these formulas, the desired conclusion follows. 

(b) First, we prove that $\mathrm{qftp}^\mathcal N(\bar a)\forces \tp^\mathcal N(\bar a)$ for all tuples $\bar a\in \mathcal B(M)$; clearly, this implies the $\aleph_0$-categoricity and elimination of quantifiers. So, suppose that $\bar a,\bar b\in\mathcal B(M)$ are such that $\bar a\equiv_{\mathrm{qf}}^\mathcal N\bar b$. Since by Lemma \ref{Lemma_example_MisoNeq}, $\mathcal M$ is quantifier free interpretable in $\mathcal N^{eq}$, we have $\bar a\equiv_{\mathrm{qf}}^\mathcal M\bar b$. Then $\bar a\equiv_{\mathrm{qf}}\bar b$ in the structure $\mathcal M$ with $\infty_{\mathcal L}$ named, so by (a) we conclude $\bar a\equiv^\mathcal M\bar b$. Therefore, $\bar a\equiv^\mathcal N\bar b$, as desired.

Regarding the weak o-minimality, it suffices to check that every $\{<_{\mathcal B},R\}$-atomic formula with one free variable (and parameters) defines a convex subset of $(\mathcal B(M),<_{\mathcal B})$; this verification is rather straightforward.  
The uniqueness of $p\in S_1(\Th(\mathcal N))$ is immediate from Lemma \ref{Lemma_Aut_L(M)_properties}(a). 
\end{proof}

Let $\mathcal N[\mathcal P]$ be the expansion of $\mathcal N$ defined as follows. For each $k$-ary relation $S$ of $\mathcal P$, we add a new $2k$-ary relation $S_{\mathcal B}$ defined on $\mathcal B(M)$ by:
\begin{center}
$S_{\mathcal B}(x_1,y_1,x_2,y_2,\dots,x_k,y_k) \ \Leftrightarrow \ S(\ell(x_1,y_1), \dots,\ell(x_k,y_k).$
\end{center}
Note that the interpretation of $\mathcal M$ in $\mathcal N^{eq}$ given in Lemma \ref{Lemma_example_MisoNeq}, when expanded by the obvious interpretations of the new relations $S$ in terms of $S_{\mathcal B}$, becomes an isomorphic copy of $\mathcal M[\mathcal P]$ in $\mathcal N[\mathcal P]^{eq}$. Therefore, theories $\Th(\mathcal M[\mathcal P])$ and $\Th(\mathcal N[\mathcal P])$ are bi-interpretable. 

\begin{Proposition}\label{Prop_Pwom_implies_MPwom}
Each of the following properties of the theory $\Th(\mathcal P)$: weak o-minimality with respect to $<_{\mathcal L}$, $\aleph_0$-categoricity, and quantifier elimination, carries over to the theory $\Th(\mathcal N[\mathcal P])$.
\end{Proposition}
\begin{proof} If $\Th(\mathcal P)$ is $\aleph_0$-categorical, or if it admits quantifier elimination, then the desired conclusion is easily obtained from Lemma \ref{Lemma_example_qfmain}(f). Assume that $\Th(\mathcal P)$ is weakly o-minimal.
Since $\mathcal M[\mathcal P]$ is saturated,
it is enough to show that $\mathcal N[\mathcal P]$ is weakly o-minimal; equivalently, that the set of realizations of every complete 1-type over a finite parameter set in $\mathcal N[\mathcal P]$ is convex with respect to $<_{\mathcal B}$. 
Let $\bar b$ be a subtree of $(\mathcal B(M)\cup\mathcal T(M),\triangleleft,\mm)$ and let $a\in \mathcal B(M)$ be nonalgebraic over $\bar b$. We will prove that the locus of $\tp(a/\bar b)$, call it $P$, is a convex subset of $\mathcal B(M)$. 
Let $c=\max_{\triangleleft}\{\mm(a,b)\mid b\in \bar b\}$ and choose $b_0\in\bar b$ such that $\mm(a,b_0)=c$. Note that $c\in \mathcal T(M)$ as $a$ is nonalgebraic over $\bar b$, and that $c\bar b$ and $ac\bar b$ are subtrees of $(\mathcal B(M)\cup\mathcal T(M),\mm)$.
Also note that by the maximality of $c$ no element of the open cone $C_c(a)$ contains an element of $\bar b$, so $a$ is a leaf of the subtree $ac\bar b$ in which $c$ is its immediate predecessor; also,
for each $b_i\in C_c\cap \bar b$ either $C_c(b_i)<C_c(a)$ or $C_c(a)<C_c(b_i)$.  

\smallskip
Case 1. \ $c\in \dcl(\bar b)$. 

In this case, changing $\bar b$ with $c\bar b$, which is also a subtree, does not alter $P$, so we can assume $c\in\bar b$.  
Let $\Phi(x,c,\bar b)$ be the conjunction of the formula $x\in C_c^{\mathcal B}$ and all formulas of the forms $x<C_c(b_i)$ and $C_c(b_i)<x$ ($b_i\in\bar b\cap C_c$) that are satisfied by $a$. Note that $\Phi(M,c,\bar b)$ is a convex subset of $\mathcal B(M)$. 
We claim that $\Phi(x,c,\bar b)$ describes the isomorphism type of $ac\bar b$, viewed as a subtree of the ordered tree $(\mathcal B(M)\cup \mathcal T(M),\triangleleft, <)$. Let $a'\in \Phi(M,c,\bar b)$. First, we see that $ac\bar b$ and $a'c\bar b$ have the same $<$-order type: Let $b_i\in\bar b$. If $b_i\in C_c$, then by the choice of $\Phi$ we have $\models a'<b_i\leftrightarrow a<b_i$; if $b_i\notin C_c$, then since $C_c^{\mathcal B}$ is $<$-convex, $C_c<b_i$ or $b_i<C_c$ holds; in either case, $\models a'<b_i\leftrightarrow a<b_i$ holds.  
It remains to verify that
$a'c\bar b$ is a subtree isomorphic to $ac\bar b$; for this, it suffices to show that the open cone $C_c(a')$ does not contain any element of $\bar b$. If some $b_i\in  \bar b$ were in $C_c(a')$, then we would have $b_i\in C_c$, so by the definition of $\Phi$ one of $a'<C_c(b_i)$ and $C_c(b_i)<a'$ would be valid and contradict $a'\in C_c(b_i)$.
Therefore, $\Phi(x,c,\bar b)$ describes the isomorphism type of the ordered subtree $ac\bar b$. Since $\ell(ac\bar b)\subset \dcl(\bar b)$ and the levels of the corresponding elements of $ac\bar b$ and $a'c\bar b$ are equal,
we conclude $\models \Phi(x,c,\bar b)\rightarrow\Delta_{ac\bar b}(xc\bar b,\ell(ac\bar b))$ and, by Lemma \ref{Lemma_example_qfmain}(b), deduce
$\Phi(x,c,\bar b)\vdash \tp(a/c\bar b)$.  Thus, $P=\Phi(M,c,\bar b)$, so $P$ is convex in $\mathcal B(M)$.   

\smallskip 

Case 2. \ $c\notin \dcl(\bar b)$. 

In this case, $c\triangleleft b_0$. Denote $q=\tp(c/\bar b)$ and $D=q(M)$, and let $d=\Lev(c)$.  
Clearly, $D$ is infinite, and we claim that it is convex in $(-\infty,b_0)_{\triangleleft}$. By Lemma \ref{Lemma_example_qfmain}(g) we have $\tp^{\mathcal P}(d/\ell(\bar b))\vdash \tp(d/\bar b)$, so since $\Th(\mathcal P)$ is weakly o-minimal, the locus of $\tp(d/\bar b)$ is $<_{\mathcal L}$-convex in $\mathcal L(M)$. Since
the function $\Lev\restriction (-\infty,b_0)_{\triangleleft}$ is a $\bar b$-definable isomorphism of the orders $(-\infty,b_0)_{\triangleleft}$ and $(-\infty,\Lev(b_0))_{<_{\mathcal L}}$, $\Lev(c)=d$ implies that this function maps $D$ onto the locus of $\tp(d/\bar b)$; since the latter is $<_{\mathcal L}$-convex in $\mathcal L(M)$, $D$ is convex in $(-\infty,b_0)_{\triangleleft}$, which proves the claim. 
From now on, assume that $b_0<a$; the other case is similar.
We will prove:
\begin{equation} 
q(\mm(x,b_0))\cup\{b_0<x\land x\in\mathcal B(M)\}\vdash \tp_x(a/\bar b).
\end{equation}
Denote $\Pi(x,\bar b):=q(\mm(x,b_0))\cup\{b_0<x\land x\in\mathcal B(M)\}$.
First, we prove that the locus $\Pi(M,\bar b)$ is convex in $(\mathcal B(M),<_{\mathcal B})$. Choose $b'\in\mathcal B(M)$ with $b_0\trianglelefteq b'$ and consider the function $f:(b',\infty)_{<_{\mathcal B}}\to (b',-\infty)_{\triangleleft}$ defined by $f(x)=\mm(x,b')$. It is easy to see that $f$ is decreasing and surjective. Since $D$ is a convex in $(b',-\infty)_{\triangleleft}$ and $f^{-1}(D)= \Pi(M,\bar b)$, we conclude that $\Pi(M,\bar b)$ is convex in $(b',\infty)_{<_{\mathcal B}}$. 
Next, we claim
\begin{equation} 
\mbox{ for all $b_i\in\bar b$ one of $\Pi(x,\bar b)\vdash b_i<x$ and $\Pi(x,\bar b)\vdash x<b_i$ holds.}
\end{equation}
Fix $b_i\in\bar b$. There are two subcases to consider, and the first is when $C_{b_i}^{\mathcal B}$ and $\Pi(M,\bar b)$ are disjoint. Then, since $C_{b_i}^{\mathcal B}$ and $\Pi(M,\bar b)$ are convex, one of $C_{b_i}^{\mathcal B}<_{\mathcal B}\Pi(M,\bar b)$ and $\Pi(M,\bar b)<_{\mathcal B} C_{b_i}^{\mathcal B}$ holds; here, the first option implies $b_i<\Pi(M,\bar b)$ and the second $\Pi(M,\bar b)< b_i$, as desired. In the second subcase, some element of $\Pi(M,\bar b)$ belongs to $C_{b_i}^{\mathcal B}$; without loss, let it be $a$. Then $b_i\triangleleft a$ and $\mm(b_i,a)=b_i$, so by the maximality of $c$ and since $c\notin \bar b$, we have $b_i\triangleleft c$ and thus $b_i\triangleleft D$. Since every element of $\Pi(M,\bar b)$ is $\triangleleft$-above some element of $D$, we conclude $b_i<\Pi(M,\bar b)$, completing the proof of (2).

Now, we can finish the proof of (1). Let $a'\in \Pi(M,\bar b)$; then $c':=\mm(a',b_0)\in D$, $b_0<a'$, and $a'\in\mathcal B(M)$; we need to prove $a\equiv a'\,(\bar b)$. Since $c$ and $c'$ realize the type $q\in S(\bar b)$, after moving $c'$ to $c$ by some $\bar b$-automorphism and replacing $a'$ by its image, we can assume $c=c'$. As in the proof of (a), we show that $ac\bar b$ and $a'c\bar b$ are isomorphic
over $\ell(ac\bar b)=\ell(a'c\bar b)$ as subtrees of the ordered tree $\mathcal B(M)\cup\mathcal T(M)$. By (2) these subtrees have the same $<$-order type, so it remains to show that the open cone $C_c(a)$ contains no element of $\bar b$: otherwise, if some $b_i\in \bar b$ were in $C_c(a')$, then we would have $C_{b_i}^{\mathcal B}\cap \Pi(M,\bar b)\neq \emptyset$ which, by the proof of (2) implies $b_i\triangleleft D$ and contradicts $c\triangleleft b_i$. This completes the proof of (1). Finally, from (1) we obtain $P=\Pi(M,\bar b)$, so $P$ is convex in $\mathcal B(M)$.

In both cases, we proved that $P$ is convex in $\mathcal B(M)$, completing the proof of the proposition.
\end{proof}

\begin{Proposition}\label{Prop_I=}
$I(\Th(\mathcal P),\aleph_0)=I(\Th(\mathcal M[\mathcal P]),\aleph_0)=I(\Th(\mathcal N[\mathcal P]),\aleph_0)$.  
\end{Proposition}
\begin{proof}
By Lemma \ref{Lemma_example_qfmain},  $\tp^{\mathcal P}(\bar c/\ell(\bar a))\vdash \tp^{\mathcal M[\mathcal P]}(\bar c/\bar a)$ \ holds for all $\bar a\in\mathcal B(M)\cup\mathcal T(M)$ and all $\bar c\in \mathcal L(M)$. This implies that the theory $\Th(\mathcal P)$ is interdefinable with the theory of the relativization of the structure $\mathcal M[\mathcal P]$ to its $\mathcal L$-sort. By the Omitting types theorem, every countable model of $\Th(\mathcal P)$ is isomorphic to the $\mathcal L$-sort of some countable model of $\Th(\mathcal M[\mathcal P])$. Since by Lemma \ref{Lemma_example_qfmain}, every countable   $\mathcal M'\models \Th(\mathcal M[\mathcal P])$ is atomic over $\mathcal L(M')$, the isomorphism type of $\mathcal M'$ is uniquely determined by the isomorphism type of its $\mathcal L$-sort, viewed as a model of $\Th(\mathcal P)$; $I(\Th(\mathcal P),\aleph_0)=I(\Th(\mathcal M[\mathcal P]),\aleph_0)$ follows. The second equality holds because the theories $\Th(\mathcal M[\mathcal P])$ and $\Th(\mathcal N[\mathcal P])$ are bi-interpretable. 
\end{proof}

\begin{Example}\label{Example_nonC4}
(A weakly o-minimal theory with 3 countable models for which condition \ref{C4} fails)\\
Let $\mathcal P_1$ be a variant of Ehrenfeucht's example of a theory with 3 countable models: $\mathcal P_1=(\mathbb Q\cup\{\infty_{\mathcal L}\},<,c_n)_{n\in\omega}$ where $(c_n)$ is an increasing sequence of constants (interpreted as unary relations), with $\lim c_n=\sqrt 2$. $\mathcal P_1$ is saturated, while $\Th(\mathcal P_1)$ is o-minimal and eliminates quantifiers. 
Consider the theory $T_1=\Th(\mathcal N[\mathcal P_1])$. By Proposition \ref{Prop_Pwom_implies_MPwom}, $T_1$ is weakly o-minimal, and by Proposition \ref{Prop_I=}, $T_1$ has 3 countable models. 
There is a unique type $p\in S_1(T_1)$. If $a,b\in \mathcal N[\mathcal P_1]$ and $a<b$, then by Proposition \ref{Prop_I=}, $\tp(a,b)$ is uniquely determined by the type $\tp^{\mathcal P_1}(\mm(a,b))$; in particular, there are infinitely many such types and $\mathcal E_p$ is infinite. 
Using elimination of quantifiers, we see that the type $q(x)=p(x)\cup\{\Lev(x,a)>_{\mathcal L} c_n\mid n\in\omega\}\cup\{a<x\}$ is complete and nonisolated. Since the locus of $q$ is a bounded subset of $\mathcal N[\mathcal P_1]$, $q$ is a forking extension of $p$. Therefore, $q\in S_1(a)$ is a nonisolated forking extension of $p$, so condition \ref{C4} fails for $T_1$. Since the partial type $p(x)\cup p(y)$ has infinitely many completions, $T_1$ is not almost $\aleph_0$-categorical. 
\end{Example}

\begin{Example}(A weakly o-minimal theory with 3 countable models for which condition \ref{C3} fails)\\
Let $\mathcal P_2$ be as $\mathcal P_1$, but with the sequence $(c_n)$ decreasing with $\lim c_n=\sqrt 2$. 
Consider the theory $T_2=\Th(\mathcal N[\mathcal P_2])$. Let $p\in S_1(T_2)$ and let $a\models p$. The forking extensions of $p$ in $S_1(a)$ are precisely those satisfying $\ell(x,a)>c_n$ for some $n\in \omega$. The set of all their realizations, $\bigcup_{n\in\omega}\{x\in \mathcal N[\mathcal P_2]\mid \ell(x,a)>c_n \}$, is not definable, so $p$ is not simple. Therefore, condition \ref{C3} fails for $T_2$. As in Example \ref{Example_nonC4}, we conclude that $T_2$ is not almost $\aleph_0$-categorical.  
\end{Example}

\begin{Proposition}\label{Proposition_trivial types in N}
Let $T=\Th(\mathcal N[\mathcal P])$, where $\mathcal P$ is a saturated weakly o-minimal expansion of $(\mathbb Q\cup\{\infty_{\mathcal L}\},< ,\infty_{\mathcal L})$ with all complete 1-types over a finite domain trivial. Let $p$ be the unique type in $S_1(T)$.
\begin{enumerate}[(a),left=0em,labelsep=.7em]
    \item $p$ is non-trivial, but all $q\in S_1(A)$ are trivial for all finite $A\neq\emptyset$.
    \item For some infinite $A$ there is a non-trivial type $q\in S_1(A)$.  
\end{enumerate}
\end{Proposition}
\begin{proof} Let $\mathbf p=(p,<)$.

(a) Choose $a_1,a_2,a_3\models p$ with $a_1\triangleleft^{\mathbf p}a_2 \triangleleft^{\mathbf p}a_3$ with $\mm(a_1,a_2)=\mm(a_1,a_3)$. Then $a_3\dep a_1a_2$ because the set defined by $\mm(a_1,a_2)=\mm(a_1,x)$ is bounded in $\mathcal B(M)$, so $(a_1,a_2,a_3)$ is not a Morley sequence in $\mathbf p_r$ and $p$ is non-trivial. 

Next, we prove that $q=\tp(a/\bar b)$ is trivial for all $a\in \mathcal B(M')$. Without loss of generality, assume that $\bar b$ is a subtree of $\mathcal B(M')\cup \mathcal T(M')$ and $a\notin\dcl(\bar b)$. Let $c=\max_{\triangleleft}\{\mm(a,b)\mid b\in \bar b\}$ and choose $b_0\in\bar b$ such that $\mm(a,b_0)=c$. We use the case analysis from the proof of Proposition \ref{Prop_Pwom_implies_MPwom}.

Case 1. \ $c\in \dcl(\bar b)$. 

In this case, we can assume $c\in\bar b$. 
In Proposition \ref{Prop_Pwom_implies_MPwom}, we proved that $\Phi(x,c,\bar b)\vdash q(x)$, where $\Phi(x,c,\bar b)\in q$ is the conjunction of $x\in C_c^{\mathcal B}$ and all formulas of the forms $x<C_c(b_i)$ and $C_c(b_i)<x$ ($b_i\in\bar b\cap C_c$) that are satisfied by $a$. To prove that $q$ is trivial, it suffices to show that any two sequences $A^k=(a_i^k\mid i\in\omega)$ of realizations of $q$ ($k=1,2$) with $C_c(a_i^k)<C_c(a_j^k)$ for all $i<j<\omega$ have the same type over $\bar b$.  
It is easy to see that $\ell(A^1c\bar b)=\ell(A^2c\bar b)$ and $\Delta_{A_1c\bar b}=\Delta_{A^2c\bar b}$, so by Lemma \ref{Lemma_example_qfmain}(f) we deduce $A^1\equiv A^2\,(\bar b)$, as desired. 
  
 Case 2. \ $c\notin \dcl(\bar b)$. 

In this case, since $c\in \dcl(a\bar b)\smallsetminus \dcl(\bar b)$, we have $q\nwor \tp(c/\bar b)$, so it suffices to prove that $\tp(c/\bar b)$ is weakly o-minimal and trivial. Let $d=\Lev(c)$. Then $c$ and $d$ are interdefinable over $\bar b$ because the function $\Lev\restriction (-\infty,b_0)_{\triangleleft}$ is a definable isomorphism of orders $(-\infty,b_0)_{\triangleleft}$ and $(-\infty,\Lev(b_0))_{<_{\mathcal L}}$. By our assumptions on $\mathcal P$, the type $\tp^{\mathcal P}(d/\ell(\bar b))$ is weakly o-minimal and trivial. Then so is $\tp(d/\bar b)$, because $\tp^{\mathcal P}(d/\ell(\bar b))\vdash \tp(d/\bar b)$ holds by Lemma \ref{Lemma_example_qfmain}(g) and nonforking extensions of a trivial type are all trivial by Theorem \ref{Theorem_nwor preserves triviality}(a). By interdefinability, $\tp(c/\bar b)$ is weakly o-minimal and trivial, as desired. 

\smallskip
(b) By the weak o-minimality of $\Th(\mathcal P) $ there is a $<_{\mathcal L}$-maximal type $r\in S_1(\emptyset)$ that satisfies $x<_{\mathcal L}\infty_{\mathcal L}$; by our assumptions, $r$ is weakly o-minimal and trivial.
Denote $\mathbf r=(r,<_{\mathcal L})$ and choose a $\triangleleft^{\mathbf r}$-increasing sequence $(d_i\mid i\in \omega)$. Fix $b\in\mathcal B(M)$. For each $n\in\omega$, let $c_n$ be the unique element of $\mathcal T(M)$ such that  $c_n\triangleleft b$ and $\Lev(c_n)=d_n$. Let $C=\{c_n\mid n\in \omega\}$. 
Then $\tp(b/C)$ is determined by $\{c_n\triangleleft x\mid n\in\omega\}\cup \{x\in\mathcal B(M')\}$. 
Since $\{x\in\mathcal B(M)\cup\mathcal T(M)\mid c_n\triangleleft x \text{ for all } n\in\omega\}$ is a dense meet-tree, arguing as in the proof of the non-triviality of $p$ in (a), one shows that $\tp(b/C)$ is non-trivial. Now, we find parameters $A\subset \mathcal B(M)$ such that $\tp(b/A)$ is non-trivial.
Choose a sequence $a_0<b_0<a_1<b_1<\dots<a_n<b_n< \dots<b$ of realizations of $\tp(b)$ such that $\mm(a_n,b_n)=c_n$ for all $n\in \omega$. Let $A=\{a_n,b_n\mid n\in\omega\}$. Then $bAC$ is a subtree and $\tp(bA/C)$ is uniquely determined by the above conditions, so $\tp(b/C)\vdash \tp(b/A)$. As in the proof of (a), we conclude that $\tp(b/A)$ is non-trivial. 
\end{proof}

\begin{Example}(A non-binary weakly o-minimal theory $T$ such that every complete 1-type over a finite domain is trivial)\ 
Let $\mathcal N_c$ be the model $\mathcal N$ with one named element. By Proposition \ref{Proposition_trivial types in N} the theory $T=\Th(\mathcal N_c)$ satisfies the listed conditions. The same conclusion holds for theories $T_1$ and $T_2$ expanded by a constant. 
\end{Example}

\printbibliography

\end{document}